\documentclass[smallextended]{svjour3}       
\usepackage{url}
\usepackage{graphicx}
\usepackage{lipsum}
\usepackage{amsfonts}
\usepackage{graphicx}
\usepackage{epstopdf}
\usepackage{xcolor}
\usepackage{lmodern} 
\usepackage{anyfontsize}
\usepackage[linesnumbered,ruled,vlined]{algorithm2e}
\usepackage{subcaption}
\usepackage{slashbox,pict2e}
\usepackage{hyphenat}
\usepackage{microtype}
\usepackage{amssymb}
\usepackage{multirow}
\usepackage{silence}
\usepackage{amsmath}
\usepackage{mathtools}
\usepackage{appendix}
\usepackage[utf8]{inputenc}
\usepackage{bm}
\usepackage{mathptmx}      
\usepackage{latexsym}

\providecommand{\openbox}{\leavevmode
  \hbox to.77778em{%
  \hfil\vrule
  \vbox to.675em{\hrule width.6em\vfil\hrule}%
  \vrule\hfil}}
\makeatletter
\DeclareRobustCommand{\qed}{%
  \ifmmode
    \eqno \def\@badmath{$$}
    \let\eqno\relax \let\leqno\relax \let\veqno\relax
    \hbox{\openbox}%
  \else
    \leavevmode\unskip\penalty9999 \hbox{}\nobreak\hfill
    \quad\hbox{\openbox}%
  \fi
}
\makeatother

\numberwithin{equation}{section}
\newcommand{\norm}[1]{{\left\|{#1}\right\|}}

\newcommand{\mD}{{\mathbf D}}
\newcommand{\mB}{{\mathbf B}}
\newcommand{\mI}{{\mathbf I}}
\newcommand{\mT}{{\mathbf T}}

\newcommand{\mK}{{\mathbf K}}
\newcommand{\bR}{{\mathbb R}}
\newcommand{\bN}{{\mathbb N}}

\SetKwInput{KwInput}{Input}                

\SetKwInput{KwInitialization}{Initialization} 
\SetKwInput{KwOutput}{Output}              
\SetKwInput{KwCondition}{Condition}

\SetKwRepeat{Do}{do}{while}%

\begin{document}

\begin{sloppypar}

\title{Inexact Fixed-Point Proximity Algorithm for the $\ell_0$ Sparse Regularization Problem

\thanks{Y. Xu is supported in part by the US National Science Foundation under grants DMS-1912958 and DMS-2208386, and by the US National Institutes of Health under grant R21CA263876.}

}

\titlerunning{Inexact FPPA for the $\ell_0$ Sparse Regularization Problem}        

\author{Ronglong Fang  \and Yuesheng Xu \and Mingsong Yan}

\authorrunning{R. Fang, Y. Xu, and M. Yan} 

\institute{R. Fang \at Department of Mathematics and Statistics,  Old Dominion University, Norfolk, VA 23529, USA. \email{rfang002@odu.edu}           
\and
           Y. Xu \at
              Department of Mathematics and Statistics,  Old Dominion University, Norfolk, VA 23529, USA. \email{y1xu@odu.edu}. Correspondence author
\and
           M. Yan \at
              Department of Mathematics and Statistics,  Old Dominion University, Norfolk, VA 23529, USA. \email{myan007@odu.edu}
}

\date{Received: date / Accepted: date}

\maketitle
\begin{abstract}
We study {\it inexact} fixed-point proximity algorithms for solving a class of sparse regularization problems involving the $\ell_0$ norm. Specifically, the $\ell_0$ model has an objective function that is the sum of a convex fidelity term and a Moreau envelope of the $\ell_0$ norm regularization term. Such an $\ell_0$ model is non-convex. Existing exact algorithms for solving the problems require the availability of closed-form formulas for the proximity operator of convex functions involved in the objective function. When such formulas are not available, numerical computation of the proximity operator becomes inevitable. This leads to inexact iteration algorithms. We investigate in this paper how the numerical error for every step of the iteration should be controlled to ensure global convergence of the resulting inexact algorithms. We establish a theoretical result that guarantees the sequence generated by the proposed inexact algorithm converges to a local minimizer of the optimization problem. We implement the proposed algorithms for three applications of practical importance in machine learning and image science, which include regression, classification, and image deblurring. The numerical results demonstrate the convergence of the proposed algorithm and confirm that local minimizers of the $\ell_0$ models found by the proposed inexact algorithm outperform global minimizers of the corresponding $\ell_1$ models, in terms of  approximation accuracy and sparsity of the solutions.

\end{abstract}
 

\section{Introduction}

Sparse learning plays a pivotal role in today's era of big data due to its capacity for significantly decreasing the computational burden through sparse solution representation. Natural images or data streams are often inherently sparse in certain bases or dictionaries, under which they can be approximately expanded by only few significant elements carrying the most relevant information \cite{4016283,mallat1999wavelet,suter1997multirate,Xu2023Sparse}. Sparsity-promoting regularization is a common and powerful tool to capture the most important features hidden in the enormous data. Such a technique has been successfully applied in many areas, for example, compressed sensing \cite{boche2015survey,donoho2006compressed}, image processing \cite{elad2006image,hansen2006deblurring,li2015multi,micchelli2011proximity,shen2014approximate}, and machine learning \cite{li2018sparse,song2013reproducing}. 

The most intuitive choice of the regularization term for promoting sparsity is the $\ell_0$-norm which counts the number of the nonzero components of a vector. The geometric viewpoint of the $\ell_0$-norm has been explored in \cite{2021Sparse}, demonstrating that the $\ell_0$ regularization term results in lifting the graph of the target function according to the level of sparsity. Nevertheless, it is challenging to develop an efficient algorithm to solve the $\ell_0$-norm based regularization problem due to the discontinuity and non-convexity of the $\ell_0$ norm.
This challenge has been discussed in \cite{1542412,garey1979computers,natarajan1995sparse} that finding a global minimizer of the $\ell_0$-norm based regularization problem is an NP-hard problem. A widely adopted approach to address this issue is replacing the $\ell_0$-norm by the $\ell_1$-norm which transfers a non-convex problem to a convex one. There are plenty of feasible algorithms from convex optimization for solving the $\ell_1$-norm based regularization problem \cite{chambolle2011first,li2015multi,micchelli2011proximity}. However, the $\ell_1$ norm whose proximity operator is the well-known soft thresholding operator, could bring up unnecessary bias on the components of the solution even if some of them are dominant \cite{fan2001variable,zou2006adaptive} and such a bias may jeopardize the performance of the solution, especially the accuracy on the data. The mentioned disadvantage of the $\ell_1$-norm in fact can be avoided by the $\ell_0$-norm, which takes the hard thresholding operator as its proximity operator and potentially elevates the accuracy of the solution applying to the data. Consequently, there arises a need to devise algorithms that directly address the challenge of solving the $\ell_0$-norm based regularization problem.

Recently, there was notable advancement in the development of convergent algorithms designed for the $\ell_0$-norm based regularization problem.  It was proposed in \cite{shen2016wavelet} a fixed-point proximity algorithm in the context of image inpainting, 
with guaranteed convergence to a local minimizer of the objective function involving the $\ell_0$-norm.
Its proof reveals that the support of the sequence generated by the algorithm remains unchanged after a finite number of iterations, and thus on the fixed support the non-convex optimization problem is reduced to a convex one. This idea has also been applied to the problems of the seismic wavefield
modeling \cite{wu2022inverting} and medical image reconstruction \cite{zheng2019sparsity}. Moreover, the numerical examples provided in \cite{shen2016wavelet,wu2022inverting,zheng2019sparsity}  affirm that solutions obtained from the $\ell_0$-norm based regularization model exhibit superior performance compared to those from the $\ell_1$-norm based regularization model. The algorithm studied in \cite{shen2016wavelet,wu2022inverting} usually requires the evaluation of the proximity operator of the composition of a convex function and a matrix for each iteration. The precise value of that proximity operator is difficult to obtain except for the case of the matrix having some special property, such as the matrix being reduced to be an identity matrix or having circulant blocks structure \cite{hansen2006deblurring}. This requirement significantly constrains the applicability of the algorithm in some areas, for example, regression and classification problems in machine learning and deblurring problems in image processing.  The fixed-point iterative algorithms proposed in \cite{li2015multi,micchelli2011proximity} facilitate computing the proximity operator of a convex function composite with a general matrix, in which errors are unavoidable with a finite number of iterations. This motivates us to consider an inexact version of the algorithm proposed in \cite{shen2016wavelet,wu2022inverting}, namely the inexact fixed-point proximity algorithm. 

In the field of optimization, inexact methods have attracted considerable attention.
Inexact methods for solving convex optimization problems related to monotone operators and the proximal point algorithm were developed in \cite{rockafellar1976monotone}. Moreover, the inexact forward-backward algorithm, the inexact variant of the Krasnosel’skiĭ-Mann iteration, and the inexact implicit fixed-point proximity algorithm were proposed in \cite{solodov1999hybrid,villa2013accelerated}, \cite{combettes2004solving,liang2016convergence} and  \cite{ren2023}, respectively. In \cite{attouch2013convergence}, inexact descent methods were developed for non-convex optimization problems with the objective function satisfying the Kurdyka–Łojasiewicz (KL) property \cite{kurdyka1998gradients,lojasiewicz1963propriete} and their convergence to a critical point of the objective function was established. 

The goal of this paper is to propose an inexact fixed-point proximity algorithm to solve the $\ell_0$-norm based regularization problem and prove its convergence to a local minimizer of the objective function without assuming the KL property. Specifically, we show that the proposed algorithm first pursues the support of a sparse solution of the non-convex model involving the $\ell_0$-norm and then searches for the solution on the resulting support set, and based on this understanding, we establish the convergence result. 
We apply the inexact fixed-point proximity algorithm to three application problems including regression and classification and image deblurring. The convergence and effectiveness of the proposed algorithm are confirmed by numerical experiments.

This paper is organized in six sections. In Section \ref{section: model}, we describe the $\ell_0$-norm based regularization model and characterize a solution of the model as the fixed-point problem. In Section \ref{section: inexact iteration}, we propose the inexact fixed-point proximity algorithm to solve the $\ell_0$-norm based regularization model and demonstrate its feasibility.  Section \ref{section: convergence} is devoted to convergence analysis of the proposed inexact algorithm. In Section \ref{section: Experiments} we apply the proposed algorithm to solve regression, classification and image deblurring problems  and demonstrate the performance of the proposed algorithm. Our conclusions are drawn in Section \ref{section: Conclusion}.

\section{Sparse Regularization Model}\label{section: model}

In this section, we describe the sparse regularization model and present a necessary condition for a global minimizer of the resulting minimization problem.

Many application problems may be modeled as sparse regularization problems. The sparsity of a vector is naturally measured by the number of its nonzero components. Specifically, for $\mathbf{x}\in \mathbb{R}^n$,
$\norm{\mathbf{x}}_0$ counts the number of nonzero components of $\mathbf{x}$ and we call $\norm{ \mathbf{x}}_0$ the $\ell_0$-norm of $\mathbf{x}$, even though it is not really a norm. A sparse regularization problem may be described as
\begin{equation}\label{model: 0 l0}
    \text{argmin} \left\{ \psi(\mB \mathbf{v}) + \lambda \norm{\mD \mathbf{v}}_0  : \mathbf{v}\in \bR^m \right\},
\end{equation}
where the function $\psi:\mathbb{R}^p\to\mathbb{R}$ is proper, convex,  continuously differentiable, and bounded below, $\mB \in \bR^{p \times m}$, $\mD \in \bR^{n \times m}$ are matrices, and $\lambda$ is a positive regularization parameter. 
The matrix $\mD$ appearing in \eqref{model: 0 l0} is often a redundant system chosen as a mathematical transformation, such as a discrete cosine transform \cite{strang1999discrete}, a wavelet transforms \cite{daubechies1992ten,lian2011filters,mallat1999wavelet}, or a framelet
transform \cite{chan2004tight,ron1997affine}, depending on specific applications. In this paper, we are mainly interested in the case of $\mD$ being a discrete tight framelet system, that is, $\mD^{\top}\mD = \mI$ with $\mD^{\top}$ being the transpose of $\mD$ and $\mI$ the identity matrix.


The discontinuity and non-convexity of the $\ell_0$-norm are major barriers to developing efficient algorithms for solving the optimization problem \eqref{model: 0 l0}. 
Practically the $\ell_0$-norm in the model \eqref{model: 0 l0} is often replaced by the $\ell_1$-norm \cite{chambolle2011first,li2015multi,micchelli2011proximity} which is both continuous and convex. The resulting model is 
\begin{equation}\label{model: l1}
    \mathrm{argmin}\left\{\psi(\mathbf{B} \mathbf{v})+\lambda\|\mathbf{D} \mathbf{v}\|_1: \mathbf{v}\in\mathbb{R}^m\right\}.
\end{equation}
{Model \eqref{model: l1} is a special case of the model in \cite{chambolle2011first,li2015multi},  which may be solved by the first-order primal-dual algorithm \cite{chambolle2011first} and multi-step fixed-point proximity algorithm \cite{li2015multi}. Both of these algorithms are required to compute the proximity operator of $\lambda \norm{\mathbf{D}\mathbf{v}}_1$, which often has no closed-form. Alternatively, one may introduce an auxiliary variable $\mathbf{u}$ to free $\mathbf{D}\mathbf{v}$ from the non-differentiable norm $\norm{\cdot}_1$ and add the difference between $\mathbf{u}$ and $\mathbf{D}\mathbf{v}$ as a penalized term. This yields the model
\begin{equation}\label{model: l1-1}
    \text{argmin} \left\{   \psi(\mB \mathbf{v})  + \frac{\lambda}{2\gamma}\norm{\mathbf{u} - \mD \mathbf{v}}_2^2+ \lambda\norm{\mathbf{u}}_1: (\mathbf{u}, \mathbf{v}) \in \bR^n \times\bR^m \right\}.
\end{equation}  
This falls into the quadratic penalty approach which can be tracked back to \cite{Courant1943VariationalMF} and has been widely used in image reconstruction \cite{tan2014smoothing,wang2008new,zeng2016convergent}.
However, the $\ell_1$-norm may promote biases \cite{fan2001variable}. For this reason, we prefer using the $\ell_0$-norm by overcoming difficulties brought by it. Specifically, we adopt the quadratic penalty approach to free $\mathbf{D}\mathbf{v}$ from $\norm{\cdot}_0$.}
This leads to the following two variables $\ell_0$-norm based regularization problem
\begin{equation}\label{model: l0-0}
    \text{argmin} \left\{   \psi(\mB \mathbf{v})  + \frac{\lambda}{2\gamma}\norm{\mathbf{u} - \mD \mathbf{v}}_2^2+ \lambda\norm{\mathbf{u}}_0: (\mathbf{u}, \mathbf{v}) \in \bR^n \times\bR^m \right\}
\end{equation}
where $\gamma>0$ is the envelope parameter. We remark that the sum of the last two terms in \eqref{model: l0-0} approaches to $\norm{\mD \mathbf{v}}_0$ as $\gamma$ goes to zero. 
For notation simplicity, we let
\begin{equation}\label{target function in model: l0}
F(\mathbf{u}, \mathbf{v}) :=  \psi\left(\mB \mathbf{v}\right)+ \frac{\lambda}{2\gamma}\norm{\mathbf{u} - \mD \mathbf{v}}_2^2 + \lambda\norm{\mathbf{u}}_0,\quad(\mathbf{u}, \mathbf{v}) \in \bR^{n} \times \bR^m.
\end{equation}
In this notation, problem \eqref{model: l0-0} may be restated as
\begin{equation}\label{model: l0}
    \text{argmin} \left\{  F(\mathbf{u},\mathbf{v}):  (\mathbf{u}, \mathbf{v}) \in \bR^n \times\bR^m \right\}.
\end{equation}

We next present a necessary condition for a solution of \eqref{model: l0}. To this end, we recall the definition of the proximity operator \cite{bauschke2011convex}. For a proper, lower semi-continuous function $f:\bR^m\to\bR \cup \{\infty\}$ and a positive number $t$, the proximity operator $\text{prox}_{tf} : \bR^m \to \bR^m $ at $\mathbf{y}\in\bR^{m}$ is defined as
\begin{equation}\label{def: general proximity operator}
    \text{prox}_{tf}(\mathbf{y}) := \text{argmin} \left\{\frac{1}{2t}\norm{\mathbf{x} - \mathbf{y}}_2^2 +  f(\mathbf{x}) :  \mathbf{x}\in \bR^m\right\}.
\end{equation}
We remark that the proximity operator $\text{prox}_{tf}$ defined by \eqref{def: general proximity operator} is set-valued and if $f$ is further assumed to be convex, the corresponding proximity operator will be single-valued. In particular, if $f$ is continuously differentiable and convex, then 
$$
\text{prox}_{tf}(\mathbf{y}) = \left(\mI + t\nabla f\right)^{-1}(\mathbf{y}).
$$

The remaining part of this section is to establish a necessary condition for a solution of \eqref{model: l0} in the form of a fixed-point of the proximity operators.
To prepare for this, we present three technical lemmas. 

\begin{lemma}\label{lemma: a,b}
If $\mathbf{a}, \mathbf{b}\in \bR^m$ and $q\in [0,1]$, then
\begin{equation}\label{ineq in Rn}
        q\|\mathbf{a}\|_2^2-q\|\mathbf{b}\|_2^2\leq \left\|\mathbf{a}-(1-q)\mathbf{b}\right\|_2^2-\left\|q\mathbf{b}\right\|_2^2.
\end{equation}
\end{lemma}
\begin{proof}
A direct computation confirms that
\begin{align*}
\left\|\mathbf{a}-(1-q)\mathbf{b}\right\|_2^2-\left\|q \mathbf{b}\right\|_2^2-q\|\mathbf{a}\|_2^2+q\|\mathbf{b}\|_2^2=&(1-q)[\|\mathbf{a}\|_2^2-2\left<\mathbf{a}, \mathbf{b}\right>+\|\mathbf{b}\|_2^2]\\
=&(1-q)\|\mathbf{a}- \mathbf{b}\|_2^2\geq 0.
\end{align*}
This establishes the desired inequality \eqref{ineq in Rn}. 
\qed
\end{proof}
\begin{lemma}\label{lemma: fixed-point iteration lemma}
Suppose that $f:\mathbb{R}^n\to\mathbb{R} \cup \left\{\infty\right\}$ is a proper, lower semi-continuous function. If $p>0$ and $\mathbf{x}^*, \mathbf{y}^*\in\mathbb{R}^n$ such that 
\begin{equation}\label{assume ineq fx* - fx}
    p\|\mathbf{x}^*-\mathbf{y}^*\|_2^2+f(\mathbf{x}^*)\leq p\|\mathbf{x}-\mathbf{y}^*\|_2^2+f(\mathbf{x}), \ \  \mbox{for all}\ \ \mathbf{x}\in\mathbb{R}^n, 
\end{equation}
then for all $q\in[0,1]$,
\begin{equation*}
    \mathbf{x}^* \in \mathrm{prox}_{\frac{q}{2p}f}\left(q \mathbf{y}^*+(1-q)\mathbf{x}^*\right).
\end{equation*}
\end{lemma}

\begin{proof}
From the definition of the proximity operator, it suffices to prove that
\begin{equation}\label{equi proximity q/2pf}
\frac{1}{2}\left\|q(\mathbf{x}^*-\mathbf{y}^*)\right\|_2^2+\frac{q}{2p}f(\mathbf{x}^*)\leq\frac{1}{2}\left\|\mathbf{x}-\left(q\mathbf{y}^*+(1-q)\mathbf{x}^*\right)\right\|_2^2+\frac{q}{2p}f(\mathbf{x}), \ \  \mbox{for all}\ \ \mathbf{x}\in\mathbb{R}^n.
\end{equation}
By employing inequality \eqref{ineq in Rn} of Lemma \ref{lemma: a,b} with $\mathbf{a}:=\mathbf{x}- \mathbf{y}^*$ and $\mathbf{b}:= \mathbf{x}^*-\mathbf{y}^*$, we obtain that 
\begin{equation}\label{transformed equi proximity q/2pf}
    q\|\mathbf{x}- \mathbf{y}^*\|_2^2-q\|\mathbf{x}^*-\mathbf{y}^*\|_2^2\leq \left\|\mathbf{x}-\left(q\mathbf{y}^*+(1-q)\mathbf{x}^*\right)\right\|_2^2-\left\|q(\mathbf{x}^*-\mathbf{y}^*)\right\|_2^2, \ \  \mbox{for all}\ \ \mathbf{x}\in\mathbb{R}^n.
\end{equation}
Combining \eqref{transformed equi proximity q/2pf} with the assumption \eqref{assume ineq fx* - fx} leads to estimate \eqref{equi proximity q/2pf}. 
\qed
\end{proof}

\begin{lemma}\label{lemma: relation u-Dv and v-DTu}
If $\mD \in \bR^{n \times m}$ such that $\mD^{\top}\mD = \mI$, then
\begin{equation}\label{eq: relation u-Dv and v-DTu}
\norm{\mathbf{u} - \mD \mathbf{v}}_2^2 = \norm{\mathbf{v} - \mD^{\top}\mathbf{u}}_2^2 + \mathbf{u}^{\top}(\mI - \mD \mD^{\top})\mathbf{u}, \ \   \mbox{for all}\ \ \mathbf{v} \in \bR^m,\ \mathbf{u}\in \bR^n.
\end{equation}
\end{lemma}
\begin{proof}
We prove \eqref{eq: relation u-Dv and v-DTu} by direct computation. Since by assumption $\mD^\top\mD=\mI$,  we may expand  $\norm{\mathbf{u} - \mD \mathbf{v}}_2^2$ as 
\begin{equation}\label{expand u - D v}
    \norm{\mathbf{u} - \mD \mathbf{v}}_2^2= \mathbf{u}^{\top} \mathbf{u}-2\mathbf{u}^{\top}\mD \mathbf{v}+\mathbf{v}^{\top} \mathbf{v}
\end{equation}
and  $\norm{\mathbf{v} - \mD^{\top} \mathbf{u}}_2^2$ as
\begin{equation}\label{expand v - Dtop u}
    \norm{\mathbf{v} - \mD^{\top}\mathbf{u}}_2^2= \mathbf{v}^{\top} \mathbf{ v}-2 \mathbf{u}^{\top}\mD \mathbf{v}+\mathbf{u}^{\top}\mD\mD^\top \mathbf{u}.
\end{equation}
Subtracting \eqref{expand v - Dtop u} from \eqref{expand u - D v} yields that 
\begin{equation}
    \norm{\mathbf{u} - \mD \mathbf{v}}_2^2-\norm{\mathbf{v} - \mD^{\top}\mathbf{u}}_2^2=\mathbf{u}^{\top} \mathbf{u}-\mathbf{u}^\mathbf{\top}\mD\mD^{\top} \mathbf{u}= \mathbf{u}^{\top}(\mI - \mD \mD^{\top})\mathbf{u},
\end{equation}
which leads to equation \eqref{eq: relation u-Dv and v-DTu}. \qed
\end{proof}



The next proposition presents a necessary condition for a solution of \eqref{model: l0} as a fixed-point of the proximity operator.  For notation simplicity, we define the operator $\mathcal{S}: \bR^m \to \bR^m$ as
$$
\mathcal{S}(\mathbf{v}):=  \left(\mI + \frac{\gamma}{\lambda} \nabla \left(\psi \circ \mB \right)\right)^{-1}(\mathbf{v}), \ \ \text{for all} \ \ \mathbf{v}\in \bR^m.
$$

\begin{proposition}\label{thm: fixed-point outer loop} 
If the pair $(\mathbf{u}^*, \mathbf{v}^*) \in \bR^n \times \bR^m$ is a global minimizer of problem \eqref{model: l0}, then for any {$\alpha \in (0, 1]$}, 
{
\begin{align}
    \mathbf{u}^* &\in \mathrm{prox}_{\alpha\gamma \norm{\cdot}_0}\left((1-\alpha)\mathbf{u}^*+\alpha \mD \mathbf{v}^*\right), \label{eq: fix-point 1}\\
    \mathbf{v}^* &= \mathcal{S}\left(\mD^{\top} \mathbf{u}^* \right). \label{eq: fix-point 2}
\end{align}
}
\end{proposition}

\begin{proof}
Since $(\mathbf{u}^*, \mathbf{v}^*)$ is a global minimizer of \eqref{model: l0}, we obtain that 
\begin{equation}\label{uv is global minimizer in the proof}
F(\mathbf{u}^*, \mathbf{v}
^*) \leq F(\mathbf{u}, \mathbf{v}), \ \  \mbox{for all}\ \ (\mathbf{u}, \mathbf{v})\in\bR^n\times \bR^m.
\end{equation}
By noting the definition of $F$ and choosing $\mathbf{v}:= \mathbf{v}^*$ in \eqref{uv is global minimizer in the proof}, we obtain that 
\begin{equation}\label{leq u*}
\frac{\lambda}{2\gamma}\norm{\mathbf{u}^* - \mD \mathbf{v}^*}_2^2+\lambda \norm{\mathbf{u}^*}_0\leq \frac{\lambda}{2\gamma}\norm{\mathbf{u} - \mD \mathbf{v}^*}_2^2+\lambda \norm{\mathbf{u}}_0, \ \ \text{for all} \ \mathbf{u} \in \bR^n.
\end{equation}
Applying Lemma \ref{lemma: fixed-point iteration lemma} with $p:=\frac{\lambda}{2\gamma}$, $q:=\alpha$, $\mathbf{x}^*:= \mathbf{u}^*$,  $\mathbf{y}^*:= \mD \mathbf{v}^*$ and $f:=\lambda\|\cdot\|_0$ and employing the definition of the proximity operator, we get \eqref{eq: fix-point 1}. We next take $\mathbf{u}:= \mathbf{u}^*$ in \eqref{uv is global minimizer in the proof} and find that
\begin{equation}\label{leq v*}
\psi(\mB \mathbf{v}^*)+\frac{\lambda}{2\gamma}\norm{\mathbf{u}^* - \mD \mathbf{v}^*}_2^2 \leq \psi(\mB \mathbf{v})+\frac{\lambda}{2\gamma}\norm{\mathbf{u}^* - \mD \mathbf{v}}_2^2, \ \ \text{for all} \ \mathbf{v} \in \bR^m.
\end{equation}
It follows from Lemma \ref{lemma: relation u-Dv and v-DTu} that 
\begin{equation}\label{replace u-Dv*}
    \norm{\mathbf{u}^* - \mD \mathbf{v}^*}_2^2 = \norm{\mathbf{v}^* - \mD^{\top}\mathbf{u}^*}_2^2 + (\mathbf{u}^*)^{\top}(\mI - \mD \mD^{\top})\mathbf{u}^*
\end{equation}
and 
\begin{equation}\label{replace u-Dv}
    \norm{\mathbf{u}^* - \mD \mathbf{v}}_2^2 = \norm{\mathbf{v} - \mD^{\top}\mathbf{u}^*}_2^2 + (\mathbf{u}^*)^{\top}(\mI - \mD \mD^{\top})\mathbf{u}^*.
\end{equation}
Substituting \eqref{replace u-Dv*} and \eqref{replace u-Dv} into \eqref{leq v*} yields the inequality
\begin{equation}\label{leq v*-DTu*}
\psi(\mB \mathbf{v}^*)+\frac{\lambda}{2\gamma}\norm{\mathbf{v}^* - \mD^{\top} \mathbf{u}^*}_2^2  \leq \psi(\mB \mathbf{v})+\frac{\lambda}{2\gamma}\norm{\mathbf{v} - \mD^{\top} \mathbf{u}^*}_2^2, \ \ \text{for all} \ \mathbf{v} \in \bR^m.
\end{equation}
Inequality \eqref{leq v*-DTu*} is equivalent to the equation $\mathbf{v}^* = \mathrm{prox}_{\frac{\gamma}{\lambda}\psi \circ \mB}\left(\mD^{\top} \mathbf{u}^*\right)$, which is reduced to  $\mathbf{v}^* = \mathcal{S}\left(\mD^{\top}\mathbf{u}^*\right)$ due to the  continuous differentiability of $\frac{\gamma}{\lambda}\psi\circ \mB$.
\qed
\end{proof}

The necessary condition for a global minimizer of optimization problem \eqref{model: l0} presented in Proposition \ref{thm: fixed-point outer loop} suggests an inexact fixed-point proximity algorithm for solving  \eqref{model: l0}, which will be discussed in the next section.

\section{Inexact Fixed-Point Proximity Algorithm}
\label{section: inexact iteration}

In this section, we first describe an {\it exact} fixed-point proximity algorithm for the model \eqref{model: l0}, and discuss a computational issue of the exact algorithm. We then propose an {\it inexact} fixed-point proximity algorithm to address the computational issue. 

The fixed-point equations \eqref{eq: fix-point 1}-\eqref{eq: fix-point 2} introduced in Proposition \ref{thm: fixed-point outer loop} suggest an exact fixed-point proximity algorithm for model \eqref{model: l0}. That is,
\begin{align}
    &\mathbf{u}^{k+1} \in \text{prox}_{\alpha \gamma \norm{\cdot}_0}\left( \left(1 - \alpha\right)\mathbf{u}^k+\alpha \mD \mathbf{v}^k \right), \label{iter: prox u ell0}\\
    & \mathbf{v}^{k+1} = \mathcal{S}\left(\mD^{\top} \mathbf{u}^{k+1} \right), \label{iter: prox v ell0}
\end{align}
{where $\alpha \in (0, 1]$.} 
{
The exact fixed-point proximity algorithm \eqref{iter: prox u ell0}-\eqref{iter: prox v ell0} has been introduced in \cite{zeng2016convergent} for $\alpha \in (0, 1)$. 

Comments for the case  $\alpha = 1$ are in order.
In this case, iteration \eqref{iter: prox u ell0} reduces to
\begin{align}
    &\mathbf{u}^{k+1} \in \text{prox}_{ \gamma \norm{\cdot}_0}\left( \mD \mathbf{v}^k \right). \label{iter: prox u ell0 alpha = 1}
\end{align}
The sequence $\left\{ \left(\mathbf{u}^k, \mathbf{v}^k\right)\right\}_{k=1}^{\infty}$ generated from \eqref{iter: prox u ell0 alpha = 1} and \eqref{iter: prox v ell0} is well-defined. Moreover, it can be verified from inequalities (2.25) and (2.28) of \cite{zeng2016convergent} that
the sequence of objective function values $\left\{F(\mathbf{u}^k, \mathbf{v}^k)\right\}_{k=1}^{\infty}$ is non-increasing, indicating convergence of the  sequence $\left\{F(\mathbf{u}^k, \mathbf{v}^k)\right\}_{k=1}^{\infty}$, when the function $\psi$ has a lower bound. Furthermore,
for the special case where $\psi(\mathbf{B}\mathbf{v}) = \frac{1}{2}\norm{\mathbf{B}\mathbf{v} - \mathbf{r}}_2^2$ for a given $\mathbf{r} \in \bR^p$, with $\mathbf{B}$ satisfying $\mathbf{B}\mathbf{B}^{\top} = \mathbf{I}$, and $\mathbf{D} = \mathbf{I}$, the sequence $\left\{\left(\mathbf{u}^k, \mathbf{v}^k\right)\right\}_{k=1}^{\infty}$ generated by algorithm \eqref{iter: prox u ell0 alpha = 1} and \eqref{iter: prox v ell0}  was proved in \cite{wu2022inverting} to be convergent under the condition $0<\frac{\gamma}{\lambda}<\frac{\sqrt{5}-1}{2}$.  However, it remains uncertain whether the sequence $\left\{\left(\mathbf{u}^k, \mathbf{v}^k\right)\right\}_{k=1}^{\infty}$ generated by algorithm \eqref{iter: prox u ell0 alpha = 1} and \eqref{iter: prox v ell0} in general is convergent. 
}

{ One may apply the alternating direction method of multipliers (ADMM) to solve optimization problem \eqref{model: 0 l0}.
ADMM was originally proposed for solving convex optimization \cite{boyd2011distributed,gabay1976dual,glowinski1975approximation}, and was later extended to solving nonconvex nonsmooth optimization \cite{wang2019global}. Convergence of ADMM was established in \cite{wang2019global} by requiring that the augmented Lagrangian function be of Kurdyka-Łojasiewicz and that the penalty parameter $\lambda/\gamma$ in the augmented Lagrangian function defined by \eqref{augmented Lagrangian function} below be sufficiently large.}



{ To describe ADMM for  \eqref{model: 0 l0}, we rewrite problem \eqref{model: 0 l0} in an equivalent form
\begin{equation}\label{ADMM form}
\begin{array}{ll}
\text{minimize} &\psi(\mathbf{B}\mathbf{v})+\lambda\norm{\mathbf{u}}_0 \\
\text{subject to} & \mathbf{D}\mathbf{v} - \mathbf{u} = 0
\end{array}
\end{equation}
where $(\mathbf{u}, \mathbf{v}) \in \bR^{n} \times \bR^{m}$. We then define the augmented Lagrangian function for \eqref{ADMM form} as
\begin{equation}\label{augmented Lagrangian function}
    L(\mathbf{u}, \mathbf{v}, \mathbf{w}) := \psi(\mathbf{B}\mathbf{v})+\lambda\norm{\mathbf{u}}_0 + \mathbf{w}^{\top}\left(\mathbf{D}\mathbf{v} - \mathbf{u}\right) + \frac{\lambda}{2\gamma}\norm{\mathbf{D}\mathbf{v} - \mathbf{u}}_2^2
\end{equation}
where $\mathbf{w} \in \bR^n$ is a dual variable.}
{ ADMM minimizes the augmented Lagrangian function \eqref{augmented Lagrangian function} with respect to variables $\mathbf{u}$ and $\mathbf{v}$ alternatively and then updates the dual variable $\mathbf{w}$, 
leading to the following iteration scheme
\begin{align}
    &\mathbf{u}^{k+1} = \text{argmin}_{\mathbf{u}} \quad L(\mathbf{u}, \mathbf{v}^k, \mathbf{w}^k), \label{admm u}\\
    &\mathbf{v}^{k+1} = \text{argmin}_{\mathbf{v}} \quad L(\mathbf{u}^{k+1}, \mathbf{v}, \mathbf{w}^k), \label{admm v}\\
    &\mathbf{w}^{k+1} = \mathbf{w}^k + \frac{\lambda}{\gamma}\left(\mathbf{D}\mathbf{v}^{k+1} - \mathbf{u}^{k+1}\right). \label{admm w}
\end{align}
Note that the gradient of the augmented Lagrangian function $L$ with respect to the dual variable $\mathbf{w}$ is given by $\mathbf{D}\mathbf{v} - \mathbf{u}$. ADMM employs the gradient ascent method to update the dual variable $\mathbf{w}$, utilizing a fixed step-size $\frac{\lambda}{\gamma}$. Consequently, for each iteration, the function value $L(\mathbf{u}, \mathbf{v}, \mathbf{w})$ decreases when updating the primal variables $\mathbf{u}$ and $\mathbf{v}$ and increases when updating the dual variable $\mathbf{w}$. 
In other words, the gradient ascent step that updates the dual variable in ADMM may lead to oscillations in values of the augmented Lagrange function. As such, ADMM may fail to converge for the augmented Lagrangian function $L$ involving the $\ell_0$ norm, 
since due to the non-convexity of the $\ell_0$ norm, ensuring the monotonic decreasing behavior of the objective function at every iteration step is crucial for its convergence. 
Unlike ADMM, the exact fixed-point proximity algorithm \eqref{iter: prox u ell0}-\eqref{iter: prox v ell0} consists only the $\mathbf{u}$-minimization step \eqref{iter: prox u ell0} and $\mathbf{v}$-minimization step \eqref{iter: prox v ell0}. It was proved in Theorem 2.1 of \cite{zeng2016convergent} that the objective function $F$ defined by \eqref{target function in model: l0} is monotonically decreasing at every iteration step of \eqref{iter: prox u ell0}-\eqref{iter: prox v ell0}, leading to its convergence.  For this reason, we prefer minimizing the objective function $F$ defined by \eqref{target function in model: l0} rather than the augmented Lagrangian function $L$. Our goal is to develop a convergence guaranteed inexact fixed-point algorithm to solve the optimization problem \eqref{model: l0} without requiring the objective function to be a Kurdyka-Łojasiewicz function and the penalty parameter $\frac{\lambda}{\gamma}$ to be sufficiently large, as required in \cite{wang2019global}.

}

We now consider the implementation of the exact fixed-point proximity algorithm \eqref{iter: prox u ell0}-\eqref{iter: prox v ell0}.
It is well-known that for $\mathbf{u} \in \bR^n$, the proximity operator of the $\ell_0$-norm is given by
\begin{equation}\label{eq:prox_ell0_norm 1}
    {\rm prox}_{t \norm{\cdot}_0}(\mathbf{u}) := [{\rm prox}_{t |\cdot|_0}(u_1), {\rm prox}_{t |\cdot|_0}(u_2), \dots, {\rm prox}_{t |\cdot|_0}(u_n)]^\top
\end{equation}
with $t>0$, where 
\begin{equation}\label{eq:prox_ell0_abs 1}
\text{prox}_{t |\cdot|_0}(u_i) := \begin{cases}
\{u_i\}, & \ \text{if} \ |u_i|>\sqrt{2t},\\
\{u_i, 0\}, & \ \text{if} \ |u_i| = \sqrt{2t},\\
\{0\}, & \ \text{otherwise}.
\end{cases}    
\end{equation}
Note that $\text{prox}_{t|\cdot|_0}$ is the hard thresholding operator. The closed-form formula \eqref{eq:prox_ell0_norm 1} facilitates the implementation of \eqref{iter: prox u ell0}.  However, the implementation of \eqref{iter: prox v ell0} encounters a computational issue of 
evaluating the operator $\mathcal{S}$. We assume that the operator $\left(\mI + \lambda/\gamma\nabla  \psi \right)^{-1}$ has a closed form, which is true in many problems of practical importance.  In the special case when the matrix $\mB$ in \eqref{iter: prox v ell0} satisfies $\mB \mB^{\top} = \mI$, the operator $\mathcal{S}$ also has a closed form. It is known (for example, Theorem 6.15 of \cite{beck2017first}) that if  $\mB \mB^{\top} = \mI$, then for any $\mathbf{x} \in \bR^m$, 
$$
\mathcal{S}(\mathbf{x}) = (\mI - \mB^{\top} \mB)\mathbf{x} + \mB^{\top}\left(\mI + \gamma/\lambda \nabla \psi \right)^{-1}(\mB \mathbf{x}).
$$
In this case, the update $\mathbf{v}^{k+1}$ in \eqref{iter: prox v ell0} can be computed exactly by using the above formula and a closed-form formula of $\left(\mI + \gamma/\lambda \nabla  \psi \right)^{-1}$. For general cases where $\mB \mB^{\top} \neq \mI$, there is no closed form for $\mathcal{S}$ and thus, the update $\mathbf{v}^{k+1}$ in \eqref{iter: prox v ell0} cannot be computed exactly. According to \cite{micchelli2011proximity}, the operator $\mathcal{S}$ may be computed by solving a fixed-point equation.
Therefore, we can only obtain approximate values of $\mathbf{v}^{k+1}$. This leads to an inexact fixed-point approach: Suppose that $\mathbf{\mathbf{\tilde \epsilon}}^{k+1} \in \bR^{m}$ is a predetermined  error and {$\alpha \in (0, 1]$}.  Assuming that $\mathbf{\mathbf{\tilde u}}^k$ and $\mathbf{\tilde v}^k$ are given, we find
\begin{align}
    &\mathbf{\mathbf{\tilde u}}^{k+1} \in \text{prox}_{\alpha \gamma \norm{\cdot}_0}\left( \left(1 - \alpha\right) \mathbf{\mathbf{\tilde u}}^k+\alpha \mD \mathbf{\tilde v}^k \right), \label{iter: prox u ell0 inexact}\\
    &\mathbf{\tilde v}^{k+1} = \mathbf{\tilde v}^{k+1}_* + \mathbf{\tilde{ \mathbf{\epsilon}}}^{k+1},  \label{iter: prox v ell0 inexact} 
\end{align}
where 
\begin{equation}\label{eq: inexact exact v}
\mathbf{\tilde v}^{k+1}_* := \mathcal{S}\left(\mD^{\top} \mathbf{\mathbf{\tilde u}}^{k+1} \right)
\end{equation}
and define $\mathbf{\tilde v}^{1}_{*}:= \mathbf{\tilde v}^1$. 
The equation \eqref{iter: prox v ell0 inexact} is uncomputable since it is not feasible to calculate $\mathbf{\tilde v}^{k+1}_*$ exactly. In practice, computing an approximation of $\mathbf{\tilde v}^{k+1}_*$ requires an inner loop, which leads to an inexact algorithm.

We now describe a specific algorithm to compute an approximation of $\mathbf{\tilde v}_*^{k+1}$ at the $(k+1)$-step, which will serve as an inner loop for the inexact algorithm for solving model \eqref{model: l0}. By Theorem 3.1 of \cite{micchelli2011proximity}, for an $\mathbf{x}\in \bR^m$, we have that 
\begin{equation}\label{eq: prox chain rule}
\mathcal{S}(\mathbf{x}) = \mathbf{x} - \frac{\gamma}{\lambda}\mB^{\top} \mathbf{w},
\end{equation}
where  $\mathbf{w}$ satisfies the fixed-point equation
\begin{equation}\label{eq: fixed-point equation for v}
   \mathbf{w} = \frac{1}{q}\left(\mI - \left(\mI + q \nabla \psi \right)^{-1}\right)\circ \left( q \mathbf{w} + \mB \left(\mathbf{x} - \frac{\gamma}{\lambda}\mB^{\top} \mathbf{w}\right)  \right),
\end{equation}
for any positive parameter $q$. Note that introduction of the parameter $q$ allows us to choose its values so that the corresponding fixed-point algorithm for finding $\mathbf{w}$ converges.
When $\mathbf{\tilde{u}}^{k+1}$ is available, we can find an update $\mathbf{\tilde{v}}^{k+1}$ via employing the Fixed-Point Proximity Algorithm (FPPA) proposed in \cite{li2015multi,micchelli2011proximity} to solve equations \eqref{eq: prox chain rule} and \eqref{eq: fixed-point equation for v} numerically. 
For given positive constants $p$, $q$ and initial points $\mathbf{\tilde v}^{k+1}_1$, $\mathbf{\tilde w}^{k+1}_1$, the FPPA is described as
\begin{equation}\label{inner: FPPA}
\left\{\begin{array}{l}
\mathbf{\tilde v}^{k+1}_{l+1} = \frac{\lambda}{p\gamma + \lambda}\mD^{\top}\mathbf{\tilde u}^{k+1} + \frac{p\gamma}{p\gamma + \lambda}\left(\mathbf{\tilde v}^{k+1}_l - \frac{1}{p}\mB^{\top}\mathbf{\tilde w}^{k+1}_l\right),\\
\mathbf{\tilde w}^{k+1}_{l+1} = \frac{1}{q}\left(\mI- \left(\mI + q\nabla \psi\right)^{-1}\right)\left(q \mathbf{\tilde w}^{k+1}_{l} + \mB \left(2 \mathbf{\tilde v}^{k+1}_{l+1}- \mathbf{\tilde v}^{k+1}_{l}\right)\right).
\end{array}\right.
\end{equation}
From \cite{li2015multi,micchelli2011proximity}, iteration \eqref{inner: FPPA} converges if $\|\mathbf{B}\|_2^2<pq$. Computing an exact value of the proximity operator of $\psi \circ \mB$ requires executing infinite times of iteration \eqref{inner: FPPA}.  While, in practical implementation, iteration \eqref{inner: FPPA} can only be executed finite number of times. This means that we can only obtain an approximate surrogate of $\mathbf{\tilde v}^{k+1}_*$. The error between  $\mathbf{\tilde v}^{k+1}_*$ and the approximation will influence the global error and thus appropriately choosing a stopping criterion for the inner loop is essential.

We next motivate the choice of the stopping criterion for the inner loop.
To this end, we define
\begin{equation}\label{subproblem function}
 H(\mathbf{v}; \mathbf{\tilde u}^{k+1}) := \frac{\lambda}{2\gamma}\norm{ \mathbf{v} - \mD^{\top} \mathbf{\tilde u}^{k+1}}_2^2 + \psi(\mB \mathbf{v}), \ \  \mathbf{v} \in \bR^m.  
\end{equation}
Note that for fixed $\mathbf{\tilde u}^{k+1}$, $H(\cdot;  \mathbf{\tilde u}^{k+1})$ is a continuously differentiable convex function. From the definition of operator $\mathcal{S}$, we observe that $\mathbf{\tilde{v}}_*^{k+1}$ is equal to 
\begin{equation}\label{eq: RHS prox v}
\text{argmin} \left\{H(\mathbf{v}; \mathbf{\tilde u}^{k+1}) : \mathbf{v} \in \bR^m\right\}.    
\end{equation}
It is advantageous to have an equivalent form of \eqref{eq: RHS prox v}. Recalling the definition of $F$, we define the optimization problem
\begin{equation} \label{iter: prox v}
    \text{argmin}\left\{F\left(\mathbf{\tilde u}^{k+1}, \mathbf{v}\right) :  \mathbf{v}\in \bR^m\right\}.
\end{equation}
For the sake of simplicity,  we introduce the similarity notation of two functions. We say two functions $f$ and $g$ are {\it similar} if their difference is a constant and we write it as $f \sim g$. Hence, if $f \sim g$, then
$$
\text{argmin} \left\{f(\mathbf{x}): \mathbf{x} \in \bR^m\right\}  = \text{argmin}\left\{g(\mathbf{x}): \mathbf{x} \in \bR^m\right\}.
$$
Two minimization problems are said to be equivalent if they have the same set of global minimizers.

\begin{lemma}\label{lemma: prox equivance}
The minimization problems  \eqref{eq: RHS prox v} and \eqref{iter: prox v} are equivalent.
\end{lemma}
\begin{proof}
From Lemma \ref{lemma: relation u-Dv and v-DTu}, the difference between $\norm{\mathbf{\tilde u}^{k+1} - \mD \mathbf{v}}_2^2$ and $ \norm{\mathbf{v} - \mD^{\top} \mathbf{\tilde u}^{k+1}}_2^2$ is $\left(\mathbf{\tilde u}^{k+1}\right)^{\top}(\mI - \mD \mD^{\top}) \mathbf{\tilde u}^{k+1}$, which is a constant with respect to $\mathbf{v}$. Notice that the term $\lambda \norm{\mathbf{\tilde u}^{k+1}}_0$ is also a constant  with respect to $\mathbf{v}$. Hence, 
$$
\frac{\lambda}{2\gamma}\norm{\mathbf{v} - \mD^{\top}\mathbf{\tilde u}^{k+1}}_2^2 + \psi(\mB \mathbf{v})  
 \sim  \psi(\mB \mathbf{v}) + \frac{\lambda}{2\gamma}\norm{\mD \mathbf{v} - \mathbf{\tilde u}^{k+1}}_2^2 + \lambda \norm{\mathbf{\tilde u}^{k+1}}_0.
$$
By the definition of $F$, the minimization problems \eqref{eq: RHS prox v} and \eqref{iter: prox v} are equivalent.  \qed
\end{proof}

We next review the notion of the strictly convex function (see, for example, Chapters 8 and 11 of \cite{bauschke2011convex}). 
Let $f: \bR^d \to \bR \cup \left\{\infty\right\}$ be a proper function. A function $f$ is strictly convex if
$$
f\left(t\mathbf{x} + (1 - t)\mathbf{y} \right) < t f(\mathbf{x}) + (1 - t) f(\mathbf{y}), \ \ \mbox{for all}\ \ t \in (0, 1) \ \mbox{and}\  \mathbf{x}, \mathbf{y}\in\bR^d \ \mbox{with}\ \mathbf{x} \neq \mathbf{y}.
$$
It is well-known that the minimizer of a proper strictly convex function is unique. We use $\mathbb{N}$ to denote the set of all positive integers.

\begin{lemma}\label{lemma: nonzero condition}
       Suppose that {$\alpha \in (0, 1]$} and $\left\{\left(\mathbf{\tilde u}^{k}, \mathbf{\tilde v}^{k}, \mathbf{\tilde v}^{k}_{*}\right)\right\}_{k=1}^{\infty}$ is the sequence defined by \eqref{iter: prox u ell0 inexact}-\eqref{eq: inexact exact v}. Let $k\in \bN$ be fixed.
   \begin{enumerate}
       \item [(a)] If  $\norm{\mathbf{\tilde u}^{k+1} - \mathbf{\tilde u}^{k}}_2^2 + \norm{\nabla H(\mathbf{\tilde v}^{k}; \mathbf{\tilde u}^{k+1})}_2^2 = 0$, then $(\mathbf{\tilde u}^k, \mathbf{\tilde v}^k)$ is a fixed-point of \eqref{eq: fix-point 1}-\eqref{eq: fix-point 2}.

       \item [(b)]
       
       If $\norm{\mathbf{\tilde u}^{k+1} - \mathbf{\tilde u}^{k}}_2^2 + \norm{\nabla H(\mathbf{\tilde v}^{k}; \mathbf{\tilde u}^{k+1})}_2^2 > 0$, then either $\norm{\mathbf{\tilde u}^{k+1} - \mathbf{\tilde u}^{k}}_2^2 >0$ or $\norm{\mathbf{\tilde u}^{k+1} -
       \mathbf{\tilde u}^{k}}_2^2  = 0$ and 
    
\begin{equation}\label{eq: strict greater}
   F(\mathbf{\tilde u}^{k+1}, \mathbf{\tilde v}^k) > F(\mathbf{\tilde u}^{k+1}, \mathbf{\tilde v}^{k+1}_*). 
\end{equation}   \end{enumerate}
\end{lemma}
\begin{proof}
We first prove (a). The hypothesis of (a) is equivalent to $\mathbf{\tilde u}^{k+1} = \mathbf{\tilde u}^k$ and $\nabla H(\mathbf{\tilde v}^k; \mathbf{\tilde u}^{k+1}) = 0$. It follows from \eqref{iter: prox u ell0 inexact} with $\mathbf{\tilde u}^{k+1} = \mathbf{\tilde u}^k$ that
\begin{equation}\label{eq: zero condition proof 1}
     \mathbf{\tilde u}^{k} \in \text{prox}_{\alpha \gamma \norm{\cdot}_0}\left((1 
 -\alpha)\mathbf{\tilde u}^k + \alpha \mD \mathbf{\tilde v}^k\right).    
    \end{equation}
    Since $H(\cdot; \mathbf{\tilde u}^{k+1}  )$ is a  continuously differentiable convex function and $\nabla H(\mathbf{\tilde v}^k ; \mathbf{\tilde u}^{k+1})=0$, we observe that $\mathbf{\tilde v}^{k}$ is the global minimizer of problem \eqref{eq: RHS prox v}. 
Since $\psi \circ \mB$ is continuously differentiable, we have that
$
\mathbf{\tilde v}^k = \mathcal{S}\left(\mD^{\top} \mathbf{\tilde u}^{k+1} \right)    
$.
This together with $\mathbf{\tilde u}^{k+1} = \mathbf{\tilde u}^k$ yeilds that
\begin{equation}\label{eq: zero condition proof 2}
\mathbf{\tilde v}^k = \mathcal{S}\left(\mD^{\top} \mathbf{\tilde u}^{k} \right).
\end{equation}
Equations \eqref{eq: zero condition proof 1} and \eqref{eq: zero condition proof 2} show that $(\mathbf{\tilde u}^k, \mathbf{\tilde v}^k)$ is a fixed-point of \eqref{eq: fix-point 1}-\eqref{eq: fix-point 2}.

We next establish (b).  It suffices to show that $\norm{\mathbf{\tilde u}^{k+1} - \mathbf{\tilde u}^k}_2^2 = 0$ implies the inequality \eqref{eq: strict greater}.
The hypothesis of (b) together with $\norm{\mathbf{\tilde u}^{k+1} - \mathbf{\tilde u}^k}_2^2 = 0$ leads to $\nabla H(\mathbf{\tilde v}^k; \mathbf{\tilde u}^{k+1}) \neq 0$. On the other hand, $\mathbf{\tilde v}^{k+1}_*$ is the global minimizer of \eqref{eq: RHS prox v}. Therefore, $\nabla H(\mathbf{\tilde v}^{k+1}_*; \mathbf{\tilde u}^{k+1}) = 0$, which implies that $\mathbf{\tilde v}^{k+1}_{*} \neq \mathbf{\tilde v}^k$. We next show \eqref{eq: strict greater}. Noting that $\mathbf{\tilde v}^{k+1}_*$ is the global minimizer of \eqref{eq: RHS prox v},  Lemma \ref{lemma: prox equivance} ensures that
$
\mathbf{\tilde v}^{k+1}_* = \text{argmin} \left\{F(\mathbf{\tilde u}^{k+1},  \mathbf{v}): \mathbf{v} \in \bR^m\right\}.
$
By employing the identity in Lemma \ref{lemma: relation u-Dv and v-DTu}, we observe that
$$
F(\mathbf{\tilde u}^{k+1},  \mathbf{v}) = \psi(\mB \mathbf{v}) + \frac{\lambda}{2\gamma}\norm{\mathbf{v} - \mD^{\top} \mathbf{\tilde u}^{k+1}}_2^2 + \frac{\lambda}{2\gamma}\mathbf{\tilde u}^{k+1}(\mI - \mD \mD^{\top})\mathbf{\tilde u}^{k+1} + \lambda \norm{\mathbf{\tilde u}^{k+1}}_0.
$$
The convexity of $\psi$ ensures that $F(\mathbf{\tilde u}^{k+1}, \cdot)$ is strictly convex for any fixed $\mathbf{\tilde u}^{k+1}$. Hence, the minimizer of $F(\mathbf{\tilde u}^{k+1}, \cdot)$ is unique. Therefore, the fact $\mathbf{\tilde v}^k \neq \mathbf{\tilde v}^{k+1}_* $ leads to  \eqref{eq: strict greater}. \qed
\end{proof}

Now, we provide insights leading to the stopping criterion for the inner loop in the inexact algorithm to be described.
Let { $\alpha \in (0, 1]$} and $\left\{\left(\mathbf{\tilde u}^{k}, \mathbf{\tilde v}^{k}, \mathbf{\tilde v}^{k}_{*}\right)\right\}_{k=1}^{\infty}$ be the sequence defined by \eqref{iter: prox u ell0 inexact}-\eqref{eq: inexact exact v}. We define a subset of $\bN$ by
\begin{equation}\label{eq: set K}
\mathbb{K}:= \left\{k\in \mathbb{N}: \norm{\mathbf{\tilde u}^{k+1} - \mathbf{\tilde u}^k}_2^2 + \norm{\nabla H(\mathbf{\tilde v}^k; 
\mathbf{\tilde u}^{k+1})}_2^2 > 0 \right\}.    
\end{equation}

\begin{lemma}\label{lemma: fullied condition modified}
    Suppose that {$\alpha \in (0, 1]$}, $\left\{\left(\mathbf{\tilde u}^{k}, \mathbf{\tilde v}^{k}, \mathbf{\tilde v}^{k}_{*}\right)\right\}_{k=1}^{\infty}$ is the sequence defined by \eqref{iter: prox u ell0 inexact}-\eqref{eq: inexact exact v}. Let $\rho'$ be a given positive constant and $\left\{e^k\right\}_{k=1}^{\infty}$ a given sequence of positive numbers. If for each $k \in \mathbb{K}$,  $\left\{\mathbf{x}^{k+1}_l\right\}_{l=1}^{\infty}$ is a sequence convergent to $\mathbf{\tilde v}^{k+1}_*$, 
then there exists a positive integer $L_{k}$ such that for all $l \geq L_k$, 
\begin{equation}\label{ineq: condition for v inexact modified}
F\left(\mathbf{\tilde u}^{k+1}, \mathbf{x}^{k+1}_{l}\right)  - F\left(\mathbf{\tilde u}^{k+1}, \mathbf{\tilde v}^{k}\right)\leq \frac{\rho'}{2}\norm{\mathbf{\tilde u}^{k+1} - \mathbf{\tilde u}^k}_2^2
\end{equation}
and
\begin{equation}\label{error control modified} 
    \norm{\nabla H\left( \mathbf{x}^{k+1}_{l}; \mathbf{\tilde u}^{k+1}\right)}_2 \leq e^{k+1}. 
\end{equation}
\end{lemma}
\begin{proof}

Let $k\in \mathbb{K}$ be fixed. We first show inequality \eqref{ineq: condition for v inexact modified}. Since $\psi$ is continuous, so is $F\left(\mathbf{\tilde u}^{k+1}, \cdot\right)$. Because $\left\{\mathbf{x}^{k+1}_l\right\}_{l=1}^{\infty}$ is a sequence convergent to $\mathbf{\tilde v}^{k+1}_*$, we find that 
\begin{equation}\label{eq: F converges modified}
    \lim_{l \to \infty} F(\mathbf{\tilde u}^{k+1},  \mathbf{x}^{k+1}_l) = F(\mathbf{\tilde u}^{k+1}, \mathbf{\tilde v}^{k+1}_*).
 \end{equation}
 From (b) of Lemma \ref{lemma: nonzero condition}, we need to show \eqref{ineq: condition for v inexact modified} by considering two cases: (i) $\norm{\mathbf{\tilde u}^{k+1} - \mathbf{\tilde u}^k}_2^2 >0$ and (ii) $\norm{\mathbf{\tilde u}^{k+1} - \mathbf{\tilde u}^k}_2^2 = 0$ with \eqref{eq: strict greater}. For case (i),
 by \eqref{eq: F converges modified}, there exists a positive integer $L_{k,1}$ such that for all $l \geq L_{k,1}$, 
\begin{equation}\label{eq: condition of v in proof modified}
    F\left(\mathbf{\tilde u}^{k+1}, \mathbf{x}^{k+1}_{l}\right) - F\left(\mathbf{\tilde u}^{k+1}, \mathbf{\tilde v}^{k+1}_*\right) \leq \frac{\rho'}{2}\norm{\mathbf{\tilde u}^{k+1} - \mathbf{\tilde u}^{k}}_2^2.
\end{equation}
Since $\mathbf{\tilde v}^{k+1}_*$ is the global minimizer of \eqref{eq: RHS prox v}, it follows from Lemma \ref{lemma: prox equivance} that 
$$
F\left(\mathbf{\tilde u}^{k+1}, \mathbf{\tilde v}^{k+1}_* \right) - F(\mathbf{\tilde u}^{k+1}, \mathbf{\tilde v}^k) \leq 0.
$$
Adding the above inequality to inequality \eqref{eq: condition of v in proof modified} yields inequality 
\eqref{ineq: condition for v inexact modified} for all $l \geq L_{k_1}$. For case (ii), inequality \eqref{ineq: condition for v inexact modified} is equivalent to
\begin{equation}\label{eq: zero case modified}
F(\mathbf{\tilde u}^{k+1},  \mathbf{x}^{k+1}_l) - F(\mathbf{\tilde u}^{k+1}, \mathbf{\tilde v}^k) \leq 0.
\end{equation}
It suffices to show \eqref{eq: zero case modified} for all $l$ greater than some positive integer. From inequality \eqref{eq: strict greater}, 
$F(\mathbf{\tilde u}^{k+1}, \mathbf{\tilde v}^{k}) - F(\mathbf{\tilde u}^{k+1}, \mathbf{\tilde v}^{k+1}_*)$ is positive.
Therefore, equation \eqref{eq: F converges modified} ensures that there exists a positive integer $L_{k, 2}$ such that
$$
F(\mathbf{\tilde u}^{k+1},  \mathbf{x}^{k+1}_l) - F(\mathbf{\tilde u}^{k+1}, \mathbf{\tilde v}^{k+1}_*)  \leq 
F(\mathbf{\tilde u}^{k+1}, \mathbf{\tilde v}^{k}) - F(\mathbf{\tilde u}^{k+1}, \mathbf{\tilde v}^{k+1}_*)
$$
for all $l\geq L_{k, 2}$. This is equivalent to inequality \eqref{eq: zero case modified} for all $l \geq L_{k,2}$, which completes the proof of inequality \eqref{ineq: condition for v inexact modified}.

We next show inequality \eqref{error control modified}. Since $\mathbf{\tilde v}^{k+1}_*$ is the global minimizer of \eqref{eq: RHS prox v}, $\nabla H\left( \mathbf{\tilde v}^{k+1}_*; \mathbf{\tilde u}^{k+1}\right) = 0$. Noting that $\psi$ is continuously differentiable, we conclude that $\nabla H\left(\cdot; \mathbf{\tilde u}^{k+1}\right)$ is continuous. Again, because $\left\{\mathbf{x}^{k+1}_l\right\}_{l=1}^{\infty}$ is a sequence convergent to $\mathbf{\tilde v}^{k+1}_*$, we have that $\lim_{l \to \infty }\norm{\nabla H\left( \mathbf{x}^{k+1}_l; \mathbf{\tilde u}^{k+1}\right)}_2 = 0$. Consequently, the positivity of $e^{k+1}$
ensures that
there exists a positive integer $L_{k,3}$ such that for all $l\geq L_{k,3}$, inequality \eqref{error control modified} holds true.

The proof is completed by choosing $L_k:= \max\{L_{k,1}, L_{k,3}\}$ if $\norm{\mathbf{\tilde u}^{k+1} - \mathbf{\tilde u}^k}_2^2 > 0$ and $L_k:= \max\{L_{k,2}, L_{k,3}\}$ if $\norm{\mathbf{\tilde u}^{k+1} - \mathbf{\tilde u}^k}_2^2 = 0$.  \qed
\end{proof}

Inequalities \eqref{ineq: condition for v inexact modified} and \eqref{error control modified} in Lemma  \ref{lemma: fullied condition modified} may be used as a stopping criterion for the inner loop iteration  \eqref{inner: FPPA}.
Combining  \eqref{iter: prox u ell0 inexact}-\eqref{iter: prox v ell0 inexact}, iteration \eqref{inner: FPPA} for finding the proximity operator of $\psi \circ \mB$, with the stopping criterion provided by Lemma \ref{lemma: fullied condition modified}, we obtain an executable inexact fixed-point proximity algorithm, for model \eqref{model: l0}. We summarize it in Algorithm \ref{algo: inexact FPPA l0}.

\begin{algorithm}[hbt!]
  \caption{Inexact fixed-point proximity algorithm for problem \eqref{model: l0}.}
  \label{algo: inexact FPPA l0}
  \KwInput{ The function $\psi$, matrices $\mB$ and $\mD$, constants $\lambda$ and $\gamma$ in problem  \eqref{model: l0}; constants {$\alpha \in (0,1]$} and $\rho' \in (0, \infty)$; positive sequence $\left\{e^{k+1}\right\}_{k=1}^{\infty}$; parameters $p>0, q>0$}
  \KwInitialization{$\mathbf{\tilde u}^1 , \mathbf{\tilde v }^1, \mathbf{\tilde w}^1$}

   \For{k = 1, 2, \ldots}
   {
   		$\mathbf{\tilde u}^{k+1} \in \text{prox}_{\alpha \gamma \norm{\cdot}_0}\left( \left(1 - \alpha\right) \mathbf{\tilde u}^k+\alpha \mD^{\top} \mathbf{\tilde v}^k \right)$

     $\mathbf{\tilde v}^{k+1}_1 := \mathbf{\tilde v}^k, \mathbf{\tilde w}^{k+1}_1 := \mathbf{\tilde w}^k, l=1$

        \eIf{ $\norm{\mathbf{\tilde u}^{k+1} - \mathbf{\tilde u}^k}_2^2 + \norm{\nabla H(\mathbf{\tilde v}^k; \mathbf{\tilde u}^{k+1})}_2^2  > 0$}
        {\Do{{\footnotesize $F\left(\mathbf{\tilde u}^{k+1}, \mathbf{\tilde v}^{k+1}_{l}\right)  - F\left(\mathbf{\tilde u}^{k+1}, \mathbf{\tilde v}^{k}\right) > \frac{\rho'}{2}\norm{\mathbf{\tilde u}^{k+1} - \mathbf{\tilde u}^k}_2^2$ or $\norm{\nabla H\left(\mathbf{\tilde v}^{k+1}_{l}; \mathbf{\tilde u}^{k+1}\right)}_2 > e^{k+1}$
        }}
        {
    $\mathbf{\tilde v}^{k+1}_{l+1} = \frac{\lambda}{p\gamma + \lambda}\mD^{\top}\mathbf{\tilde u}^{k+1} + \frac{p\gamma}{p\gamma + \lambda}\left(\mathbf{\tilde v}^{k+1}_l - \frac{1}{p}\mB^{\top}\mathbf{\tilde w}^{k+1}_l\right)$

    $\mathbf{\tilde w}^{k+1}_{l+1} = \frac{1}{q}\left(\mI- \left(\mI + q\nabla \psi\right)^{-1}\right)\left(q \mathbf{\tilde w}^{k+1}_{l} + \mB \left(2 \mathbf{\tilde v }^{k+1}_{l+1}- \mathbf{\tilde v }^{k+1}_{l}\right)\right)$
         
         $l:=l+1$
        }
        $\mathbf{\tilde{v}}^{k+1}:=\mathbf{\tilde v}^{k+1}_{l}$, $\mathbf{\tilde w}^{k+1} := \mathbf{\tilde w}^{k+1}_{l}$
        }
        {$\mathbf{\tilde v}^{k+1}:=\mathbf{\tilde v}^k$, $\mathbf{\tilde w}^{k+1} :=\mathbf{\tilde w}^{k}$}
        }

\KwOutput{$\left\{\left(\mathbf{\tilde u}^{k}, \mathbf{\tilde v}^k\right)\right\}_{k=1}^{\infty}$.}
\end{algorithm}

The next proposition confirms that Algorithm \ref{algo: inexact FPPA l0} is executable in the sense that it will generate 
a sequence $\left\{\left(\mathbf{\tilde u}^k, \mathbf{\tilde v}^k\right)\right\}_{k=1}^{\infty}$.

\begin{proposition}\label{prop: executable}
Suppose that $\psi$ is a proper,  continuously differentiable,  and bounded below convex function, $\mB$ is a $p\times m$ matrix, $\mD$ is an $n\times m$ matrix, and $\lambda, \gamma$ are positive constants. If {$\alpha \in (0, 1]$} and 
$\rho'$ is a positive constant, $\left\{e^k\right\}_{k=1}^{\infty}$ is a positive sequence,
and $pq > \norm{\mB}_2^2$, then  Algorithm \ref{algo: inexact FPPA l0} is executable.
\end{proposition}

\begin{proof}
To show the executability of Algorithm \ref{algo: inexact FPPA l0}, it suffices to establish that for each $k \in \mathbb{N}$, given $\left(\mathbf{\tilde u}^{k}, \mathbf{\tilde v}^{k}\right)$, the pair $(\mathbf{\tilde u}^{k+1}, \mathbf{\tilde v}^{k+1})$ can be obtained within a finite number of iteration steps. Since the proximity operator of $\norm{\cdot}_0$ has a closed form, $\mathbf{\tilde u}^{k+1}$ in Line 2 of Algorithm \ref{algo: inexact FPPA l0} can be computed directly. If $\norm{\mathbf{\tilde u}^{k+1} - \mathbf{\tilde u}^{k}}_2^2 + \norm{\nabla H(\mathbf{\tilde v}^{k}; \mathbf{\tilde u}^{k+1})}_2^2 > 0$, then Algorithm \ref{algo: inexact FPPA l0} will enter the do-while inner loop.
Since $pq > \norm{\mB}_2^2$, it follows from Theorem 6.4 of \cite{li2015multi} that the sequence $\left\{\mathbf{\tilde v}^{k+1}_{l}\right\}_{l=1}^{\infty}$ generated by
the inner loop converges to $\mathbf{\tilde v}^{k+1}_*$.
Hence, according to Lemma \ref{lemma: fullied condition modified}, we conclude that conditions \eqref{ineq: condition for v inexact modified} and \eqref{error control modified} with $\mathbf{x}_l^{k+1}:=\mathbf{\tilde v}_l^{k+1}$ are satisfied within a finite number $l$ of steps.
Therefore, $\mathbf{\tilde v}^{k+1}$ can be computed within a finite number of the inner iterations. If $\norm{\mathbf{\tilde u}^{k+1} - \mathbf{\tilde u}^{k}}_2^2 + \norm{\nabla H(\mathbf{\tilde v}^{k}; \mathbf{\tilde u}^{k+1})}_2^2 = 0$, then from Line 12 of Algorithm \ref{algo: inexact FPPA l0} we have that  $\mathbf{\tilde v}^{k+1} = \mathbf{\tilde v}^k$. Thus, Algorithm \ref{algo: inexact FPPA l0} is executable.  \qed
\end{proof}

Proposition \ref{prop: executable} guarantees that Algorithm \ref{algo: inexact FPPA l0} will generate a sequence $\left\{\left(\mathbf{\tilde u}^k, \mathbf{\tilde v}^k\right)\right\}_{k=1}^{\infty}$, which is for sure to satisfy the conditions
\begin{equation}\label{condition for v inexact general}
    F(\mathbf{\tilde u}^{k+1}, \mathbf{\tilde v}^{k+1}) - F(\mathbf{\tilde u}^{k+1},\mathbf{\tilde v}^k) \leq \frac{\rho'}{2}\norm{\mathbf{\tilde u}^{k+1} - \mathbf{\tilde u}^k}_2^2
\end{equation}
and 
\begin{equation}\label{condition for v inexact general-2}
\norm{\nabla H\left(\mathbf{\tilde v}^{k+1}; \mathbf{\tilde u}^{k+1}\right)}_2 \leq e^{k+1}.
\end{equation}
These two inequalities will play a crucial role in convergence analysis, of Algorithm  \ref{algo: inexact FPPA l0}, which will be provided in the next section.

\section{Convergence Analysis}\label{section: convergence}

In this section we analyze the convergence of the inexact fixed-point proximity algorithm. Specifically, we will show that a sequence $\left\{\left(\mathbf{\tilde u}^k, \mathbf{\tilde v}^k\right)\right\}_{k=1}^{\infty}$ generated by Algorithm \ref{algo: inexact FPPA l0} converges to a local minimizer of \eqref{model: l0}. Since the sequence $\left\{(\mathbf{\tilde u}^{k}, \mathbf{\tilde v}^{k})\right\}_{k=1}^{\infty}$ generated from Algorithm \ref{algo: inexact FPPA l0} is a sequence defined by \eqref{iter: prox u ell0 inexact}-\eqref{iter: prox v ell0 inexact} satisfying conditions \eqref{condition for v inexact general} and \eqref{condition for v inexact general-2}, we shall study the convergence property of sequences defined by \eqref{iter: prox u ell0 inexact} and \eqref{iter: prox v ell0 inexact}. For this purpose, we reformulate the iterations \eqref{iter: prox u ell0 inexact} and \eqref{iter: prox v ell0 inexact} as an inexact Picard sequence of a nonlinear map. We then show their convergence by studying the property of the map. Finally, we show the convergence property of inexact fixed-point proximity algorithm (Algorithm \ref{algo: inexact FPPA l0}) based on the convergence of iterations \eqref{iter: prox u ell0 inexact} and \eqref{iter: prox v ell0 inexact}. 

We first state the main results of this section.

\begin{theorem}\label{thm : convergence algo INFPPA}
{ Suppose that $\psi$ is a proper, continuously differentiable, and bounded below convex function, $\mathbf{B}$ is a $p \times m$ matrix, $\mathbf{D}$ is an $n \times m$ matrix satisfying $\mathbf{D}^{\top}\mathbf{D} = \mathbf{I}$, $\lambda$ and $\gamma$ are positive constants, $\alpha \in (0, 1)$,  and the operator $\mathcal{T}$ defined in \eqref{eq: nonlinear map} has at least one fixed-point.}  If $p>0, q>0$, $pq > \norm{\mB}_2^2$,  
$\rho' \in \left(0, \frac{\lambda}{\gamma}\frac{1 - \alpha}{\alpha}\right)$ and $\sum_{k=1}^{\infty}e^{k+1}<\infty$, then $\left\{(\mathbf{\tilde u}^k, \mathbf{\tilde v}^k)\right\}_{k = 1}^{\infty}$ generated from Algorithm \ref{algo: inexact FPPA l0} converges to a local minimizer $(\mathbf{u}^*, \mathbf{v}^*)$ of problem \eqref{model: l0}.
\end{theorem}

{ This section is devoted to providing a proof for Theorem \ref{thm : convergence algo INFPPA}. To this end, we establish a number of preliminary results.} Throughout this section, we assume without further mentioning that $\psi: \bR^p \to \bR$ is a proper, continuously differentiable,  and bounded below convex function, $\mB$ is a $p \times m$ matrix, {$\mD$ is an $n \times m$ matrix satisfying $\mathbf{D}^{\top}\mathbf{D} = \mathbf{I}$}, $\lambda$ and $\gamma$ are positive constants, and $\alpha \in (0, 1)$. 

{ Note that in Theorem \ref{thm : convergence algo INFPPA}, we assume that $\alpha\in (0,1)$ even though Algorithm \ref{algo: inexact FPPA l0} may still converge for $\alpha$ equal to 1 or slightly bigger than 1, see the numerical examples presented in section 5. For $\alpha \geq 1$, convergence of Algorithm \ref{algo: inexact FPPA l0} is not guaranteed since the technical proof presented in this section requires $\alpha\in (0,1)$. Moreover, there are examples exhibiting divergence of Algorithm  \ref{algo: inexact FPPA l0} when $\alpha>1$.}

We prepare for the proof of Theorem \ref{thm : convergence algo INFPPA}.
The iterations \eqref{iter: prox u ell0 inexact} and \eqref{iter: prox v ell0 inexact} have two stages. 
Its first stage basically searches the support of the candidate sparse vectors $\mathbf{u}$, in model \eqref{model: l0}, which involve the $\ell_0$-norm.  Once it is obtained, the support will remain unchanged in future iterations and thus, the original non-convex optimization problem \eqref{model: l0} reduces to one with $\mathbf{u}$ restricted to the fixed support. The restricted optimization problem becomes convex since the term $\lambda\|\mathbf{u}\|_0$ reduces to a fixed number independent of $\mathbf{u}$. For this reason, the second stage of the iterations is to solve the resulting convex minimization problem.

We now show that the first $N$ steps of iterations  \eqref{iter: prox u ell0 inexact} and \eqref{iter: prox v ell0 inexact}, for certain positive integer $N$, is to pursue a support of the sparse vector $\mathbf{u}$. By the definition of the proximity operator, the right hand side of \eqref{iter: prox u ell0 inexact} is
\begin{equation}\label{eq: RHS prox u}
 \text{argmin}\left\{  \frac{1}{2 \alpha \gamma}\norm{\mathbf{u} - (1 - \alpha) \mathbf{\tilde u}^k - \alpha \mD \mathbf{\tilde v}^k}_2^2 + \norm{\mathbf{u}}_0: \mathbf{u}\in \bR^n\right\}. 
\end{equation}
We establish an equivalent form of \eqref{eq: RHS prox u}. To this end, we consider the optimization problem
\begin{equation}\label{iter: prox u}
    \text{argmin}\left\{F\left(\mathbf{u}, \mathbf{\tilde v}^{k}\right) + \frac{\lambda}{2\gamma } \frac{1 - \alpha}{\alpha}\norm{\mathbf{u} - \mathbf{\tilde u}^k}_2^2:  \mathbf{u}\in \bR^n\right\}.
\end{equation}

\begin{lemma}\label{lemma: prox equivance u}
The minimization problems \eqref{eq: RHS prox u} and \eqref{iter: prox u} are equivalent.
\end{lemma}

\begin{proof}
Expanding the quadratic term in the objective function of minimization problem  \eqref{eq: RHS prox u}, we have that
\begin{align*}
     &\frac{1}{2 \alpha \gamma}\norm{\mathbf{u} -  (1 - \alpha)\mathbf{\tilde u}^k - \alpha \mD \mathbf{\tilde v}^k}_2^2 + \norm{\mathbf{u}}_0 \\
     \sim &\frac{1-\alpha}{2\alpha \gamma}\langle \mathbf{u}, \mathbf{u} \rangle + \frac{1}{2 \gamma}\langle \mathbf{u}, \mathbf{u} \rangle - \frac{1-\alpha}{\alpha \gamma}\langle \mathbf{u}, \mathbf{\tilde u}^k \rangle - \frac{1}{ \gamma}\langle \mathbf{u}, \mD \mathbf{\tilde v}^k \rangle + \norm{\mathbf{u}}_0\\
     \sim &\frac{1}{\lambda}\left[\psi(\mB \mathbf{\tilde v}^k) + \frac{\lambda}{2\gamma}\norm{\mathbf{u} - \mD \mathbf{\tilde v}^k}_2^2 + \lambda\norm{\mathbf{u}}_0\right] + \frac{1}{2\gamma}\frac{1-\alpha}{\alpha}\norm{\mathbf{u} - \mathbf{\tilde u}^k}_2^2.
\end{align*}
Note that multiplying an objective function by a constant will not change its minimizer. Recalling the definition of $F$, we find that the minimization problems  \eqref{eq: RHS prox u} and \eqref{iter: prox u} are equivalent. \qed
\end{proof}

We next investigate 
the sequence generated by \eqref{iter: prox u ell0 inexact}-\eqref{iter: prox v ell0 inexact} with inequality \eqref{condition for v inexact general} satisfied.

\begin{lemma}\label{lemma: condition for v inexact}
If the sequence $\left\{\left(\mathbf{\tilde u}^k, \mathbf{\tilde v}^k\right)\right\}_{k=1}^{\infty}$  generated by \eqref{iter: prox u ell0 inexact}-\eqref{iter: prox v ell0 inexact} satisfies inequality \eqref{condition for v inexact general} with $\rho' \in \left(0, \frac{\lambda}{\gamma}\frac{1-\alpha}{\alpha} \right)$, then
\begin{enumerate}
    \item [(a)] $F\left(\mathbf{\tilde u}^{k+1}, \mathbf{\tilde v}^{k+1}\right) \leq F\left(\mathbf{\tilde u}^{k}, \mathbf{\tilde v}^{k}\right)$ and $\lim_{k \to \infty}F\left(\mathbf{\tilde u}^k, \mathbf{\tilde v}^k\right)$ exists;
    
    \item [(b)] $\lim_{k \to \infty} \norm{\mathbf{\tilde u}^{k+1} - \mathbf{\tilde u}^k}_{2} = 0$.
\end{enumerate}
\end{lemma}

\begin{proof}
By the definition of the proximity operator, $\mathbf{\tilde u}^{k+1}$ is  a global minimizer of $\eqref{eq: RHS prox u}$. It follows from Lemma
\ref{lemma: prox equivance u} that
$$
F\left(\mathbf{\tilde u}^{k+1}, \mathbf{\tilde v}^{k}\right) + 
\frac{\lambda}{2\gamma}\frac{1- \alpha}{\alpha}\norm{\mathbf{\tilde u}^{k+1} - \mathbf{\tilde u}^{k}}_2^2 \leq F\left(\mathbf{\tilde u}^k, \mathbf{\tilde v}^k\right).
$$
Adding the above inequality to inequality \eqref{condition for v inexact general} yields that
\begin{equation}\label{ineq: decrease obj inexact}
F\left(\mathbf{\tilde u}^{k+1}, \mathbf{\tilde v}^{k+1}\right) + \frac{1}{2}\left(\frac{\lambda}{\gamma}\frac{1- \alpha}{\alpha} - \rho'\right)\norm{\mathbf{\tilde u}^{k+1} -  \mathbf{\tilde u}^k}_2^2 \leq F\left(\mathbf{\tilde u}^{k}, \mathbf{\tilde v}^{k}\right).
\end{equation}
Noting that $\lambda, \gamma$ are positive, $\alpha \in (0, 1)$, and $\rho' \in \left(0, \frac{\lambda}{\gamma}\frac{1-\alpha}{\alpha}\right)$, the second term on the left hand side of the above inequality is non-negative, which implies the inequality in part (a). Since $F$ is bounded below, the inequality of part (a) leads to the existence of $\lim_{k \to \infty}F\left(\mathbf{\tilde u}^k, \mathbf{\tilde v}^k\right)$. This completes the proof of part (a).

It remains to show part (b). The second statement of part (a) ensures that $\lim_{k\to\infty}\left[F(\mathbf{\tilde u}^k, \mathbf{\tilde v}^k)-F(\mathbf{\tilde u}^{k+1},\mathbf{\tilde v}^{k+1})\right]=0$. This together with \eqref{ineq: decrease obj inexact} and the assumption $\frac{\lambda}{\gamma}\frac{1- \alpha}{\alpha} - \rho'>0$ yields part (b). \qed

\end{proof}

Next, we show that the support of the inexact sequence $\left\{\mathbf{\tilde u}^k\right\}_{k=1}^{\infty}$ will remain unchanged after a finite number of steps, which reduces finding a local minimizer of the non-convex optimization problem to solving a constrained convex optimization problem. We first review the support of a vector. For a positive integer $n$,  let $\mathbb{N}_n:=\{1, 2, \dots, n\}$. For $\mathbf{x} \in \bR^n$, we use $S(\mathbf{x})$ to denote the support of $\mathbf{x}$, which is defined as
$$
S(\mathbf{x}) := \left\{i \in \mathbb{N}_n: x_i \neq 0\right\}.
$$
Clearly, $S(\mathbf{x})$ is a subset of $\mathbb{N}_n$. It is empty if all components of $\mathbf{x}$ are zeros and it is equal to $\mathbb{N}_n$ if all components of $\mathbf{x}$ are nonzeros. It is known (for example, Lemma 3 of \cite{shen2016wavelet}) that if $\mathbf{u} \in \text{prox}_{t\|\cdot\|_0}(\mathbf{x})$, $\mathbf{v} \in \text{prox}_{t\|\cdot\|_0}(\mathbf{y})$, for $\mathbf{x}, \mathbf{y} \in \bR^n$,  then 
\begin{equation}\label{eq: same support}
S(\mathbf{u})\neq S(\mathbf{v}) \implies
\norm{\mathbf{u} - \mathbf{v}}_2\geq\sqrt{2t}. 
\end{equation}

\begin{proposition}\label{prop: explicit unchange support prox inexact}
If $\left\{\left(\mathbf{\tilde u}^k, \mathbf{\tilde v}^k\right)\right\}_{k = 1}^{\infty}$ is the sequence generated by \eqref{iter: prox u ell0 inexact}-\eqref{iter: prox v ell0 inexact} satisfying the condition \eqref{condition for v inexact general} with  $\rho' \in \left(0, \frac{\lambda}{\gamma}\frac{1-\alpha}{\alpha} \right)$, then there exists a positive integer $N$ such that $S\left(\mathbf{\tilde u}^k\right) = S\left(\mathbf{\tilde u}^{N}\right)$ for all $k\geq N$.
\end{proposition}
\begin{proof}
From Lemma \ref{lemma: condition for v inexact}, we have that
$
\lim_{k \to \infty}\norm{\mathbf{\tilde u}^{k+1} - \mathbf{\tilde u}^{k}}_2 = 0.
$
Hence, there exists an integer $N$, such that
$$
\norm{\mathbf{\tilde u}^{k+1} - \mathbf{\tilde u}^{k}}_2 < \sqrt{2\alpha \gamma}, \ \ \mbox{for all}\ \  k\geq N. 
$$
The above inequality together with \eqref{eq: same support} leads to $S\left(\mathbf{\tilde u}^{k+1}\right) = S\left(\mathbf{\tilde u}^{k}\right)$ for all $k\geq N$. This further ensures that $S\left(\mathbf{\tilde u}^k\right) = S\left(\mathbf{\tilde u}^{N}\right)$ for all $k\geq N$. \qed
\end{proof}

Proposition \ref{prop: explicit unchange support prox inexact} reveals that the support of $\mathbf{\tilde u}^k$ remains unchanged when $k\geq N$ and it turns out that the influence of $\norm{\cdot}_0$ does not exist anymore when $k\geq N$. Throughout this section, we use $C \subseteq \mathbb{N}_{n}$ to denote the unchanged support. Accordingly, the non-convex optimization problem \eqref{model: l0} reduces to a convex one. To describe the resulting convex optimization problem, we need the notion of the indicator function. For a given index set $\Lambda \subseteq \mathbb{N}_n$, we define the subspace of $\bR^n$ by letting
$$
\mathbf{B}_{\Lambda} := \left\{\mathbf{u} \in \bR^n: S(\mathbf{u}) \subseteq \Lambda \right\},
$$
and define the indicator function $\delta_{\Lambda}: \bR^n \to \bR \cup \left\{\infty\right\}$ as 
\begin{equation}\label{eq: indicator}
\delta_{\Lambda}(\mathbf{u}) := \begin{cases}
0, &\text{if} \ \ \mathbf{u} \in  \mathbf{B}_{\Lambda},\\
\infty, &\text{otherwise}.
\end{cases}    
\end{equation}
We define a convex function by
\begin{equation}\label{model: convex objective function}
G(\mathbf{u}, \mathbf{v}) := \psi(\mB \mathbf{v})+ \frac{\lambda}{2\gamma}\norm{\mathbf{u} - \mD \mathbf{v}}_2^2 + \delta_{C}(\mathbf{u}), \ \  (\mathbf{u}, \mathbf{v})\in\bR^n\times\bR^m.
\end{equation}
Proposition \ref{prop: explicit unchange support prox inexact} implies that when $k\geq N$, the non-convex optimization problem \eqref{model: l0} boils down to the convex optimization problem 
\begin{equation}\label{model: convex model}
 \mathrm{argmin}\{G(\mathbf{u}, \mathbf{v}): (\mathbf{u}, \mathbf{v}) \in \bR^n \times \bR^m\}.
\end{equation}
Specifically, we have the next result which follows directly from Theorem 4.8 of    \cite{2021Sparse}.

\begin{proposition}   \label{prop: global minimizer and local minimizer}
Suppose that $\left(\mathbf{u}^*, \mathbf{v}^*\right)\in\bR^n \times \bR^m$ and $S(\mathbf{u}^*) = C$.
The pair $\left(\mathbf{u}^*, \mathbf{v}^*\right)$ is a local minimizer of minimization problem \eqref{model: l0} if and only if $\left(\mathbf{u}^*, \mathbf{v}^*\right)$ is a global minimizer of minimization problem \eqref{model: convex model}.
\end{proposition}

From Proposition \ref{prop: global minimizer and local minimizer}, finding a local minimizer of optimization problem \eqref{model: l0} is equivalent to finding a global minimizer of optimization problem \eqref{model: convex model}.

We now turn to showing that the inexact sequence  $\left\{(\mathbf{\tilde u}^k, \mathbf{\tilde v}^k)\right\}_{k= N}^{\infty}$ converges to a local minimizer of optimization problem \eqref{model: l0}. To this end, we reformulate the sequence as an inexact Picard sequence of a nonlinear operator and then show the nonlinear operator has certain favorable property that guarantees convergence of a sequence generated from the operator. 
For a given index set $\Lambda \subseteq \mathbb{N}_n$, we define the projection operator $\mathcal{P}_{\Lambda} : \bR^n \to \mathbf{B}_{\Lambda}$ as
\begin{equation}\label{proj: bRm to unchange support}
   \left[\mathcal{P}_{\Lambda}(\mathbf{u}) \right]_i := \begin{cases}
   u_i, & \text{if} \ i\in \Lambda,\\
   0, & \text{otherwise},
   \end{cases}
\end{equation}
and the diagonal matrix $\mT_{\Lambda} \in \bR^{n \times n}$ as 
\begin{equation}\label{matrix: T_C}
    \left[\mT_{\Lambda}\right]_{i,i} := \begin{cases}
        1, & \text{if} \ i\in \Lambda,\\
        0, & \text{if} \ i \notin \Lambda.
    \end{cases}
\end{equation}
From the definition of the operator  $\mathcal{P}_{\Lambda}$ and the matrix $\mT_{\Lambda}$, we have that
\begin{equation}\label{lemma: matrix = projection}
\mathcal{P}_{\Lambda}(\mathbf{u})  = \mT_{\Lambda} \mathbf{u}, \ \    \mbox{for any}\ \ \mathbf{u} \in \bR^n.
\end{equation}

We next express the proximity operator of the $\ell_0$-norm in terms of the linear transformation $\mT_C$. To this end, we recall that for $\mathbf{x}, \mathbf{z} \in \bR^n$, if $\mathbf{x} \in \text{prox}_{t\norm{\cdot}_0}(\mathbf{z})$, then $\mathbf{x} = \mathcal{P}_{S(\mathbf{x})}(\mathbf{z})$ (see, Lemma 5 of \cite{wu2022inverting}).

\begin{lemma}\label{prop: prox_norm0 and projection}
If $k \geq  N$ and $\mathbf{\tilde u}^{k+1} \in \mathrm{prox}_{\alpha \gamma \norm{\cdot}_0}(\mathbf{z}^k)$ where $\mathbf{z}^k:=(1 - \alpha) \mathbf{\tilde u}^{k} + \alpha \mD \mathbf{\tilde  v}^{k}$, then $\mathbf{\tilde u}^{k+1} =\mT_{C}(\mathbf{z}^k)$. 
\end{lemma}

\begin{proof}
By the hypothesis of this lemma, according to Lemma 5 of \cite{wu2022inverting}, we see that 
$\mathbf{\tilde u}^{k+1} = \mathcal{P}_{S(\mathbf{\tilde u}^{k+1})}
(\mathbf{z}^k)$.
From Proposition \ref{prop: explicit unchange support prox inexact}, we have that $S(\mathbf{\tilde u}^{k+1}) = C$ for all $k\geq N$. Hence,
$
\mathbf{\tilde u}^{k+1} = \mathcal{P}_{C}
(\mathbf{z}^k).
$
Combining this equation with \eqref{lemma: matrix = projection} completes the proof of this lemma. \qed
\end{proof}

We next reformulate iterations \eqref{iter: prox u ell0 inexact}-\eqref{iter: prox v ell0 inexact} as an inexact Picard sequence of a nonlinear operator for $k \geq N$. We define operator 
$\mathcal{T}:\bR^n\to \bR^n$ by
\begin{equation}\label{eq: nonlinear map}
\mathcal{T}(\mathbf{u}):=(1-\alpha)\mT_C \mathbf{u}+\alpha\mT_C\mD \mathcal{S}\left(\mD^{\top}  \mathbf{u} \right).
\end{equation}

\begin{proposition}\label{prop: inexact Picard sequence}
If $\left\{\left(\mathbf{\tilde u}^k, \mathbf{\tilde v}^k\right)\right\}_{k=N}^{\infty}$ is a sequence generated by \eqref{iter: prox u ell0 inexact}-\eqref{iter: prox v ell0 inexact}, then 
\begin{equation}\label{iter: Picard T inexact}
\mathbf{\tilde u}^{k+1}
= \mathcal{T}\mathbf{\tilde u}^{k} + \alpha\mT_C\mD \mathbf{\tilde\epsilon}^{k}.
\end{equation}
\end{proposition}

\begin{proof} Suppose that sequence $\left\{\left(\mathbf{\tilde u}^k, \mathbf{\tilde v}^k\right)\right\}_{k=N}^{\infty}$ satisfies \eqref{iter: prox u ell0 inexact}-\eqref{iter: prox v ell0 inexact}. It follows from \eqref{iter: prox v ell0 inexact} and \eqref{eq: inexact exact v} that 
\begin{equation}\label{vk}
\mathbf{\tilde v}^{k} = \mathcal{S}\left(\mD^{\top} \mathbf{\tilde u}^{k} \right) + \mathbf{\mathbf{\tilde \epsilon}}^{k}.
\end{equation}
Since $\left\{\left(\mathbf{\tilde u}^k, \mathbf{\tilde v}^k\right)\right\}_{k=N}^{\infty}$ satisfies \eqref{iter: prox u ell0 inexact}, according to Lemma \ref{prop: prox_norm0 and projection}, we obtain that 
\begin{equation}\label{uk+1 prop5}
\mathbf{\tilde u}^{k+1} =\mT_{C}((1 - \alpha) \mathbf{\tilde u}^{k} + \alpha \mD \mathbf{\tilde v}^{k}).
\end{equation}
Substituting  \eqref{vk} into  the right hand side of  \eqref{uk+1 prop5} and using the linearity of $\mT_{C}$, we observe that
$$
\mathbf{\tilde u}^{k+1} = (1 - \alpha) \mT_{C}\mathbf{\tilde u}^k + \alpha \mT_{C}\mD \mathcal{S}\left(\mD^{\top} \mathbf{\tilde u}^{k} \right) + \alpha \mT_{C}\mD \mathbf{\mathbf{\tilde \epsilon}}^k.
$$
We complete the proof by calling the definition of $\mathcal{T}$. \qed
\end{proof}

The next proposition identifies the relation between a global minimizer of minimization problem \eqref{model: convex model} and a fixed-point of the operator $\mathcal{T}$. We first review the notion of subdifferential \cite{bauschke2011convex}. The subdifferential of a convex function $f:\mathbb{R}^d\to \bR \cup \left\{\infty\right\}$ at $\mathbf{x} \in \bR^d$ is defined by
$$
\partial f(\mathbf{x}) := \left\{\mathbf{y}: \mathbf{y}\in \bR^d \text{ and } f(\mathbf{z}) \geq f(\mathbf{x}) + \langle \mathbf{y}, \mathbf{z} - \mathbf{x} \rangle  \text{ for all } \mathbf{z} \in \bR^d\right\}.
$$
If $f$ is a convex function on $\bR^d$ and $\mathbf{x} \in \bR^d$, then $y \in \partial f(\mathbf{x})$ if and only if $\mathbf{x} = \mathrm{prox}_{f}(\mathbf{x}+\mathbf{y})$ (see, Proposition 2.6 of \cite{micchelli2011proximity}).
The proximity operator of the indicator function $\delta_C$ defined in \eqref{eq: indicator} may be expressed as the projection operator, that is,
\begin{equation}\label{eq: prox and indicator}
  \operatorname{prox}_{\delta_C}(\mathbf{u}) = \mathcal{P}_C(\mathbf{u}), \ \ \mbox{for all}\ \ \mathbf{u} \in \bR^n.
\end{equation}

\begin{proposition}\label{prop: fixed-point solution}
Suppose a pair $(\mathbf{u}^*, \mathbf{v}^*) \in \bR^{n} \times \bR^m$ is given. Then $(\mathbf{u}^*, \mathbf{v}^*)$ is a global minimizer of \eqref{model: convex model} if and only if $\mathbf{u}^*$ is a fixed-point of $\mathcal{T}$ and 
\begin{equation}\label{eq: prox psi inexact}
       \mathbf{v}^* = \mathcal{S}\left(\mD^{\top}  \mathbf{u}^{*} \right). 
\end{equation}
\end{proposition}
\begin{proof} This proof is similar to that of Theorem 2.1 in \cite{li2015multi}.   
From Proposition 16.8 and Corollary 16.48 of \cite{bauschke2011convex}, for $\mathbf{u} \in \bR^n,  \mathbf{v}\in \bR^m$, we have that
$$
\partial G(\mathbf{u}, \mathbf{v})  = \left(\partial \delta_{C}(\mathbf{u}) + \frac{\lambda}{\gamma}(\mathbf{u} - \mD \mathbf{v}), \nabla (\psi \circ \mB)(\mathbf{v}) - \frac{\lambda}{\gamma}(\mD^{\top}\mathbf{u} - \mathbf{v})\right).
$$
The Fermat rule (see, Theorem 16.3 of \cite{bauschke2011convex}) yields that $(\mathbf{u}^*, \mathbf{v}^*)$ is a global minimizer of \eqref{model: convex model} if and only if
\begin{align}
&0   \in  \partial   \delta_{C}(\mathbf{u}^*)+ \frac{\lambda}{\gamma} (\mathbf{u}^* - \mD \mathbf{v}^*),  \label{eq: partial delta inexact}\\
&0 =   \nabla \left(\psi\circ \mB \right) (\mathbf{v}^*) -  \frac{\lambda}{\gamma} (\mD^{\top} \mathbf{u}^* - \mathbf{v}^*). \label{eq : parial psi inexact}
\end{align}
From the definition of the indicator function $\delta_{C}$, we notice that 
$$
\partial \delta_{C}(\mathbf{u}^*)=b\partial \delta_{C}(\mathbf{u}^*), \ \ \mbox{for any positive constant}\ \ b.
$$
Hence, 
\eqref{eq: partial delta inexact} is equivalent to
\begin{equation}\label{1 alpha partial delta}
    -(\mathbf{u}^* - \mD \mathbf{v}^*) \in \frac{\gamma}{\lambda}\partial\delta_{C}(\mathbf{u}^*)=\frac{1}{\alpha}\partial\delta_{C}(\mathbf{u}^*)
\end{equation}
for the quantity $\alpha$ appearing in \eqref{iter: prox u ell0 inexact}.
We then rewrite \eqref{1 alpha partial delta} as
$
-\alpha(\mathbf{u}^* - \mD \mathbf{v}^*) \in \partial \delta_{C}(\mathbf{u}^*),
$
which is equivalent to $\mathbf{u}^* = \text{prox}_{\delta_{C}}\left((1 - \alpha)\mathbf{u}^* + \alpha \mD \mathbf{v}^*\right)$, according to Proposition 2.6 of \cite{micchelli2011proximity}.
It follows from equations \eqref{eq: prox and indicator}, \eqref{lemma: matrix = projection} and the linearity of $\mT_C$ that
\begin{equation}\label{eq: prox delta inexact}
    \mathbf{u}^* = (1 - \alpha) \mT_{C}\mathbf{u}^* + \alpha \mT_{C} \mD \mathbf{v}^*. 
\end{equation}
Simplifying the relation \eqref{eq : parial psi inexact} and solving for $\mathbf{v}^*$, we can obtain \eqref{eq: prox psi inexact}. We then substitute \eqref{eq: prox psi inexact}
into the right hand side of \eqref{eq: prox delta inexact}
and get that 
\begin{equation*}\label{eq: fixed-point of T}
    \mathbf{u}^* = (1 - \alpha)\mT_{C}\mathbf{u}^* + \alpha \mT_{C}\mD\mathcal{S}\left(\mD^{\top}  \mathbf{u}^{*} \right).
\end{equation*}
Thus, by the definition of $\mathcal{T}$, $\mathbf{u}^*$ is a fixed-point of  $\mathcal{T}$.  \qed
\end{proof}

According to Proposition \ref{prop: fixed-point solution}, obtaining a global minimizer of problem \eqref{model: convex model} is equivalent to finding a fixed-point of $\mathcal{T}$.   We turn to considering the fixed-point of $\mathcal{T}$.
Proposition \ref{prop: inexact Picard sequence} shows that the sequence $\left\{\mathbf{\tilde u}^k\right\}_{k = N}^{\infty}$ generated from iteration  \eqref{iter: Picard T inexact} is an inexact Picard sequence of $\mathcal{T}$. Convergence of inexact Picard iterations of a nonlinear operator, studied in \cite{combettes2004solving,liang2016convergence,rockafellar1976monotone}, may be guaranteed by the averaged nonexpansiveness of the nonlinear operator. 
Our next task is to show the operator $\mathcal{T}$ is averaged nonexpansive.

We now review the notion of the averaged nonexpansive operators which may be found in \cite{bauschke2011convex}.  An operator $\mathcal{P}:\bR^d\to\bR^d$ is said to be nonexpansive if
$$
\norm{\mathcal{P}(\mathbf{x}) - \mathcal{P}(\mathbf{y})}_2 \leq \norm{\mathbf{x} - \mathbf{y}}_2, \ \ \mbox{for all}\ \ \mathbf{x}, \mathbf{y} \in \bR^d.
$$
An operator $\mathcal{F}:\bR^d\to\bR^d$ is said to be averaged nonexpansive if there is $\kappa \in (0, 1)$  and  a nonexpansive operator $\mathcal{P}$ such that $\mathcal{F}= \kappa \mathcal{I} +  (1 - \kappa) \mathcal{P}$, where  $\mathcal{I}$ denotes the identity operator. We also say that $\mathcal{F}$ is $\kappa$-averaged nonexpansive. It is well-known that for any proper lower semi-continuous convex function $f:\mathbb{R}^d\to\mathbb{R} \cup \left\{\infty\right\}$, its proximity operator $\text{prox}_{f}$ is nonexpansive.

We now recall a general result, which concerns the convergence of the inexact Picard sequence
\begin{equation}\label{inexact Picard iteration}
 \mathbf{z}^{k+1} := \mathcal{F} \mathbf{z}^k + \mathbf{\epsilon}^{k+1},    
\end{equation}
for $\mathcal{F}:\bR^d\to \bR^d$. Its proof may be found in \cite{combettes2004solving,liang2016convergence}. See \cite{rockafellar1976monotone} for a special case.

\begin{lemma}\label{lemma: inexact Picard iteration}
Suppose that $\mathcal{F}:\bR^d\to \bR^d$ is averaged nonexpansive and has at least one fixed-point, $\left\{\mathbf{z}^{k}\right\}_{k=1}^{\infty}$ is the sequence defined by \eqref{inexact Picard iteration}, and $\left\{e^{k+1}\right\}_{k=1}^{\infty}$ is a positive sequence of numbers. If $\norm{\mathbf{\epsilon}^{k+1}}_2 \leq e^{k+1}$ for $k \in \mathbb{N}$ and $\sum_{k=1}^{\infty}e^{k+1} < \infty$, then $\left\{ \mathbf{z}^k\right\}_{k=1}^{\infty}$ converges to a fixed-point of $\mathcal{F}$.
\end{lemma}

We next show that the operator $\mathcal{T}$ defined by \eqref{eq: nonlinear map} is averaged nonexpansive. To this end, we make the following observation.

\begin{lemma}\label{lemma: nonexpansive of psi-B-D}
The operator $\mathcal{S} \circ \mD^{\top}$ is nonexpansive.   
\end{lemma}
\begin{proof}
Since $\frac{\gamma}{\lambda}\psi \circ \mB$ is continuously differentiable and convex, we have that 
$$
\mathcal{S}( \mD^{\top}\mathbf{u}) = \mathrm{prox}_{\frac{\gamma}{\lambda}\psi \circ \mB}\left(\mD^{\top} \mathbf{u}\right), \ \ \text{for all} \ \ \mathbf{u}\in \bR^m.
$$
It follows from the convexity of $\frac{\gamma}{\lambda}\psi \circ \mB$ that the operator $\mathrm{prox}_{\frac{\gamma}{\lambda}\psi \circ \mB}$ is nonexpansive. Therefore, the operator $\mathcal{S}$ is also nonexpansive. That is, for $\mathbf{u}_1, \mathbf{u}_2 \in \bR^n$,
$$
\norm{\mathcal{S}\left(\mD^{\top} \mathbf{u}_1\right) - \mathcal{S}\left(\mD^{\top} \mathbf{u}_2\right) }_2 \leq \norm{\mD^{\top} \mathbf{u}_1 - \mD^{\top} \mathbf{u}_2}_2. 
$$
Using the fact that $\norm{\mD^{\top}}_2 = 1$, we have that
$$
\norm{\mD^{\top}\mathbf{u}_1 - \mD^{\top}\mathbf{u}_2}_2 \leq \norm{\mathbf{u}_1 - \mathbf{u}_2}_2.
$$
The above two inequalities ensure that the operator $\mathcal{S} \circ \mD^{\top}$ is nonexpansive.
  \qed
\end{proof}

\begin{proposition}\label{prop: kappa-averaged nonexpansive}
The operator $\mathcal{T}$ defined by \eqref{eq: nonlinear map} is $\frac{1-\alpha}{2}$-averaged nonexpansive.
\end{proposition}
\begin{proof}


We write $\mathcal{T}= \frac{1-\alpha}{2} \mathcal{I}+ \frac{1+\alpha}{2}\mathcal{T}_1$ with  
$$
\mathcal{T}_1:=\frac{2}{1+\alpha}\left(\mathcal{T}-\frac{1-\alpha}{2} \mathcal{I}\right)
$$
and show that $\mathcal{T}_1$ is nonexpansive. The definition of $\mathcal{T}$ leads to  that 
\begin{equation}\label{T1u}
    \mathcal{T}_1\mathbf{u}=\frac{2}{1+\alpha}\left(\frac{1-\alpha}{2}\left(2\mT_C - \mI\right) \mathbf{u}+\alpha\mathbf{T}_C\mathbf{D}\mathcal{S}(\mathbf{D}^\top \mathbf{u})\right), \ \ \mbox{for} \ \ \mathbf{u}\in\mathbb{R}^n.
\end{equation}
Then, for any $\mathbf{u}_1, \mathbf{u}_2\in\mathbb{R}^n$, it follows from the triangle inequality that 
\begin{align}
    \left\|\mathcal{T}_1 \mathbf{u}_1-\mathcal{T}_1 \mathbf{u}_2\right\|_2 
    \leq \frac{2}{1+\alpha}\Big(& \frac{1-\alpha}{2}\left\|2\mathbf{T}_C-\mathbf{I}\right\|_2\left\|\mathbf{u}_1-\mathbf{u}_2\right\|_2  + \nonumber\\
    &\alpha\left\|\mathbf{T}_C\right\|_2\left\|\mathbf{D}\right\|_2\left\|\mathcal{S}\left(\mD^{\top} \mathbf{u}_1\right) - \mathcal{S}\left(\mD^{\top} \mathbf{u}_2\right)\right\|_2\Big).\label{triangle ineq} 
\end{align}
Since the matrix $\mathbf{T}_C$ is diagonal with diagonal entries being either $1$ or $0$, we find that the matrix $2\mathbf{T}_C- \mathbf{I}$ is also diagonal with diagonal entries being either $1$ or $-1$. Therefore, 
\begin{equation}\label{norm 2TC - I}
\norm{2\mT_C - \mI}_2 = 1.    
\end{equation}
Noting that $\norm{\mT_C}_2 \leq 1$ and $\norm{\mD}_2 = 1$, we obtain from \eqref{triangle ineq}-\eqref{norm 2TC - I} and Lemma \ref{lemma: nonexpansive of psi-B-D} that
\begin{align*}
    \|\mathcal{T}_1\mathbf{u}_1-\mathcal{T}_1 \mathbf{u}_2\|_2\leq \frac{2}{1+\alpha} &\left(\frac{1-\alpha}{2}\norm{\mathbf{u}_1 - \mathbf{u}_2}_2 +\alpha \norm{\mathbf{u}_1 - \mathbf{u}_2}_2\right)=\norm{\mathbf{u}_1 - \mathbf{u}_2}_2
\end{align*}
which proves that $\mathcal{T}_1$ is nonexpansive. \qed

\end{proof}

We next establish convergence of the inexact fixed-point iterations \eqref{iter: prox u ell0 inexact}-\eqref{iter: prox v ell0 inexact} satisfying condition \eqref{condition for v inexact general}.

\begin{proposition}\label{prop: ineact ell1}
Suppose that $\left\{\mathbf{\tilde \epsilon}^{k+1}\right\}_{k=1}^{\infty}$ is a predetermined sequence in $\mathbb{R}^m$, $\left\{\left(\mathbf{\tilde u}^k, \mathbf{\tilde v}^k\right)\right\}_{k = 1}^{\infty}$ is the sequence generated by \eqref{iter: prox u ell0 inexact}-\eqref{iter: prox v ell0 inexact} satisfying condition \eqref{condition for v inexact general} with $\rho' \in \left(0, \frac{\lambda}{\gamma}\frac{1 - \alpha}{\alpha}\right)$, and the operator $\mathcal{T}$ has at least one fixed-point. Let $\left\{e^{k+1}\right\}_{k=1}^{\infty}$ be a positive sequence of numbers. If $\norm{\mathbf{ \tilde{\epsilon}}^{k+1}}_2 \leq e^{k+1}$ for $k \in \mathbb{N}$ and $\sum_{k= 1}^{\infty} e^{k+1} < \infty$, then $\left\{\left(\mathbf{\tilde u}^k, \mathbf{\tilde v}^k\right)\right\}_{k =  1}^{\infty}$ converges to a local minimizer $(\mathbf{u}^*, \mathbf{v}^*)$ of optimization problem \eqref{model: l0} and $S(\mathbf{u}^*) = C$.
\end{proposition}
\begin{proof}
It suffices to show the convergence of $\left\{\left(\mathbf{\tilde u}^k, \mathbf{\tilde v}^k\right)\right\}_{k = N}^{\infty}$, where $N$ is determined in Proposition \ref{prop: explicit unchange support prox inexact}. Proposition  \ref{prop: inexact Picard sequence} recasts the sequence $\left\{\mathbf{\tilde u}^k\right\}_{k = N}^{\infty}$ as an inexact Picard sequence \eqref{iter: Picard T inexact} of $\mathcal{T}$.  By Proposition \ref{prop: kappa-averaged nonexpansive}, $\mathcal{T}$ is $\frac{1-\alpha}{2}$-averaged nonexpansive. Since $\alpha \in (0, 1)$, $\left\|\mT_{C}\right\|_2 \leq 1$, and $\left\|\mD\right\|_2 = 1$, we obtain that $\left\|\alpha \mT_C\mD \mathbf{\tilde \epsilon}^k \right\|_2 \leq \left\|\mathbf{\tilde \epsilon}^k \right\|_2$. Noting that the operator $\mathcal{T}$ has at least one fixed-point, Lemma \ref{lemma: inexact Picard iteration} ensures that $\left\{\mathbf{\tilde u}^k\right\}_{k =  N}^{\infty}$ converges to a fixed-point $\mathbf{u}^*$ of $\mathcal{T}$, that is, 
\begin{equation}\label{convergence for u}
    \lim_{k\to \infty}\left\|\mathbf{\tilde u}^k - \mathbf{u}^*\right\|_2 = 0.
\end{equation}
We next prove that $\left\{\mathbf{\tilde v}^k\right\}_{k=N}^{\infty}$ converges to $\mathcal{S}\left(\mD^{\top}  \mathbf{u}^{*} \right)$. Notice that by the triangle inequality, for each $k\in\mathbb{N}$, 
\begin{align*}
    \left\|\mathbf{\tilde v}^k-\mathcal{S}\left(\mD^{\top}  \mathbf{u}^{*}\right)\right\|_2 \leq \left\|\mathbf{\tilde v}^k-\mathcal{S}\left(\mD^{\top}  \mathbf{\tilde u}^{k} \right)\right\|_2 + \left\|\mathcal{S}\left(\mD^{\top} \mathbf{\tilde u}^{k} \right)-\mathcal{S}\left(\mD^{\top}  \mathbf{u}^{*} \right)\right\|_2.
\end{align*}
Combining definition \eqref{iter: prox v ell0 inexact} of $\mathbf{\tilde v}^k$, Lemma \ref{lemma: nonexpansive of psi-B-D}, and the above inequality, we have that 
\begin{equation}\label{eq: vk DTU* two part}
    \left\|\mathbf{\tilde v}^k-\mathcal{S}\left(\mD^{\top}  \mathbf{u}^{*} \right)\right\|_2\leq\|\tilde\epsilon^k\|_2+\left\|\mathbf{\tilde u}^k-\mathbf{u}^*\right\|_2.
\end{equation}
According to the assumption that $\norm{\mathbf{\tilde \epsilon}^{k+1}}_2 \leq e^{k+1}$ for $k \in \mathbb{N}$ and $\sum_{k= 1}^{\infty} e^{k+1} < \infty$, we obtain that $\lim_{k\to\infty}\left\|\tilde\epsilon^k\right\|_2=0$.
This together with \eqref{convergence for u} and \eqref{eq: vk DTU* two part} yields that
$$
\lim_{k\to\infty}\left\|\mathbf{\tilde v}^k- \mathcal{S}\left(\mD^{\top}  \mathbf{u}^{*} \right)\right\|_2=0.
$$
That is,  $\left\{\mathbf{\tilde v}^k\right\}_{k=N}^{\infty}$ converges to $\mathcal{S}\left(\mD^{\top}  \mathbf{u}^{*} \right)$.

According to Proposition \ref{prop: fixed-point solution}, $(\mathbf{u}^*, \mathbf{v}^*)$ is a global minimizer of optimization problem \eqref{model: convex model}. Proposition \ref{prop: explicit unchange support prox inexact} shows that $S(\mathbf{\tilde u}^k) = C$, for all $k  \geq N$. Hence, $S( \mathbf{u}^*) = C$. This combined with Proposition \ref{prop: global minimizer and local minimizer} guarantees that $(\mathbf{u}^*, \mathbf{v}^*)$ is a local minimizer of optimization problem \eqref{model: l0}.  \qed
\end{proof}

Proposition \ref{prop: ineact ell1} requires the condition $\norm{\mathbf{\tilde \epsilon}^{k+1}}_2=\|\mathbf{\tilde v}^{k+1}-\mathbf{\tilde v}^{k+1}_*\|_2 \leq e^{k+1}$, which is difficult to verify in general. This condition may be replaced by \eqref{condition for v inexact general-2} due to the known estimate (see,  Proposition 3 of \cite{rockafellar1976monotone})
\begin{equation}\label{eq: upper bdd of v v* differentiable}
\norm{\mathbf{\tilde v}^{k+1} - \mathbf{\tilde v}^{k+1}_*}_2 \leq \frac{ \gamma}{\lambda}  \norm{\nabla H(\mathbf{\tilde v}^{k+1}; \mathbf{\tilde u}^{k+1} )}_2, \ \ \mbox{for all}\ \ k \in \mathbb{N}.
\end{equation}

{ Finally, we present the proof for Theorem \ref{thm : convergence algo INFPPA}}.

\begin{proof} [\textbf{Proof of Theorem \ref{thm : convergence algo INFPPA}}]
Note the sequence $\left\{(\mathbf{\tilde u}^{k}, \mathbf{\tilde v}^{k})\right\}_{k=1}^{\infty}$ generated from Algorithm \ref{algo: inexact FPPA l0} is a sequence defined by \eqref{iter: prox u ell0 inexact}-\eqref{iter: prox v ell0 inexact} satisfying conditions \eqref{condition for v inexact general} and \eqref{condition for v inexact general-2}.  The condition \eqref{condition for v inexact general-2} together with inequality \eqref{eq: upper bdd of v v* differentiable} implies 
$$
\norm{\mathbf{\tilde \epsilon}^{k+1}}_2 \leq \frac{\gamma}{\lambda} e^{k+1},\ \ \mbox{for all}\ \ k \in \mathbb{N}.
$$
Noting that $\lambda, \gamma$ are positive constants,  we have that 
$$
\sum_{k=1}^{\infty}\frac{ \gamma}{\lambda}e^{k+1} < \infty.
$$
Therefore, all conditions of Proposition \ref{prop: ineact ell1} are satisfied. Hence, $\left\{\left(\mathbf{\tilde u}^k, \mathbf{\tilde v}^k\right)\right\}_{k =  1}^{\infty}$ converges to a local minimizer of optimization problem \eqref{model: l0} and $S(\mathbf{u}^*) = C$.
\qed
\end{proof}

\section{Numerical Experiments}\label{section: Experiments}

In this section, we present numerical experiments to demonstrate the effectiveness of the proposed inexact algorithm for solving the $\ell_0$ models for three applications including regression, classification,  and image deblurring. 
The numerical results verify the convergence of Algorithm \ref{algo: inexact FPPA l0} and explore the impact of the inner iterations within the algorithm.
Furthermore, the numerical results 
validate the superiority of the solution obtained from the $\ell_0$ model \eqref{model: l0} solved by the proposed algorithm when compared to the solution from the $\ell_1$ model \eqref{model: l1}, across all three examples.

All the experiments are performed on the First Gen ODU HPC Cluster with Matlab. The computing jobs are randomly placed on an X86\_64 server with the computer nodes Intel(R) Xeon(R) CPU E5-2660 0 @ 2.20GHz (16 slots), Intel(R) Xeon(R) CPU E5-2660 v2 @ 2.20GHz (20 slots), Intel(R) Xeon(R) CPU E5-2670 v2 @ 2.50GHz (20 slots), Intel(R) Xeon(R) CPU E5-2683 v4 @ 2.10GHz (32 slots).

For the regression and classification models, we need a kernel as described in \cite{li2018fixed,li2019two}. In the two examples to be presented, we choose the kernel as
\begin{equation}\label{Gaussian_Kernel}
K(\mathbf{x}, \mathbf{y}):=\mathrm{exp}\left(-\|\mathbf{x}-\mathbf{y}\|_2^2/(2\sigma^2)\right),
\ \ \mathbf{x}, \mathbf{y}\in \mathbb{R}^d,
\end{equation}
where $\sigma$ will be specified later. 

We will use Algorithm \ref{algo: inexact FPPA l0} to solve model \eqref{model: l0} for the three examples.
The model \eqref{model: l1} has a convex objective function, and hence, it can be efficiently solved by many optimization algorithms. Here, we employ Algorithms \ref{algo: exact FPPA l1} and  \ref{algo: Inexact FPPA l1} (described below) to solve \eqref{model: l1} with $\mD$ being the identity matrix and $\mD$ being the tight framelet matrix \cite{shen2016wavelet} (or the first order difference matrix \cite{micchelli2011proximity}), respectively. 
\begin{algorithm}[hbt!]
  \caption{Fixed-point proximity scheme for problem \eqref{model: l1} with $\mD$ being identity matrix.}
  
  \label{algo: exact FPPA l1}
  
  \KwInput{ The functions $\psi$, matrix $\mB$ in problem \eqref{model: l1}, $\varphi(\mathbf{v}):= \lambda \norm{\mathbf{v}}_1$; parameters $p>0, q>0$ satisfying $pq > \norm{\mB}_2^2$  }
  
  \KwInitialization{$ \mathbf{v}_1, \mathbf{w}_1$}

   \For{l = 1, 2, \ldots} 
   {
   		$\mathbf{v}_{l+1} = \text{prox}_{\frac{1}{p}\varphi}\left(\mathbf{v}_{l} -  \frac{1}{p}\mB^{\top}\mathbf{w}_{l}\right)$

   		$\mathbf{w}_{l+1} = \frac{1}{q}\left(\mI-\left(\mI + q\nabla \psi\right)^{-1}\right)\left(q\mathbf{w}_{l} + \mB \left(2\mathbf{v}_{l+1}-\mathbf{v}_{l}\right)\right)$
}	
\KwOutput{$\left\{\left( \mathbf{v}_{l}, \mathbf{w}_l\right)\right\}_{l=1}^{\infty}$.}
\end{algorithm}
\begin{algorithm}[hbt!]
  \caption{Inexact fixed-point proximity scheme for problem \eqref{model: l1}}
  
  \label{algo: Inexact FPPA l1}
  
  \KwInput{ The functions $\psi$, matrix $\mD, \mB$ in problem \eqref{model: l1} and $\varphi(\mD \mathbf{v}) = \lambda \norm{\mD \mathbf{v}}_1$; parameters $p_1>0, q_1>0$ satisfying $p_1q_1 > \norm{\mD}_2^2$, $p_2>0, q_2>0$ satisfying $p_2q_2 > \norm{\mB}_2^2$, positive sequence $\left\{e^{k+1}\right\}_{k=1}^{\infty}$ }
  
  \KwInitialization{$ \mathbf{\tilde v}^1, \mathbf{\tilde z}^1, \mathbf{\tilde w}^1$}

   \For{k =1, 2, \ldots} 
   {

   $T(\mathbf{v}; \mathbf{\tilde v}^k, \mathbf{\tilde w}^k) := \frac{p_1}{2}\norm{\mathbf{v} - \left(\mathbf{\tilde v}^k - \frac{1}{p_1}\mD^{\top}\mathbf{\tilde w}^k\right)}_{2}^2 + (\psi \circ \mB)(\mathbf{v})$

   $\mathbf{\tilde z}^{k+1}_1 = \mathbf{\tilde z}^k, l=1$
   
   \Do{$\norm{\nabla T(\mathbf{\tilde v}^{k+1}_l; \mathbf{\tilde v}^k, \mathbf{\tilde w}^k)}_2 > e^{k+1}$}
   {
   
   $\mathbf{\tilde v}^{k+1}_{l+1} = \frac{p_1}{p_1+p_2}\left(\mathbf{\tilde v}^k - \frac{1}{p_1}\mD^{\top}\mathbf{\tilde w}^k\right) + \frac{p_2}{p_1+p_2}\left(\mathbf{\tilde v}^{k+1}_l - \frac{1}{p_2}\mB^{\top}\mathbf{\tilde z}^{k+1}_l\right)$

   $\mathbf{\tilde z}^{k+1}_{l+1} = \frac{1}{q_2}\left(\mI-\left(\mI + q_2 \nabla \psi\right)^{-1}\right)\left(q_2\mathbf{\tilde z}^{k+1}_l + \mB(2\mathbf{\tilde v}^{k+1}_{l+1} - \mathbf{\tilde v}^{k+1}_l)\right)$

   $l:=l+1$
   }
   }
     
     $\mathbf{\tilde v}^{k+1}= \mathbf{\tilde v}^{k+1}_l, \mathbf{\tilde z}^{k+1} = \mathbf{\tilde z}^{k+1}_l$

     $\mathbf{\tilde w}^{k+1} = \frac{1}{q_1}\left(\mI-\text{prox}_{q_1\varphi}\right)\left(q_1\mathbf{\tilde w}^{k} + \mD \left(2\mathbf{\tilde v}^{k+1}-\mathbf{\tilde v}^{k}\right)\right)$
	
\KwOutput{$\left\{\left(\mathbf{\tilde v}^{k}, \mathbf{\tilde w}^k\right)\right\}_{k=1}^{\infty}$.}
\end{algorithm}
Note that Algorithm \ref{algo: exact FPPA l1} was proposed in \cite{li2015multi,micchelli2011proximity} and  Algorithm \ref{algo: Inexact FPPA l1} was introduced in \cite{ren2023}. According to \cite{ren2023}, the sequence $(\mathbf{\tilde v}^{k}, \mathbf{\tilde w}^{k})$ generated from Algorithm \ref{algo: Inexact FPPA l1} converges to a global minimizer of problem \eqref{model: l1} if the condition $\sum_{k=1}^{\infty}e^{k+1} < \infty$ is satisfied.

We choose the parameter  in Algorithm \ref{algo: inexact FPPA l0} \begin{equation}\label{regular rho '}
\rho':=0.99\cdot(\lambda/\gamma)(1/\alpha-1)    
\end{equation}
with $\alpha, \lambda$, and $\gamma$ to be specified. { The parameter $\lambda$ serves as the regularization parameter in models \eqref{model: l0} and \eqref{model: l1}. If this parameter is set too large, it can lead to a decrease in prediction accuracy. Conversely, if it is set too small, the resulting solution may lack sparsity, potentially leading to overfitting. Thus, identifying an appropriate range for $\lambda$ is crucial. Similarly, the parameter $p$ in Algorithm \ref{algo: inexact FPPA l0} and Algorithm \ref{algo: exact FPPA l1}, where $\frac{1}{p}$ serves as the step size, affects the speed of the algorithm. It is crucial to choose an appropriate step size in both algorithms for optimal performance.}

{ The impact of the choice of the parameter $\alpha$ in the exact fixed-point proximity algorithm \eqref{iter: prox u ell0}-\eqref{iter: prox v ell0} was numerically investigated in \cite{shen2016wavelet}, revealing its influence on the convergence speed of the algorithm. 
We will also explore the effect of the parameter $\alpha$ for the inexact fixed-point proximity algorithm (Algorithm \ref{algo: inexact FPPA l0}).  For the regression, classification, and Gaussian noise image deblurring problems, we test Algorithm \ref{algo: inexact FPPA l0} with different values of $\alpha$, even those with which the algorithm is not guaranteed to converge ($\alpha\geq 1$).
In the instances where the value of $\alpha$ is greater than or equal to $1$, we set the corresponding $\rho'$ to be $0$.}
Moreover, to explore the impact of the inner iteration in Algorithm \ref{algo: inexact FPPA l0}, we design its two variants. In the first variant, the inner iteration is performed for only one iteration. In this variant, the generated sequence may not satisfy the conditions \eqref{condition for v inexact general} and \eqref{condition for v inexact general-2}. We will use several sets of parameters to explore the convergence of this case. In the second variant, we increase the value of $\rho'$ of Algorithm \ref{algo: inexact FPPA l0} in a way that it fails to satisfy the condition of Theorem \ref{thm : convergence algo INFPPA}. We will use a single set of parameters $\lambda,\gamma$ and only increase the value of $\rho'$, testing it at 10, 100, and 1000 times the choice \eqref{regular rho '}.

We adopt the following stopping criterion for all iteration algorithms tested in this study.
For a sequence $\left\{\mathbf{x}^k\right\}_{k=1}^{\infty}$ generated by an algorithm performed in this section, iterations are terminated if 
\begin{equation}\label{eq: stop}
    \frac{\norm{\mathbf{x}^{k+1} - \mathbf{x}^k}_2}{\norm{\mathbf{x}^{k+1}}_2} < \mathrm{TOL},
\end{equation}
where $\mathrm{TOL}$ denotes a prescribed tolerance value.

\subsection{Regression}

In this subsection, we consider solving the regression problem by using the $\ell_0$ model. Specifically, we explore the convergence behavior of Algorithm \ref{algo: inexact FPPA l0} and the impact of its inner iterations and compare the numerical performance of the $\ell_0$ model with that of the $\ell_1$ model. 

In our numerical experiment, we use the benchmark dataset  ``Mg"  \cite{chang2011libsvm} with 1,385 instances, each of which has 6 features. Among the instances, we take $m:=1,000$ of them as training samples and the remanding 385 instances as testing samples. 

{ 
We describe the setting of regression as follows. Our goal is to obtain a function that regresses the data set $\left\{(\mathbf{x}_j, y_j) : j\in \mathbb{N}_m\right\} \subset \mathbb{R}^6 \times \mathbb{R}$ by using the kernel method detailed in \cite{li2018fixed,li2019two}, where $\mathbf{x}_j$ is the input features and $y_j$ is the corresponding label. 
In this experiment, we choose the kernel $K$ as defined in \eqref{Gaussian_Kernel} with $\sigma:= \sqrt{10}$ and $d:=6$. The regression function $f: \bR^6 \to \bR$ is then defined as
\begin{equation}\label{regression prediction function}
f(\mathbf{x}) := \sum_{ i\in \mathbb{N}_m}v_j K(\mathbf{x}_j, \mathbf{x}),\quad\mathbf{x}\in\mathbb{R}^6.
\end{equation}
The coefficient vector $\mathbf{v} \in \bR^m$ that appears in \eqref{regression prediction function} is learned by solving the optimization problem
\begin{equation}\label{regression optimization problem}
    \text{argmin} \left\{ 
  L_\mathbf{y}(\mathbf{v}) + \frac{\lambda}{2\gamma}\norm{\mathbf{u} - \mathbf{v}}_2^2 + \lambda \norm{\mathbf{u}}_0: (\mathbf{u}, \mathbf{v} 
 ) \in \bR^m \times \bR^m\right\},
\end{equation}
where $\lambda, \gamma$ are positive constants and the fidelity term $L: \bR^m \to \bR$ is defined as
$$
 L_\mathbf{y}(\mathbf{v}) := \frac{1}{2}\sum_{k \in \mathbb{N}_m}\left(\sum_{j \in \mathbb{N}_m}v_j K(\mathbf{x}_j, \mathbf{x}_k)  - y_k\right)^2, \quad \mathbf{v} \in \bR^m.
$$
Optimization problem \eqref{regression optimization problem} may be reformulated in the form of \eqref{model: l0} by choosing
\begin{equation}\label{squared loss svm}
\psi(\mathbf{z}):=\frac{1}{2}\norm{\mathbf{z}- \mathbf{y}}^2, \ \ \mathbf{z}\in\mathbb{R}^m,
\end{equation}
$\mathbf{B}:=\mathbf{K}:=[K(\mathbf{x}_j, \mathbf{x}_k):j,k\in\mathbb{N}_m]$
and $\mathbf{D}:=\mathbf{I}_{m}$.
}


The $\ell_0$ model \eqref{model: l0} and the $\ell_1$ model \eqref{model: l1} are solved by 
Algorithm \ref{algo: inexact FPPA l0} and Algorithm \ref{algo: exact FPPA l1},  respectively. { For both of these algorithms, the parameters $p$ and $\lambda$ are varied within the range of $[10^{-1}, 10^3]$ and $[10^{-7}, 10^{-2}]$, respectively, and
$q: =(1+10^{-6})\norm{\mathbf{B}}_2^2/p$. In addition, for Algorithm \ref{algo: inexact FPPA l0}, the parameter $\gamma$ is varied within the range $[10^{-7}, 10^{-3}]$ and $e^{k+1} := M/k^2$ with $M:=10^{16}$.} 
{The stopping criterion for Algorithm \ref{algo: inexact FPPA l0} (resp.  Algorithm  \ref{algo: exact FPPA l1}) is either when \eqref{eq: stop} is satisfied with $\mathbf{x}^k:= \mathbf{\tilde u}^k$ (resp.  $\mathbf{x}^k:= \mathbf{v}^k$) and $\mathrm{TOL}:=10^{-6}$, or when the maximum iteration count of $10^{5}$ is reached.}


{ 

We first explore the effect of the parameter $\alpha$ in convergence of Algorithm \ref{algo: inexact FPPA l0} by
testing different values of $\alpha$. Values of the objective function of model \eqref{model: l0} against the number of iterations for the outer loop of  Algorithm \ref{algo: inexact FPPA l0} with different values of $\alpha$ are displayed in Figure \ref{fig: L0 convergence divergence}. 
We observe from Figure \ref{fig:regression_case_1_convergence_show} that for $\alpha \in (0, 2]$, Algorithm \ref{algo: inexact FPPA l0} with a larger value of $\alpha$ trends to produce a smaller value of the objective function. This suggests that for $\alpha \in (0, 2]$, increasing the value of $\alpha$ can accelerate the speed of convergence when implementing Algorithm \ref{algo: inexact FPPA l0}. However, for $\alpha > 2$, as shown in Figure \ref{fig:regression_case_1_divergence_show}, Algorithm \ref{algo: inexact FPPA l0} diverges. In the remaining part of this example, we choose $\alpha:=0.99$ so that the hypothesis of Theorem \ref{thm : convergence algo INFPPA} on the range of $\alpha$ is satisfied, which guarantees convergence of Algorithm \ref{algo: inexact FPPA l0}.

}

We then apply Algorithm \ref{algo: inexact FPPA l0} and its two variants to test their convergence for the $\ell_0$ model. We use four sets of parameters for Algorithm \ref{algo: inexact FPPA l0} and the first variant, and a single set of parameters for the second variant. Figure \ref{fig: regression algorithm 1}, \ref{fig: regression algorithm one inner step}, and \ref{fig: regression algorithm extend inner condition}  display values of the objective function generated from Algorithm \ref{algo: inexact FPPA l0}, the first variant, and the second variant, respectively. Figure \ref{fig: regression algorithm 1} confirms convergence of Algorithm \ref{algo: inexact FPPA l0} and Figure \ref{fig: regression algorithm one inner step} reveals divergence of the first variant. From Figure \ref{fig: regression algorithm extend inner condition}, we observe that a slight increase in $\rho'$ maintains a convergence trend in the objective function value, but a significant increase, for example, changing it to $1000 \times \rho'$, leads to divergence.  These results demonstrate that the inclusion of the inner iteration is essential for effectively solving the $\ell_0$ model. One conceivable explanation for this phenomenon could be that in the regression problem, the norm of $\mB$ in the objective function \eqref{target function in model: l0} is as large as 863, and thus even slight variations in the input variable can have a significant impact on the resulting objective function value. Therefore, the inner iteration is critical for controlling the error for each inexact step.


\begin{figure}
      \centering
	   \begin{subfigure}{0.45\linewidth}
                \includegraphics[width=\linewidth]{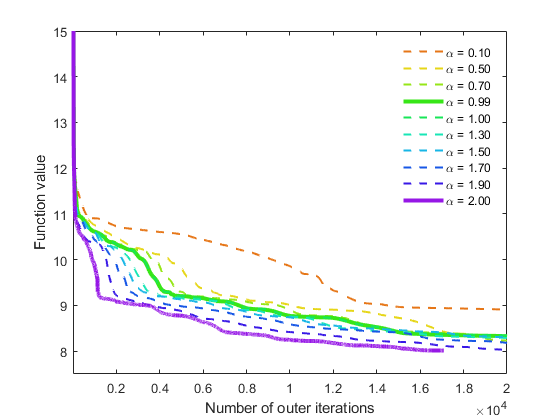}
            	\caption{}
    	       \label{fig:regression_case_1_convergence_show}
	   \end{subfigure}
    \hfill
	   \begin{subfigure}{0.45\linewidth}
        		\includegraphics[width=\linewidth]{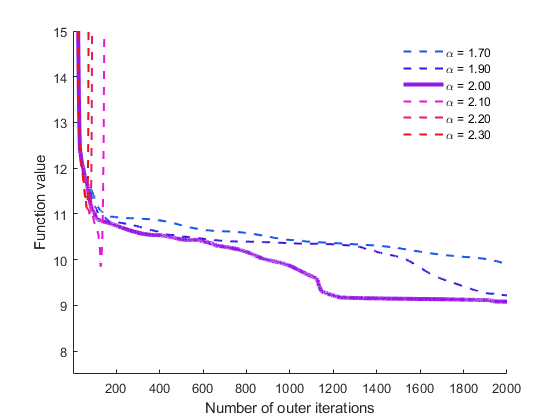}
        		\caption{}
        		\label{fig:regression_case_1_divergence_show}
	    \end{subfigure}
	
	\caption{
 The effect of choices of parameter $\alpha$ on convergence of Algorithm \ref{algo: inexact FPPA l0} for regression.}
	\label{fig: L0 convergence divergence}
\end{figure}

\begin{figure}
      \centering
	   \begin{subfigure}{0.45\linewidth}
                \includegraphics[width=\linewidth]{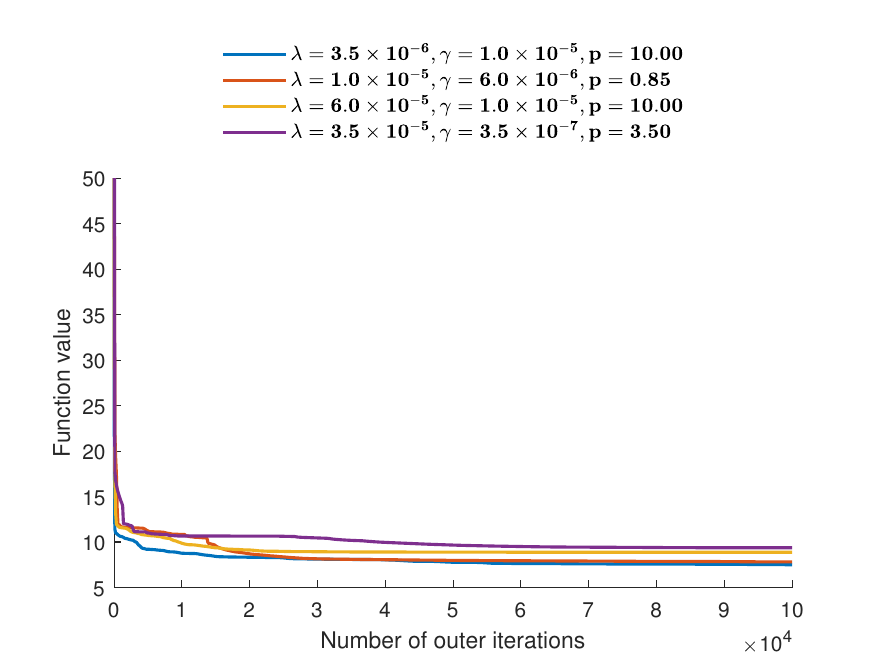}
            	\caption{}
    		  \label{fig: regression algorithm 1}
	   \end{subfigure}
    \hfill
	   \begin{subfigure}{0.45\linewidth}
        		\includegraphics[width=\linewidth]{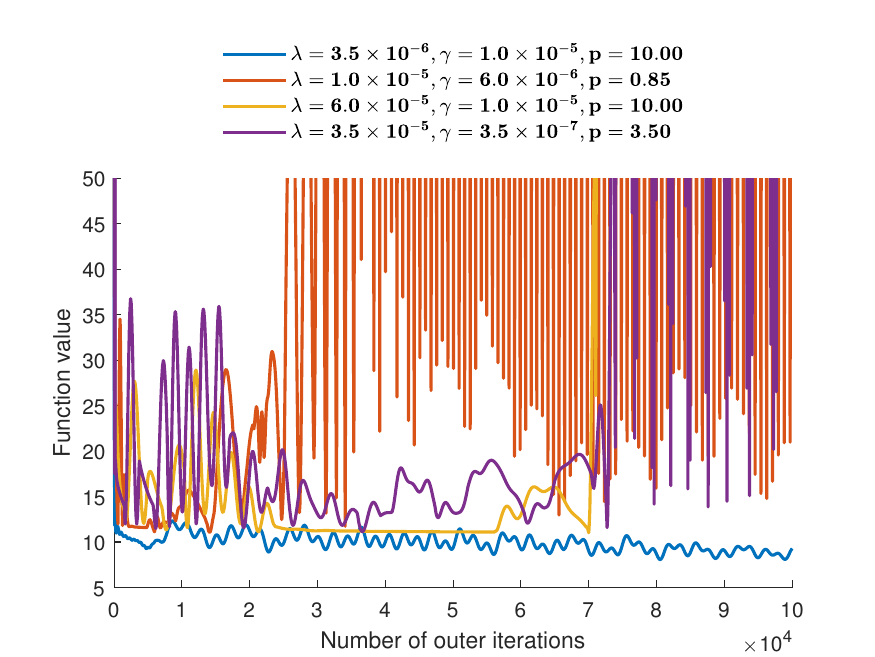}
        		\caption{}
        		\label{fig: regression algorithm one inner step}
	    \end{subfigure}
	\vfill
	     \begin{subfigure}{0.45\linewidth}
		 \includegraphics[width=\linewidth]{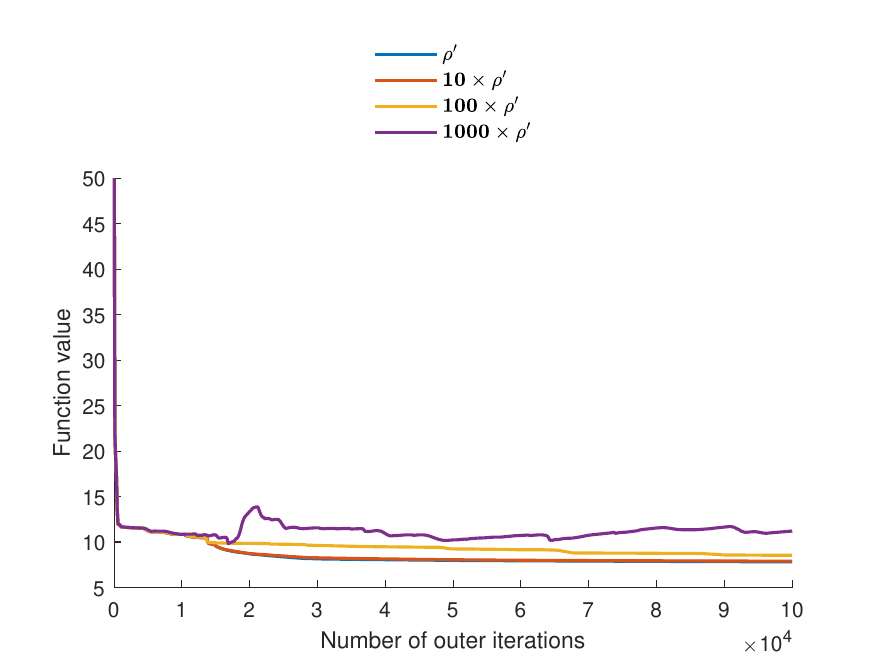}
		 \caption{}
		 \label{fig: regression algorithm extend inner condition}
	      \end{subfigure}
	\caption{
 Convergence of Algorithm \ref{algo: inexact FPPA l0} vs its  first and second variants for regression: (a), (b), and (c) display objective function values generated from Algorithm \ref{algo: inexact FPPA l0}, the first variant, and the second variant, respectively. In (c), we set $\lambda, \gamma$, and $p$  
 the values $1.0\times 10^{-5}, 6.0\times 10^{-6} $, and $0.85$, respectively.}
	\label{fig: L0 regression convergence show}
\end{figure}

In Table \ref{tab: regression same mse}, we compare the number of nonzero (NNZ) components of the solutions obtained by the $\ell_0$ model and the $\ell_1$ model, when they have comparable testing accuracy. It is clear from Table \ref{tab: regression same mse} that when the mean squared error on the testing dataset ($\mathrm{TeMSE}$) falls within varying intervals, the $\ell_0$ model consistently generates solutions sparser than those generated by the $\ell_1$ model, with comparable mean squared errors on the training dataset ($\mathrm{TrMSE}$).

\begin{table}
    \caption{Comparison of the $\ell_0$ model and the $\ell_1$ model for regression: The NNZ and TrMSE results  within fixed TeMSE intervals.} 
    \small
    \centering
    \begin{tabular}{c|cc|cc}
    \hline
     & \multicolumn{2}{c|}{ NNZ} &  \multicolumn{2}{c}{ TrMSE $(\times10^{-2})$}
    \\
    \hline 
     TeMSE $(\times10^{-2})$ & $\ell_0$ model & $\ell_1$ model & $\ell_0$ model & $\ell_1$ model \\\hline
    $[1.60, 1.65)$ & $\mathbf{491}$ & 558  & $\mathbf{1.52}$ & $1.54$ \\ 
    $[1.65, 1.70)$ & $\mathbf{65}$  & 240  & $\mathbf{1.58}$ & $1.67$ \\     
    $[1.70, 1.75)$ & $\mathbf{16}$  & 53   & $\mathbf{1.80}$ & $1.79$ \\
    $[1.75, 1.80)$ & $\mathbf{11}$  & 43   & $\mathbf{1.88}$ & $1.84$ \\
    $[1.80, 1.85)$ & $\mathbf{10}$  & 43   & $\mathbf{1.88}$ & $1.91$ \\
        $[1.85, 1.90)$ & $\mathbf{9}$   & 36   & $\mathbf{1.93}$ & $1.95$
    \\\hline
    \end{tabular}
    \label{tab: regression same mse}
\end{table}

\subsection{Classification}

In this subsection, we consider solving the classification problem by using the $\ell_0$ model. Specifically, we investigate the convergence behavior and the impact of inner iterations on Algorithm \ref{algo: inexact FPPA l0} and compare the numerical performance of the $\ell_0$ model with that of the $\ell_1$ model. 

In our numerical experiment, we use the handwriting digits from the MNIST database \cite{lecun1998gradient}, which is composed of $60,000$ training samples and $10,000$ testing samples of the digits $``0"$ through $``9"$. We consider classifying two handwriting digits ``7" and ``9" taken from MNIST. Handwriting digits ``7" and ``9" are often easy to cause obfuscation in comparison to other pairs of digits. We take $m:=5,000$ instances as training samples and $2,000$ as testing samples of these two digits from the dataset. Specifically, we consider training data set $\left\{(\mathbf{x}_j, y_j) : j\in \mathbb{N}_m\right\} \subset \mathbb{R}^d \times \left\{-1, 1\right\}$, where $\mathbf{x}_j$ are images of the digit $7$ or $9$. 

{ 
The setting of classification is described as follows.
Our goal is to find a function that separates the training data into two groups with label $y_j := 1$ (for digit 7) and $y_j:=-1$ (for digit 9). We employ the kernel method to derive such a function, as detailed in \cite{Xu2023Sparse}. We choose the kernel $K$ as defined in \eqref{Gaussian_Kernel} with $\sigma:=4$ and $d:=784$. The classification function $f: \bR^{784} \to \left\{-1, 1\right\}$  is then defined as
\begin{equation}\label{classification prediction function}
f(\mathbf{x}) := \text{sign}\left(\sum_{ i\in \mathbb{N}_m}v_j K(\mathbf{x}_j, \mathbf{x})\right),\ \mathbf{x}\in\mathbb{R}^{784}, 
\end{equation}
where $\text{sign}$ denotes the sign function,  assigning $-1$ for negative inputs and $1$ for non-negative inputs.
The coefficient vector $\mathbf{v} \in \bR^m$ that appears in \eqref{classification prediction function} is learned by solving the optimization
\begin{equation}\label{classification optimization problem}
    \text{argmin} \left\{L_\mathbf{y}(\mathbf{v}) + \frac{\lambda}{2\gamma}\norm{\mathbf{u} - \mathbf{v}}_2^2 + \lambda\norm{u}_0: (\mathbf{u}, \mathbf{v}) \in \bR^{n} \times \bR^{m}\right\},
\end{equation}
where $\lambda, \gamma$ are positive constants and the fidelity term $L: \bR^m \to \bR$ is defined as
$$
L_\mathbf{y}(\mathbf{v}) := \frac{1}{2}\sum_{k\in \mathbb{N}_m}\left(\text{max}\left\{1 - y_k\left(\sum_{j \in \mathbb{R}^m }v_j K(\mathbf{x}_j, \mathbf{x}_k)\right), 0\right\}\right)^2, \quad  \mathbf{v} \in \bR^m.
$$
Optimization problem \eqref{classification optimization problem} may be reformulated in the form of \eqref{model: l0} by choosing 
\begin{equation}\label{squared hinge loss}
    \psi(\mathbf{z}):=\frac{1}{2}\sum\limits_{j\in\mathbb{N}_m}\left(\max\{1-z_j,0\}\right)^2, \ \ \mathbf{z}\in\mathbb{R}^m,
\end{equation}
$\mathbf{D}:=\mathbf{I}_{m}$ and $\mathbf{B}:=\mathbf{Y}\mathbf{K}$ where $\mathbf{K}:=[K(\mathbf{x}_j,\mathbf{x}_k):j,k\in\mathbb{N}_m]$ and $\mathbf{Y}:=\mathrm{diag}(y_j:j\in\mathbb{N}_m)$.
}


The $\ell_0$ model \eqref{model: l0} and the $\ell_1$ model \eqref{model: l1} are solved by Algorithm \ref{algo: inexact FPPA l0} and Algorithm \ref{algo: exact FPPA l1}, respectively.
{

For both of these algorithms, the parameter $p$ is varied within the range $[1, 100]$ and $q:=(1+10^{-6})\norm{\mathbf{B}}_2^2/p$.   
For Algorithm \ref{algo: inexact FPPA l0}, the parameters $\lambda$ are varied within the range of $[0.001, 1]$, while for Algorithm \ref{algo: exact FPPA l1}, it varies within $[0.01,5]$. In addition, we vary the parameter $\gamma$ within the range $[10^{-6}, 1]$ and fix  $e^{k+1} := M/k^2$ with $M:= 10^{16}$ for Algorithm \ref{algo: inexact FPPA l0}. 
}

{ 
We first investigate the impact of the parameter $\alpha$ on convergence of Algorithm \ref{algo: inexact FPPA l0} by testing it with various values of $\alpha$ and display the results in Figure \ref{fig: L0 classification convergence divergence}. When $\alpha$ falls within the range $(0, 1)$, increasing its value accelerates the algorithm's convergence speed. However, for $\alpha \in (1, 2)$, speed-up of convergence when further increasing $\alpha$ becomes limited. Notably, Algorithm \ref{algo: inexact FPPA l0} diverges when $\alpha \geq 3.5$. To adhere the theoretical requirement $\alpha \in (0, 1)$ as established in Theorem \ref{thm : convergence algo INFPPA}, the same as in the regression problem, we set $\alpha:=0.99$ in the rest of this example.
}

\begin{figure}
      \centering
	   \begin{subfigure}{0.45\linewidth}
                \includegraphics[width=\linewidth]{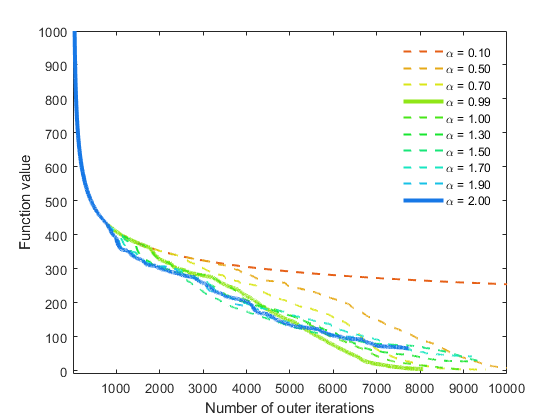}
            	\caption{}
    	       \label{fig:classification_case2_convergence}
	   \end{subfigure}
    \hfill
	   \begin{subfigure}{0.45\linewidth}
        		\includegraphics[width=\linewidth]{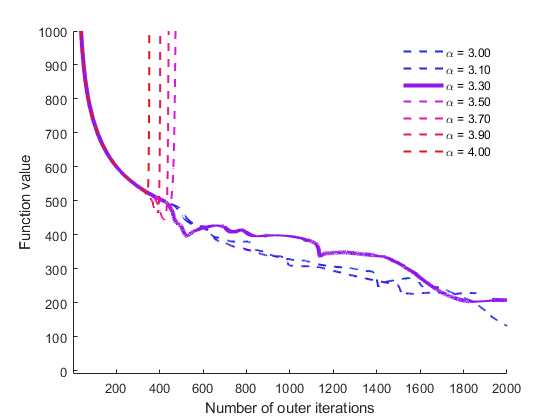}
        		\caption{}
        		\label{fig:classification_case2_divergence}
	    \end{subfigure}
	
	\caption{ The effect of choices of parameter $\alpha$ on convergence of Algorithm \ref{algo: inexact FPPA l0} for  classification.}
	\label{fig: L0 classification convergence divergence}
\end{figure}

We then apply Algorithm \ref{algo: inexact FPPA l0} and its two variants to solve the $\ell_0$ model for classification, maintaining the same stopping criterion as used in the previous subsection. Figures \ref{fig: classification algorithm 1}, \ref{fig: classification algorithm one inner step}, and \ref{fig: classification algorithm extend inner condition} display objective function values generated from Algorithm \ref{algo: inexact FPPA l0}, the first variant, and the second variant, respectively. The observed behavior closely resembles that of the previous example, with the key distinction that for the first variant, the objective function value converges for certain parameters and diverges for others. The numerical results demonstrate the critical importance of the inner iteration in solving the $\ell_0$ model, primarily due to the norm of $\mB$ being 592.

We compare the solution of the $\ell_0$ model solved by Algorithm \ref{algo: inexact FPPA l0} with that of the $\ell_1$ model solved by Algorithm \ref{algo: exact FPPA l1}. The stopping criteria for both of these algorithms remains consistent with that in the previous subsection, except for adjusting $\mathrm{TOL}$ to $10^{-4}$ and setting a maximum iteration count of $5\times 10^4$. In Table \ref{tab: classification}, we compare NNZ and the accuracy on the training dataset (TrA) of the solutions obtained from the $\ell_0$ model and the $\ell_1$ model when they have the same accuracy on the testing dataset (TeA). We observe from Table \ref{tab: classification} that the solutions of the $\ell_0$ model have almost half NNZ compared to those of the $\ell_1$ model, and the $\ell_0$ model produces a higher TrA value compared to the $\ell_1$ model. In Table \ref{tab: classification same nnz}, we compare TeA and TrA of the solutions obtained from the $\ell_0$ model and the $\ell_1$ model when they have comparable NNZ. We see from Table \ref{tab: classification same nnz} that the solution of the $\ell_0$ model has slightly better accuracy on the testing set and about one percent superior on the training set.

\begin{figure}
  \centering
   \begin{subfigure}{0.45\linewidth}
            \includegraphics[width=\linewidth]{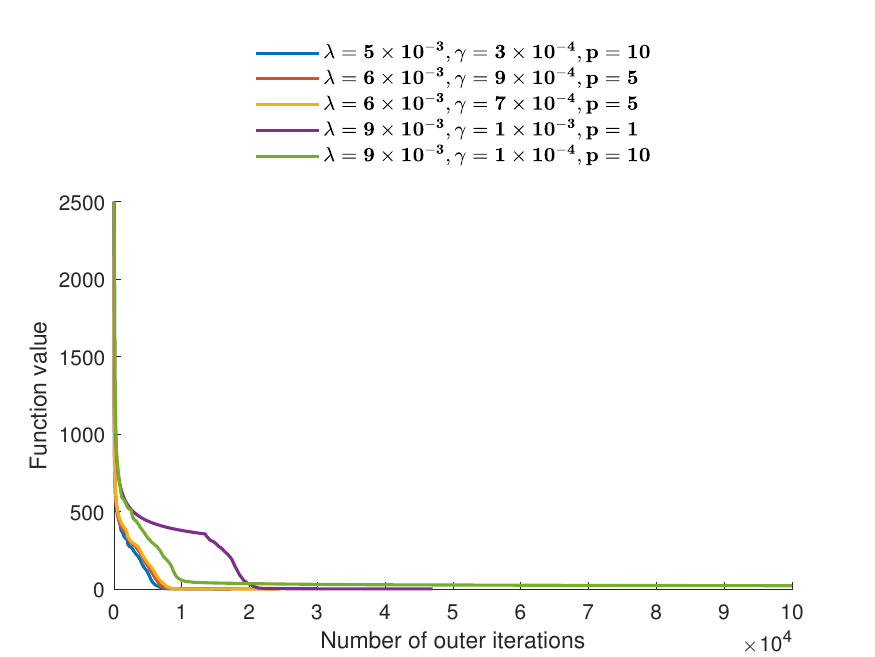}
            \caption{}
          \label{fig: classification algorithm 1}
   \end{subfigure}
\hfill
   \begin{subfigure}{0.45\linewidth}
            \includegraphics[width=\linewidth]{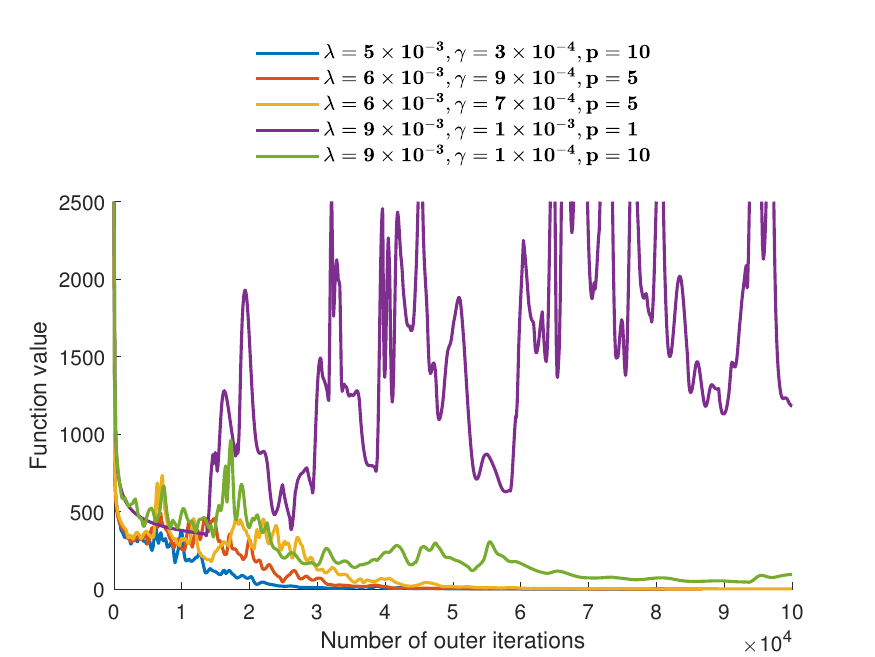}
            \caption{}
            \label{fig: classification algorithm one inner step}
    \end{subfigure}
\vfill
     \begin{subfigure}{0.45\linewidth}
     \includegraphics[width=\linewidth]{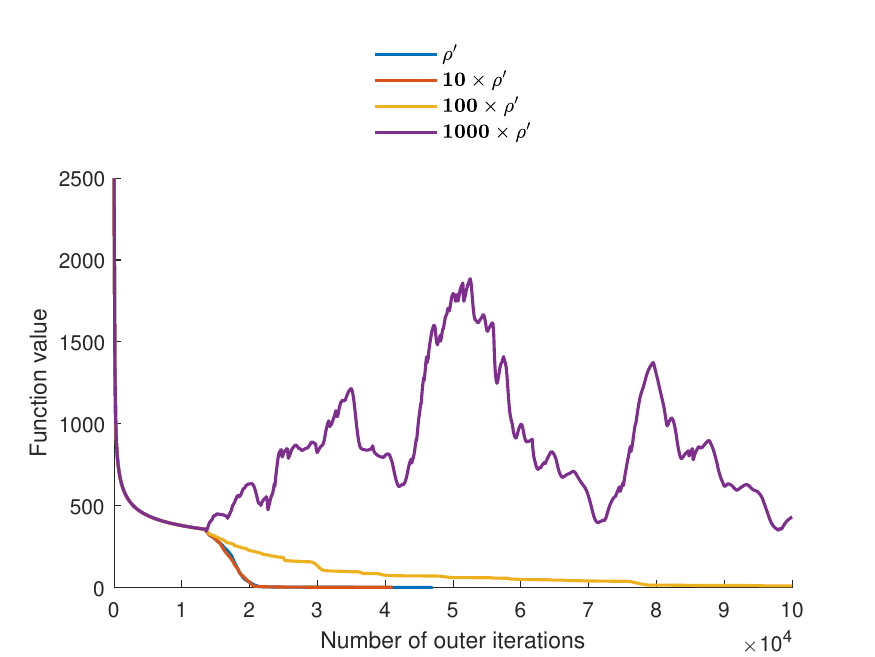}
     \caption{}
     \label{fig: classification algorithm extend inner condition}
      \end{subfigure}
\caption{
 Convergence of Algorithm \ref{algo: inexact FPPA l0} vs its  first and second variants for classification: (a), (b), and (c) display objective function values generated from Algorithm \ref{algo: inexact FPPA l0}, the first variant, and the second variant, respectively. In (c), we set $\lambda, \gamma$, and $p$  
 the values $9.0\times 10^{-3}, 1.0\times 10^{-3} $, and $1$, respectively.}
	\label{fig: L0 classification convergence show}
\end{figure}

\begin{table}
\caption{Comparison of the $\ell_0$ model and the $\ell_1$ model for classification: The NNZ and TrA results  with fixed TeA.} 
    \centering
    \begin{tabular}{c|cc|cc}
    \hline
    &
    \multicolumn{2}{c|}{NNZ} &     \multicolumn{2}{c}{TrA} 
    \\
    \hline 
    
    TeA & $\ell_0$ model & $\ell_1$ model & $\ell_0$ model &  $\ell_1$ model \\
    \hline

    $98.75\%$ & $\mathbf{160}$ & $273$ & $\mathbf{100.00\%}$ & $99.74\%$\\
    

    $98.65\%$ &  $\mathbf{139}$ &  $296$ & $\mathbf{99.96\%}$& $99.90\%$  \\
    
    $98.55\%$ &  $\mathbf{133}$ &$323$ &$\mathbf{99.98\%}$ &$99.90\%$  \\
    
    $98.45\%$ & $\mathbf{110}$& $264$ & $\mathbf{99.98\%}$& $99.66\%$  \\
    
    $98.35\%$ & $\mathbf{86}$ & $146$ & $\mathbf{99.48\%}$& $99.10\%$ \\
    \hline
    \end{tabular}
    \label{tab: classification}
\end{table}

\begin{table}
\caption{Comparison of the $\ell_0$ model and the $\ell_1$ model for classification: The TeA and TrA results within fixed NNZ intervals.} 
    \centering
    \begin{tabular}{c|cc|cc}
    \hline
    &
        \multicolumn{2}{c|}{TeA} &     \multicolumn{2}{c}{TrA} 
    \\   
    \hline 
    NNZ & $\ell_0$ model  & $\ell_1$ mode & $\ell_0$ model &  $\ell_1$ model\\   \hline 
    $[50, 60]$ & $\mathbf{98.00\%}$ & $97.90\%$ & $\mathbf{98.94\%}$ & $97.48\%$  \\ 
    $[80, 90]$ & $\mathbf{98.35\%}$ & $98.25\%$& $\mathbf{99.48\%}$ & $98.34\%$ \\  
    $[100, 120]$&$\mathbf{98.50\%}$ & $98.25\%$ & $\mathbf{99.76\%}$ &$98.74\%$  \\ 
    $[150, 160]$&$\mathbf{98.75\%}$ &$98.25\%$  &$\mathbf{100.00\%}$&$99.20\%$\\     
    $[190, 200]$ &$\mathbf{98.45\%}$& $98.30\%$ &$\mathbf{99.94\%}$&$99.32\%$  \\  \hline
    \end{tabular}

    \label{tab: classification same nnz}
\end{table}

\subsection{Image Deblurring}

{
In this subsection, we consider the $\ell_0$ model solved by Algorithm \ref{algo: inexact FPPA l0} applied to image deblurring. Specifically, we compare performance of models \eqref{model: l0} and \eqref{model: l1} for image deblurring for images contaminated with two different types of noise: the Gaussian noise and the Poisson noise. We present numerical results obtained by using the $\ell_0$ model with those obtained by using two $\ell_1$ models, which will be described precisely later.
Again, we study the impact of the inner iteration on Algorithm \ref{algo: inexact FPPA l0} when applied to image deblurring. In this regard, we present only the Gaussian noise case since the Poisson noise case is similar. 
}

In our numerical experiments, we choose six clean images of ``clock'' with size $200 \times 200$, ``barbara'' with size $256 \times 256$, ``lighthouse'', ``airplane'', and ``bike'' with size $512 \times 512$, and ``zebra'' with size $321 \times 481$, shown in Figure \ref{fig:clean images}. Each tested image is converted to a vector $\mathbf{v}$ via vectorization. The blurred image is modeled as 
\begin{equation}\label{model: deblur problem}
    \mathbf{x} := \mathcal{P}(\mK \mathbf{v})
\end{equation}
where $\mathbf{x}$ denotes the observed corrupted image, $\mK$ is the motion blurring kernel matrix, and $\mathcal{P}$ is the operation that describes the nature of the measurement noise and how it affects the image acquisition. The matrix $\mK$ is generated by the MATLAB function ``fspecial'' with the mode being motion, the angle of the motion blurring being 45 and the length of the motion blurring (``MBL'') varying from $9,15,21$ and all blurring effects are generated by using the MATLAB function ``imfilter'' with symmetric boundary conditions \cite{hansen2006deblurring}. We choose the operator $\mathcal{P}$ as the Gaussian noise or the Poisson noise. The quality of a reconstructed image $\widetilde{ \mathbf{v}}$ obtained from a specific model is evaluated by the peak signal-to-noise ratio (PSNR)
$$
\text { PSNR }:=20 \log _{10}\left(\frac{255}{\left\| \mathbf{v}-\widetilde{ \mathbf{v}}\right\|_2}\right).
$$

\begin{figure}
  \centering
   \begin{subfigure}{0.3\linewidth}
            \includegraphics[width=\linewidth]{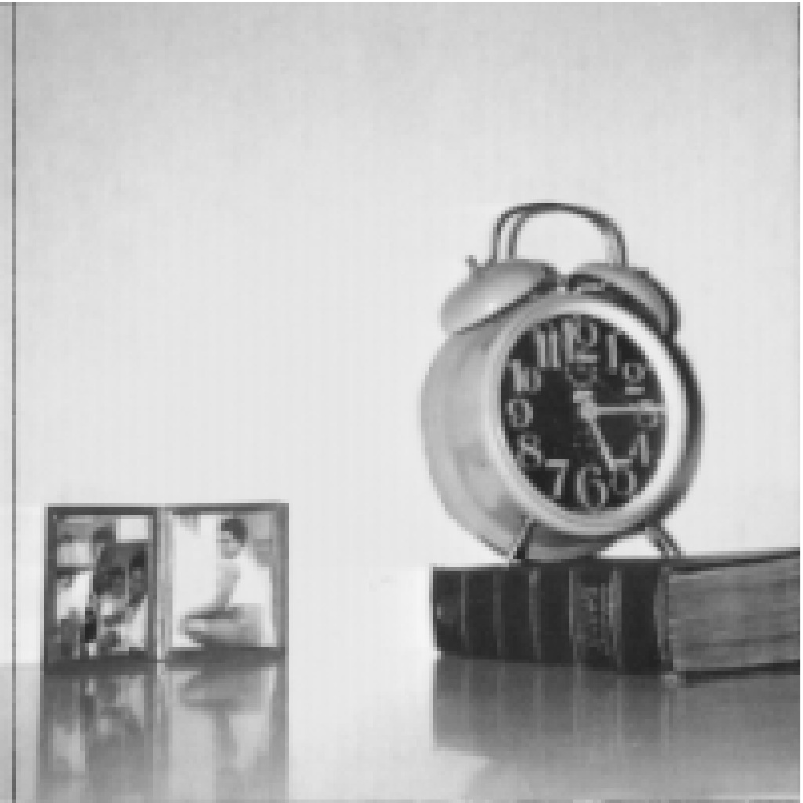}
   \end{subfigure}
   \begin{subfigure}{0.3\linewidth}
            \includegraphics[width=\linewidth]{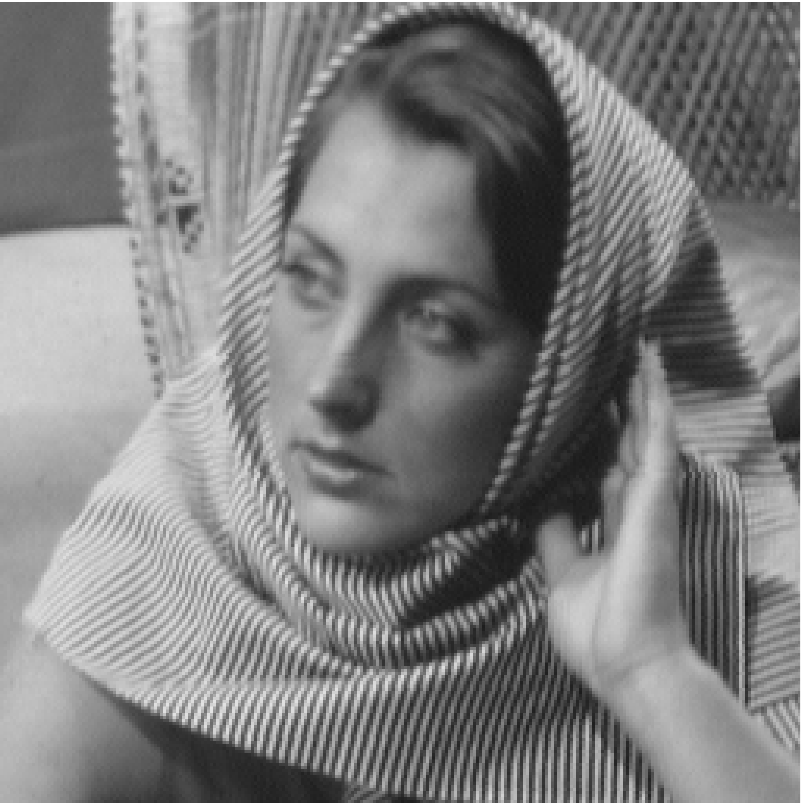}
    \end{subfigure}
     \begin{subfigure}{0.3\linewidth}
     \includegraphics[width=\linewidth]{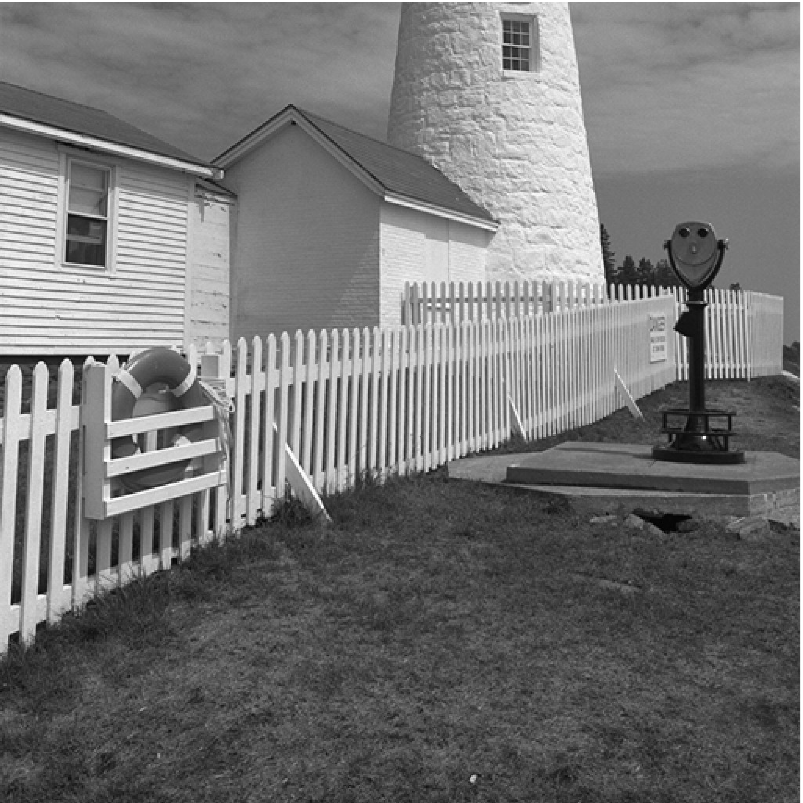}
      \end{subfigure}
   \begin{subfigure}{0.3\linewidth}
            \includegraphics[width=\linewidth]{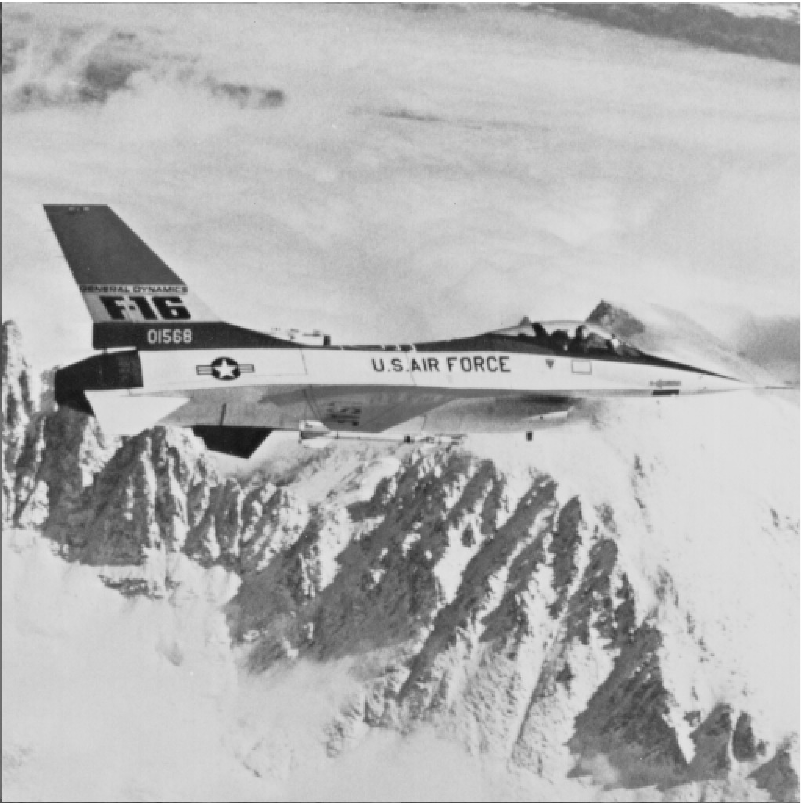}
   \end{subfigure}
   \begin{subfigure}{0.3\linewidth}
            \includegraphics[width=\linewidth]{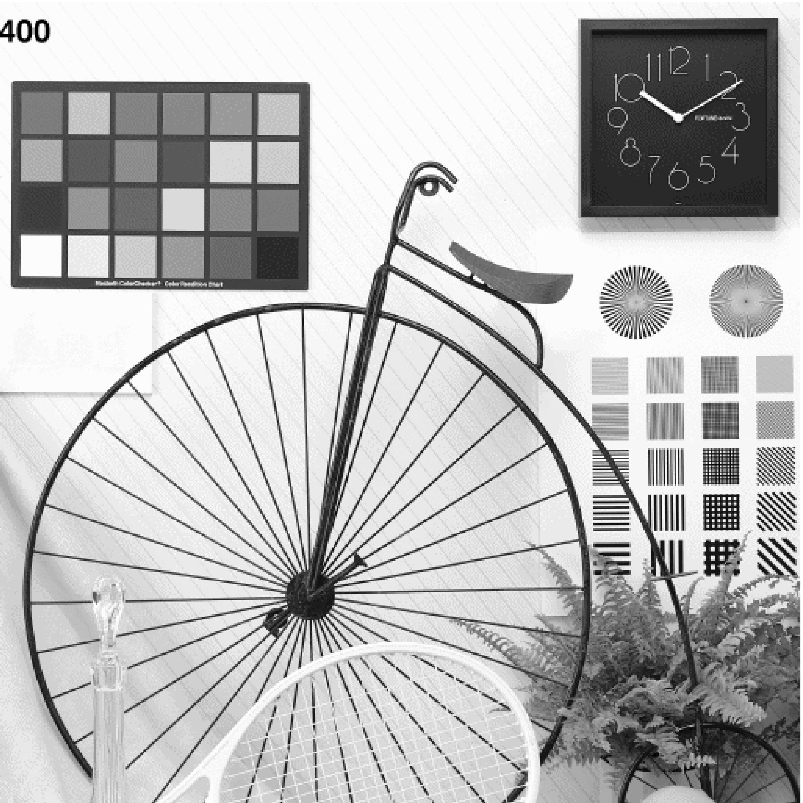}
    \end{subfigure}
     \begin{subfigure}{0.3\linewidth}
     \includegraphics[width=\linewidth]{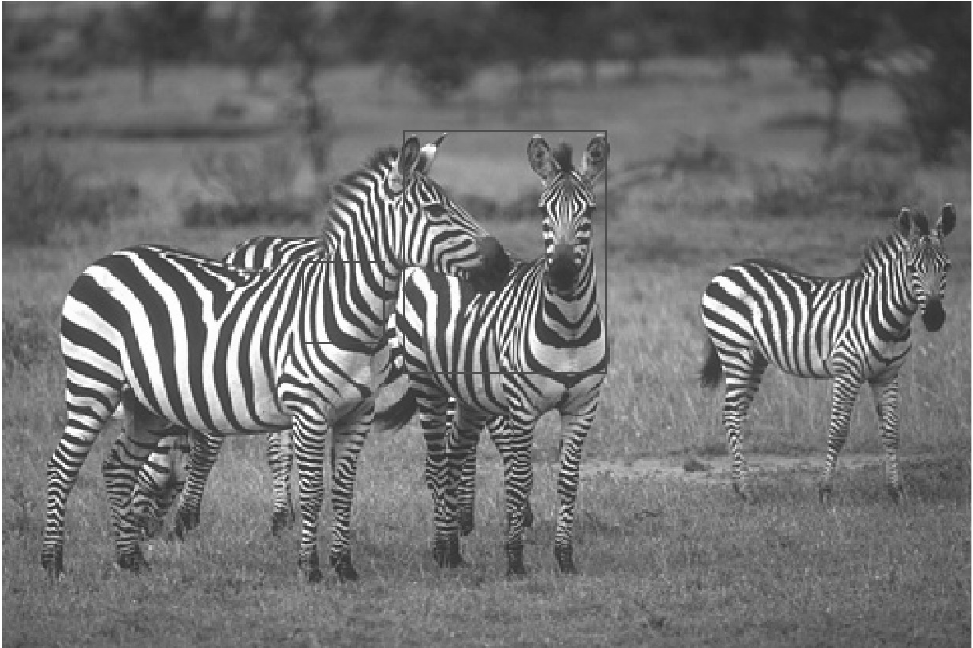}
      \end{subfigure}
	\caption{Clean images.
}
	\label{fig:clean images}
\end{figure}

\subsubsection{Gaussian noise image deblurring}

{ In this case, we choose the operation $\mathcal{P}$ as the Gaussian noise}, and model \eqref{model: deblur problem} becomes 
\begin{equation}\label{model: gaussian noise deblur problem}
   \mathbf{x} := \mathcal{N}(\mK \mathbf{v}, \sigma^2 \mathbf{I} ) 
\end{equation}
where $\mathcal{N}(\mathbf{\mu}, \sigma^2 \mathbf{I})$ denotes the Gaussian noise with mean $\mathbf{\mu}$ and covariance matrix $\sigma^2 \mathbf{I}$. The clean images are first blurred by the motion blurring kernel matrix and then contaminated by the Gaussian noise with standard deviation $\sigma := 3$. The resulting corrupted images are shown in Figure \ref{fig:Blurred and noise images}. 
Given an observed image $\mathbf{x} \in \bR^m$, the function $\psi$ in all three models is chosen to be 
\begin{equation}\label{squared loss}
    \psi(\mathbf{z}):=\frac{1}{2}\norm{\mathbf{z}-\mathbf{x}}_{2}^2, \ \ \mathbf{z} \in \bR^m. 
\end{equation}
The matrix $\mB$ appearing in both models \eqref{model: l0} and \eqref{model: l1} is specified as the motion blurring kernel matrix $\mK$. 
Models \eqref{model: l0} and \eqref{model: l1} with $\mD$ being the tight framelet matrix constructed from discrete cosine transforms of size $7\times 7$ \cite{shen2016wavelet} are respectively referred to as ``L0-TF" and ``L1-TF". Model  \eqref{model: l1} with $\mD$ being the first order difference matrix \cite{micchelli2011proximity} is referred to as ``L1-TV".

\begin{figure}
  \centering
   \begin{subfigure}{0.3\linewidth}
            \includegraphics[width=\linewidth]{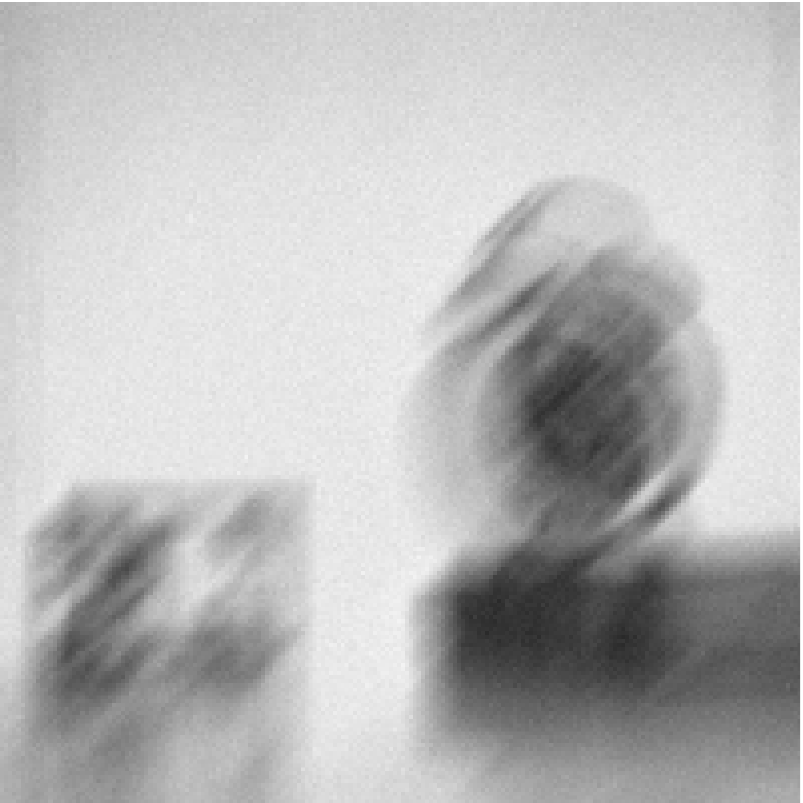}
   \end{subfigure}
   \begin{subfigure}{0.3\linewidth}
            \includegraphics[width=\linewidth]{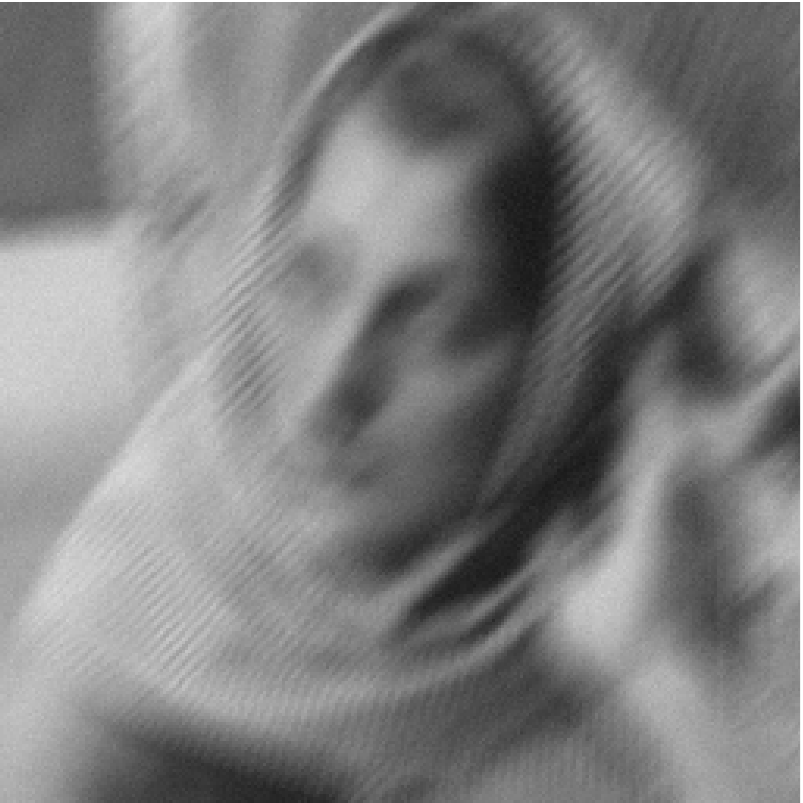}
    \end{subfigure}
     \begin{subfigure}{0.3\linewidth}
     \includegraphics[width=\linewidth]{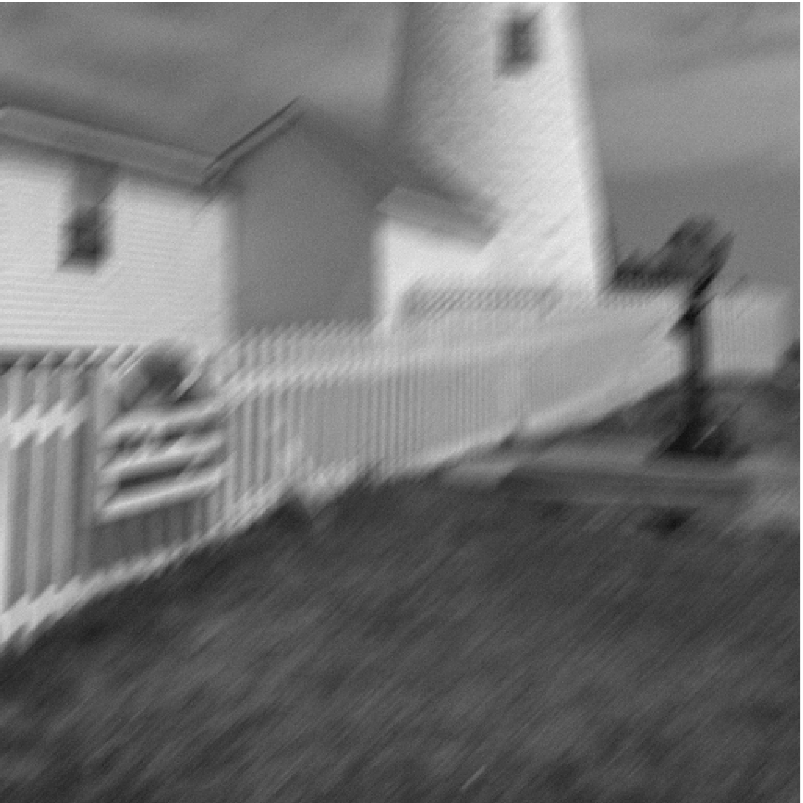}
      \end{subfigure}
   \begin{subfigure}{0.3\linewidth}
            \includegraphics[width=\linewidth]{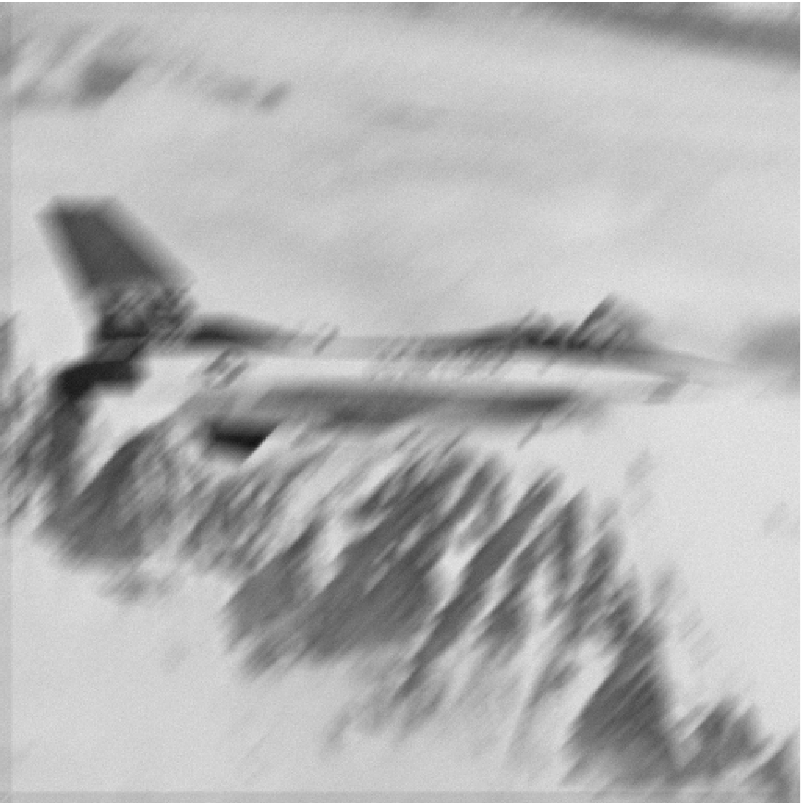}
   \end{subfigure}
   \begin{subfigure}{0.3\linewidth}
            \includegraphics[width=\linewidth]{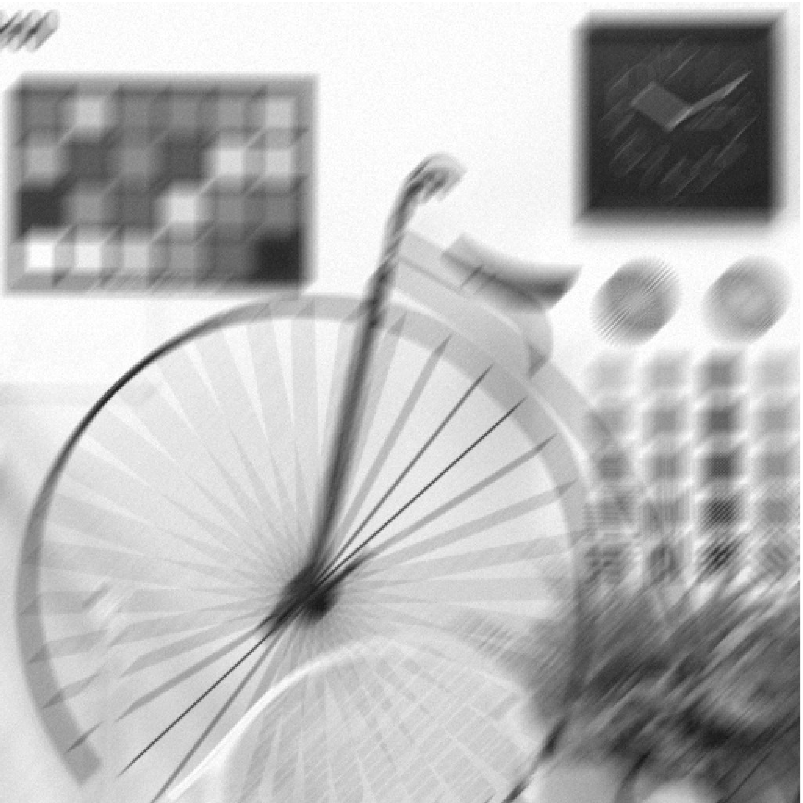}
    \end{subfigure}
     \begin{subfigure}{0.3\linewidth}
     \includegraphics[width=\linewidth]{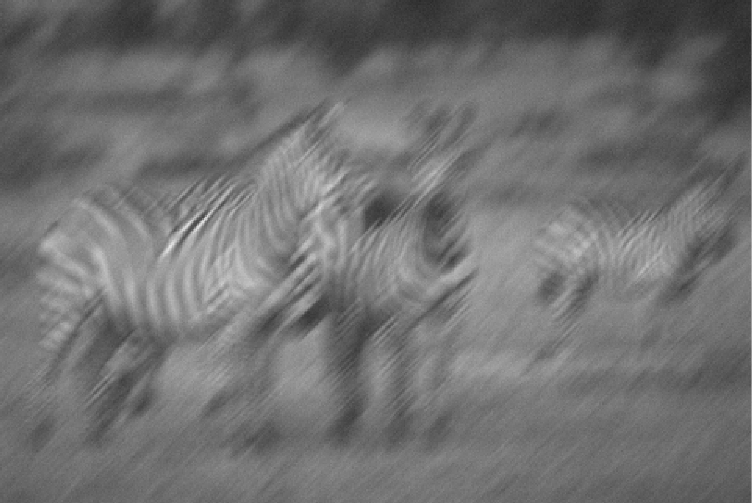}
      \end{subfigure}
	\caption{Corrupted images with MBL being 21 and Gaussian noise.
}
	\label{fig:Blurred and noise images}
\end{figure}

The L0-TF model is solved by Algorithm \ref{algo: inexact FPPA l0}, while both the L1-TF and L1-TV models are solved by Algorithm \ref{algo: Inexact FPPA l1}. In all three models, the parameter $\lambda$ is adjusted within the interval $[0.01, 2]$. For the L0-TF model, we vary the parameter $\gamma$ within the interval $[0.1, 6]$, while maintaining fixed values of $p := 0.1$, $q:=(1+10^{-6}) \times 4 /p$, and $e^{k+1}:= M/k^2$ with $M := 10^{6}$ in Algorithm \ref{algo: inexact FPPA l0}. For both the L1-TF and L1-TV models, the parameter $p_1 = p_2$ is varied within the interval $[0.01, 1]$ and $q_2:=(1+10^{-6}) \times  4/p_2$, while $q_1:=(1+10^{-6})/p_1$ for the L1-TF model and $q_1:=(1+10^{-6}) \times 8 /p_1$ for the L1-TV model.
Additionally, we set $e^{k+1}:= M/k^2$ in Algorithm \ref{algo: Inexact FPPA l1}, where $M := 10^{8}$ for the L1-TF model and $M:=10^{7}$ for the L1-TV model.
The stopping criterion for both algorithms is when \eqref{eq: stop} is satisfied with $\mathbf{x}^k:=  \mathbf{\tilde v}^k$ and $\mathrm{TOL}:=10^{-5}$.

{ 
We explore the effect of $\alpha$ in Algorithm \ref{algo: inexact FPPA l0} for images `clock' and `barbara' by testing different values of $\alpha$ and display the results in Figure \ref{fig: L0 deblurring alpha}. As observed in Figure \ref{fig: L0 deblurring alpha}, when the value of $\alpha$ is close to $1$, Algorithm \ref{algo: inexact FPPA l0} tends to yield smaller objective function values over the same number of iterations.
Figure \ref{fig: L0 deblurring convergence divergence} provides zoomed-in views of values of the objective function for the last 800 iterations for the images of `clock' and `barbara', with $\alpha:=1.00, 1.01$. While  Algorithm \ref{algo: inexact FPPA l0} converges for both the $\alpha$ values for the image `clock', it appears that it converges for $\alpha:=1$ but diverges for $\alpha:=1.01$ for the `barbara' image. These numerical results confirm the condition  $\alpha \in (0, 1)$ for convergence of Algorithm \ref{algo: inexact FPPA l0} established in Theorem \ref{thm : convergence algo INFPPA}.
Consequently, we set $\alpha:=0.99$ in the rest of this example to ensure convergence of Algorithm \ref{algo: inexact FPPA l0}.
}

\begin{figure}
  \centering
   \begin{subfigure}{0.45\linewidth}
            \includegraphics[width=\linewidth]{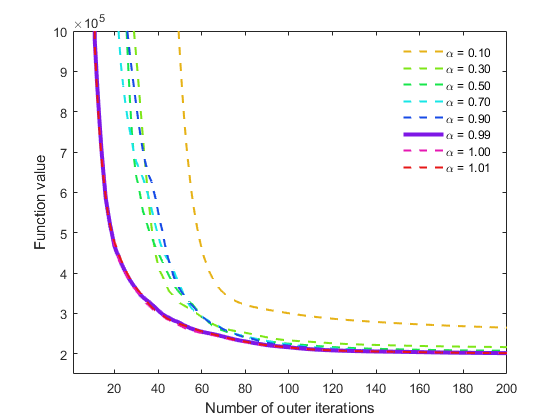}
            \caption{}
          \label{fig: clock alpha}
   \end{subfigure}
\hfill
   \begin{subfigure}{0.45\linewidth}
            \includegraphics[width=\linewidth]{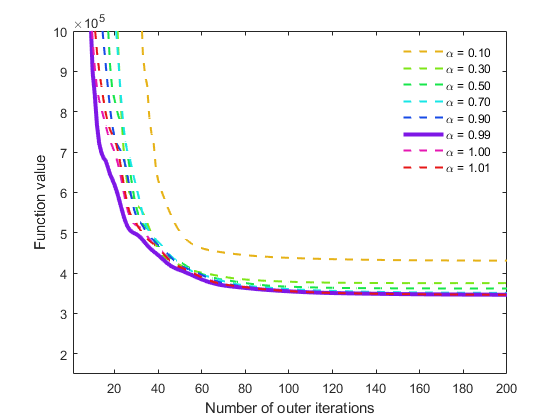}
            \caption{}
            \label{fig: barbara alpha}
    \end{subfigure}
	\caption{The effect of choices of parameter $\alpha$ on convergence of Algorithm \ref{algo: inexact FPPA l0} for the Gaussian noise image deblurring problem of the image ``clock" (left) and ``barbara" (right).
  }
	\label{fig: L0 deblurring alpha}
\end{figure}

\begin{figure}
  \centering
   \begin{subfigure}{0.45\linewidth}
            \includegraphics[width=\linewidth]{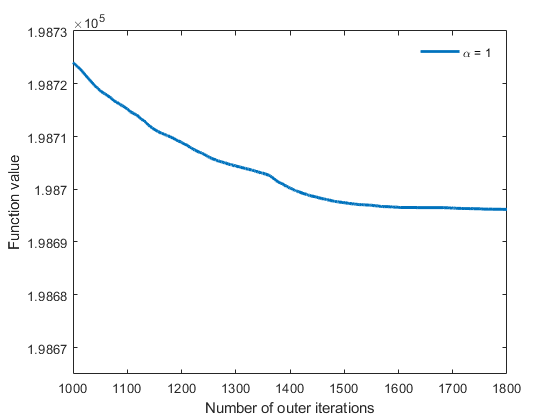}
            \caption{}
   \end{subfigure}
\hfill
   \begin{subfigure}{0.45\linewidth}
            \includegraphics[width=\linewidth]{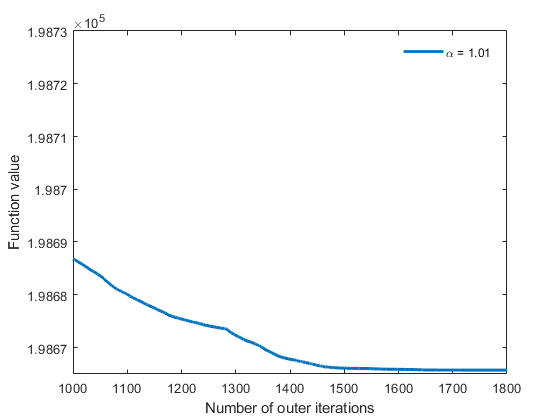}
            \caption{}
    \end{subfigure}
\vfill
     \begin{subfigure}{0.45\linewidth}
     \includegraphics[width=\linewidth]{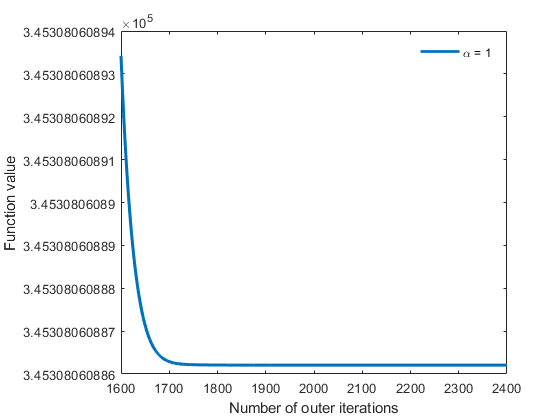}
     \caption{}
      \end{subfigure}
\hfill
        \begin{subfigure}{0.45\linewidth}
     \includegraphics[width=\linewidth]{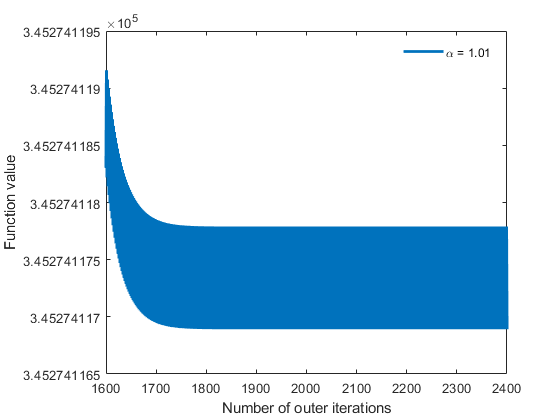}
     \caption{}
      \end{subfigure}
      \caption{Zoomed-in views of the objective function values generated by  Algorithm \ref{algo: inexact FPPA l0}  with $\alpha:=1.00, 1.01$ for the Gaussian noise image deblurring problem of the image ``clock" ((a), (b)) and ``barbara" ((c), (d)).
  }
	\label{fig: L0 deblurring convergence divergence}
\end{figure}

We then apply Algorithm \ref{algo: inexact FPPA l0} and the first variant to solve the L0-TF model. 
Figures \ref{fig: deblurring algorithm 1 function value} and \ref{fig: deblurring algorithm one inner step function value} depict objective function values generated from Algorithm \ref{algo: inexact FPPA l0} and its first variant, respectively, while Figures \ref{fig: deblurring algorithm 1 PSNR} and \ref{fig: deblurring algorithm one inner step function PSNR} illustrate the PSNR values of images reconstructed by Algorithm \ref{algo: inexact FPPA l0} and its first variant, respectively. Figure \ref{fig: deblurring algorithm 1 function value} confirms convergence of Algorithm \ref{algo: inexact FPPA l0} for all six tested images with values of parameters $\lambda$ and $\gamma$ displayed. Figure \ref{fig: deblurring algorithm one inner step function value} shows that the objective function values generated from the first variant 
exhibit significant oscillation for the first 90 steps of iteration, followed by a convergence trend. These observations are also reflected in the PSNR value as depicted in Figures \ref{fig: deblurring algorithm 1 PSNR} and \ref{fig: deblurring algorithm one inner step function PSNR}. The numerical results show that the significance of the inner iteration is limited for image deblurring problems, due to the fact that the norm of matrix $\mathbf{B}$ is not greater than 2.

{
We list in Table \ref{tab: TV-noise-elL1-TF-noise-elL0-TF-deblurred} the highest PSNR values for each of the reconstructed images, with the corresponding values of parameters $\lambda$, $\gamma$, $p$ and $p_1$ showed in Table \ref{tab: best parameter gaussian noise}.}
The results in Table \ref{tab: TV-noise-elL1-TF-noise-elL0-TF-deblurred} consistently demonstrate that the L0-TF model achieves the highest PSNR-value, surpassing the second-highest value by approximately $1$ dB. For visual comparison, we display the deblurred images in  Figures \ref{fig: TV-L0-L1-deblur part 1} and \ref{fig: TV-L0-L1-deblur part 2}.  One can see that the L0-TF model is more powerful to suppress noise on smooth areas compared to the L1-TF and L1-TV models. In Figure \ref{fig: zebra head and neck}, comparison of two specific local areas in the zebra image is presented. We observe from Figure \ref{fig: zebra head and neck} that the deblurred images obtained from the L0-TF model preserve finer details and edges more effectively than those obtained from the other two models.

\begin{figure}
  \centering
   \begin{subfigure}{0.45\linewidth}
            \includegraphics[width=\linewidth]{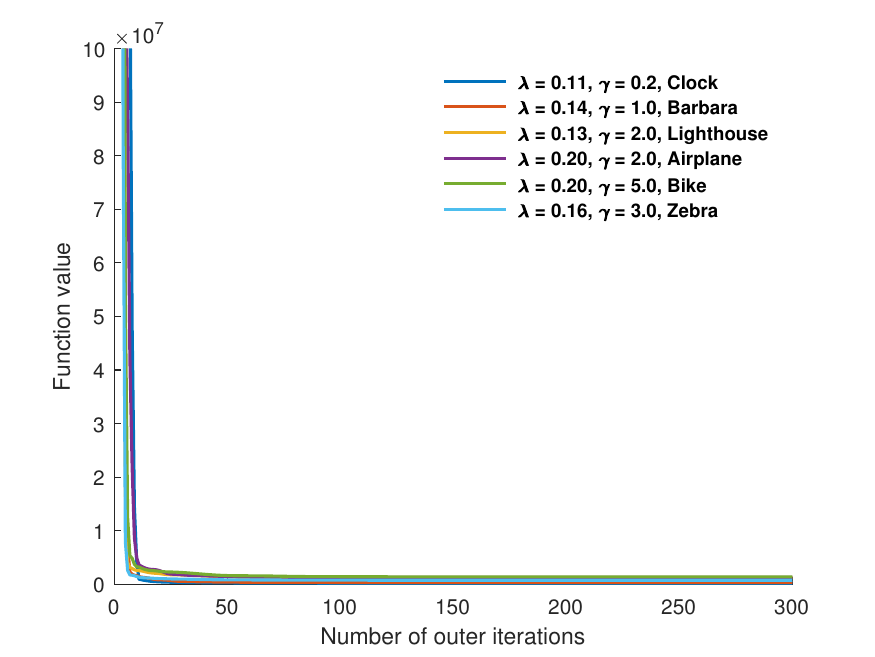}
            \caption{}
          \label{fig: deblurring algorithm 1 function value}
   \end{subfigure}
\hfill
   \begin{subfigure}{0.45\linewidth}
            \includegraphics[width=\linewidth]{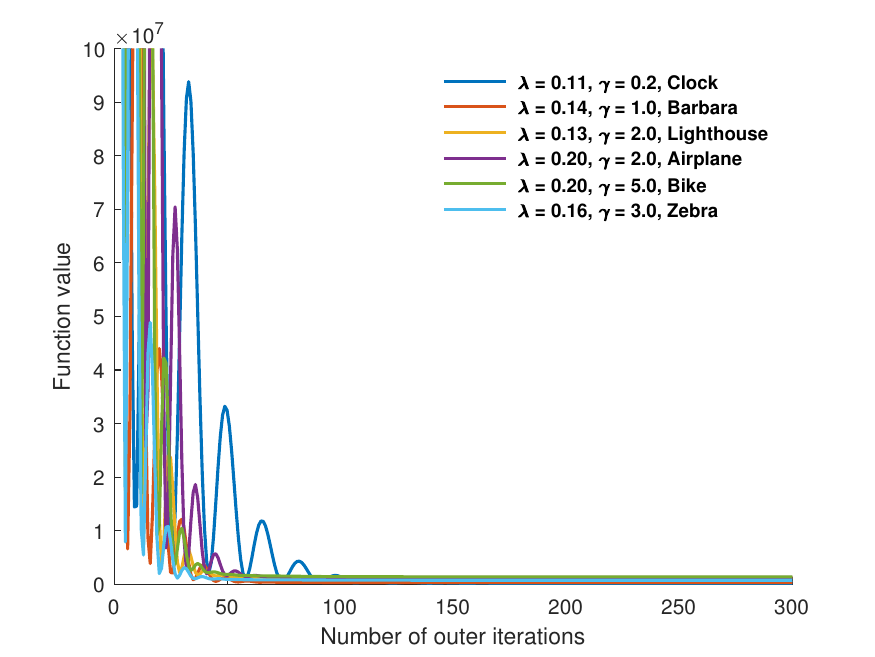}
            \caption{}
            \label{fig: deblurring algorithm one inner step function value}
    \end{subfigure}
\vfill
     \begin{subfigure}{0.45\linewidth}
     \includegraphics[width=\linewidth]{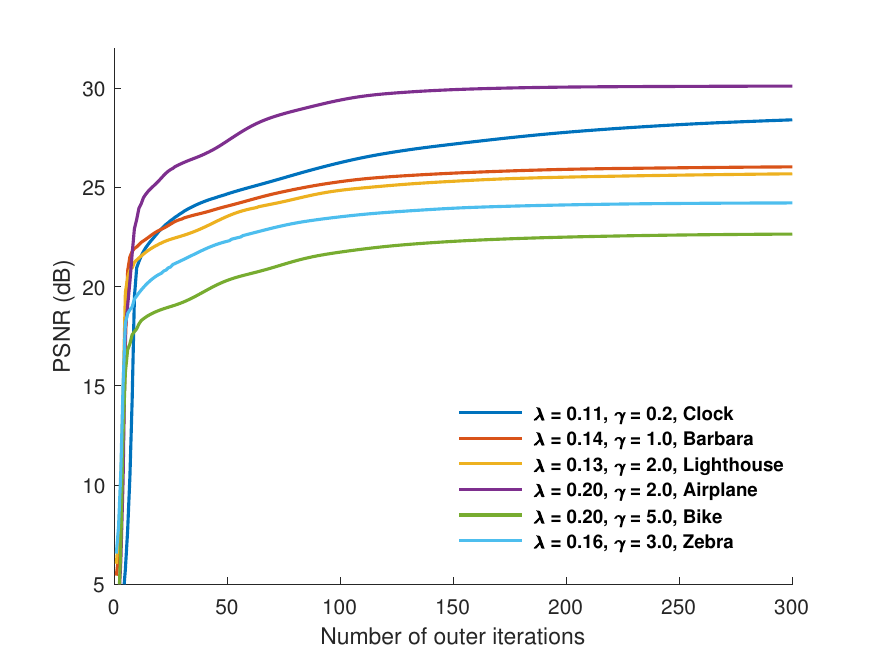}
     \caption{}
     \label{fig: deblurring algorithm 1 PSNR}
      \end{subfigure}
\hfill
        \begin{subfigure}{0.45\linewidth}
     \includegraphics[width=\linewidth]{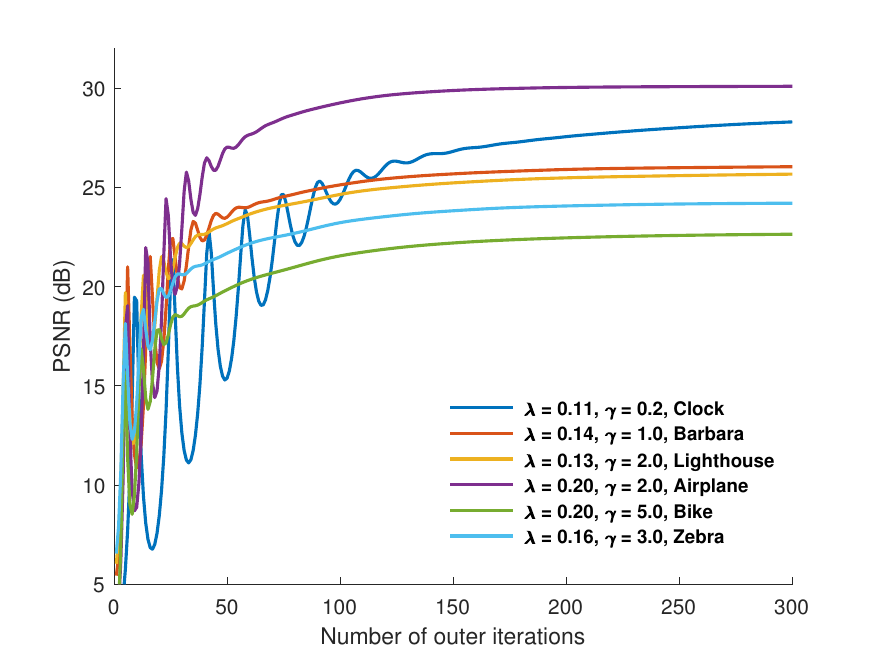}
     \caption{}
     \label{fig: deblurring algorithm one inner step function PSNR}
      \end{subfigure}
	\caption{
 Exploring the convergence of Algorithm \ref{algo: inexact FPPA l0} 
 for image deblurring problems. (a) and (b) depict the function values, while (c) and (d) illustrate the PSNR values for Algorithm \ref{algo: inexact FPPA l0} and its first variant, respectively. 
 }
	\label{fig: L0 deblurring convergence show}
\end{figure}

\begin{table}
    \caption{Gaussian noise image deblurring:
     PSNB values (dB) of L0-TF, L1-TF, and L1-TV} 
    \label{tab: TV-noise-elL1-TF-noise-elL0-TF-deblurred}
    \centering
    \begin{tabular}{c|ccc|ccc}
    \hline
    &
 
    \multicolumn{3}{c|}{clock $200 \times 200$} &     \multicolumn{3}{c}{barbara $256 \times 256$} 
    \\
    \hline 
    \backslashbox[1cm]{model}{\footnotesize{MBL}}
    & 9 & 15 & 21 &  9 & 15 & 21
    \\
    
    \hline
    L0-TF & \textbf{31.20}  &\textbf{30.21}  & \textbf{28.95} & \textbf{29.64} & \textbf{27.09} & \textbf{26.18} \\
    L1-TF & 30.13 & 29.15 & 28.05 &28.66 & 26.54 &  25.89\\
    L1-TV  &  29.22  & 28.41  & 27.18  &26.98 & 25.09 & 24.40\\
    \hline 
&  \multicolumn{3}{c|}{ lighthouse $512 \times 512$}&
     \multicolumn{3}{c}{airplane $512 \times 512$}
    \\
    \hline 
    \backslashbox[1cm]{model}{\footnotesize{MBL}}
    & 9 & 15 & 21& 9 & 15 & 21
    \\
    \hline
    L0-TF &  \textbf{27.81} & \textbf{26.50} & \textbf{25.85}
    & \textbf{32.72} & \textbf{31.31} & \textbf{30.13}\\
    L1-TF 
    & 26.75  & 25.42& 24.93 
    & 31.88 & 30.39 & 29.39\\
    L1-TV  
    & 26.14 &24.73 & 24.08
    & 31.29 &29.75 & 28.71 \\
\hline
&

    \multicolumn{3}{c|}{ bike $512 \times 512$}  &
    \multicolumn{3}{c}{ zebra $321 \times 481$} 
 
    \\
    \hline 
    \backslashbox[1cm]{model}{\footnotesize{MBL}}
    & 9 & 15 & 21& 9 & 15 & 21
    \\
    \hline
    L0-TF & \textbf{25.04}& \textbf{23.99}&  \textbf{22.73}&  
    \textbf{26.33} &  \textbf{25.04} &  \textbf{24.23} \\
     L1-TF 
    
    &  24.00&  23.01 &  21.91 
    
    &  25.42&  24.32 &  23.52 \\   
    L1-TV  &  24.05&  22.70& 21.70   
    
    &  25.05&  23.78& 23.01 

    \\
\hline
    \end{tabular}

\end{table}

\begin{small}

\begin{table}
    \caption{
    Gaussian noise image deblurring: chosen values of the parameters} 
    \label{tab: best parameter gaussian noise}
    \centering
    \begin{tabular}{cc|cc|cc|cc}
    \hline
    & &
    \multicolumn{2}{c|}{L0-TF} &     \multicolumn{2}{c|}{ L1-TF} & \multicolumn{2}{c}{L1-TV}     
    \\
\hline
image & MBL & $\lambda$ & $\gamma$ & $\lambda $  & $p_1$ &  $\lambda $  & $p_1$
    \\
    \hline
    \multirow{3}{*}{clock}& 9&0.16 & 0.6 & 0.05 & 0.09 & 0.22& 0.35\\
    &15& 0.13& 0.5&0.05 & 0.25& 0.12& 0.17\\
    &21& 0.11& 0.2&0.03 &0.07 &  0.16& 0.07 \\
    \hline
    \multirow{3}{*}{barbara}& 9&0.17 & 0.4& 0.05& 0.05& 0.12& 2.3\\
    &15&0.14 &1.0 & 0.04& 0.03& 0.10& 1.7\\
    &21&0.14 &1.0 &0.03 &0.06 &0.10 & 0.9\\
    \hline
    \multirow{3}{*}{lighthouse}& 9&0.18 &2.2 & 0.05 & 0.35 & 0.14 & 0.50 \\
    &15&0.16 &2.6 & 0.03 &0.03 & 0.14& 0.50\\
    &21&0.13 &2.0 &0.03 &0.07 & 0.12 & 0.15\\
    \hline
    \multirow{3}{*}{airplane}& 9& 0.20 & 0.6& 0.05& 0.07 & 0.28& 0.40\\
    &15&0.20 &1.0 & 0.05 &0.25 & 0.22 & 0.30\\
    &21&0.20 &2.0 &  0.03& 0.07 & 0.18& 0.20\\
    \hline
    \multirow{3}{*}{bike}& 9& 0.20& 3.0 & 0.03 & 0.03 & 0.12 & 0.09\\
    &15& 0.20 &5.0 & 0.05& 0.15& 0.10& 0.20\\
    &21& 0.20& 5.0& 0.03 & 0.07 & 0.10& 0.03\\
    \hline
    \multirow{3}{*}{zebra}& 9& 0.21& 2.0 & 0.04& 0.15 & 0.12& 0.17\\
    &15& 0.19&2.7 & 0.03& 0.10 & 0.11 & 0.14\\
    &21& 0.16&3.0 & 0.03& 0.17 & 0.10& 0.14\\
    \hline
\end{tabular}
\end{table}

\end{small}
 
\begin{figure}
  \centering
   \begin{subfigure}{0.3\linewidth}
            \includegraphics[width=\linewidth]{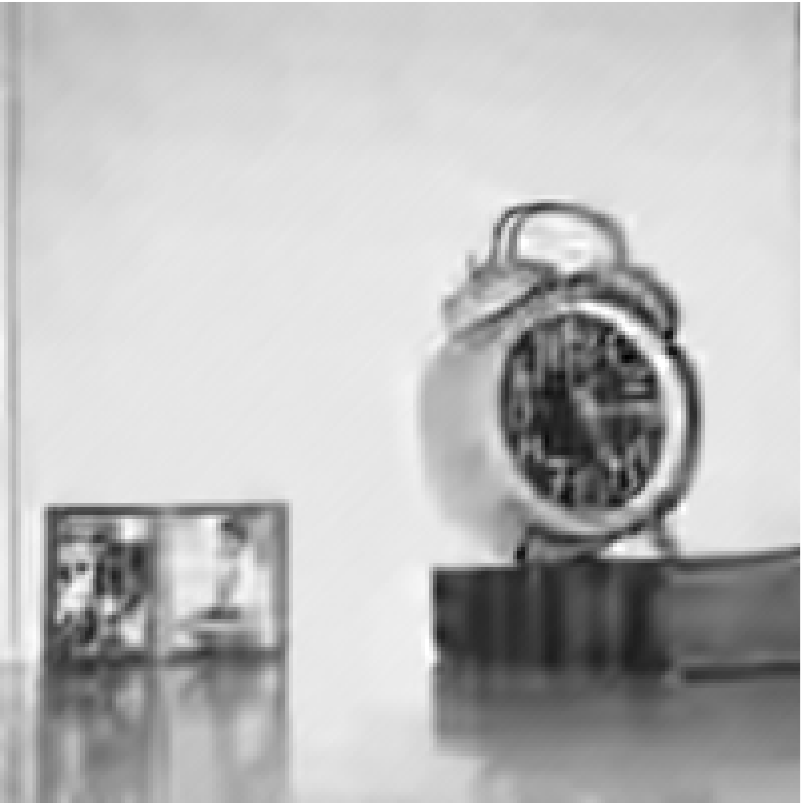}
            \caption*{L0-TF, PSNR: 28.95}
   \end{subfigure}
   \begin{subfigure}{0.3\linewidth}
            \includegraphics[width=\linewidth]{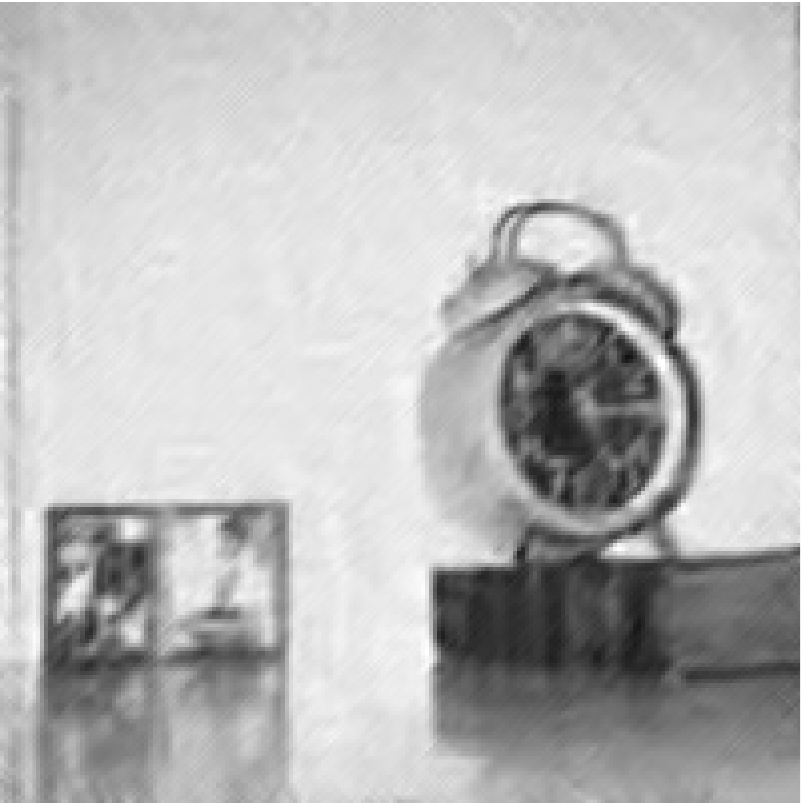}
            \caption*{L1-TF, PSNR: 28.05 }
    \end{subfigure}
     \begin{subfigure}{0.3\linewidth}
     \includegraphics[width=\linewidth]{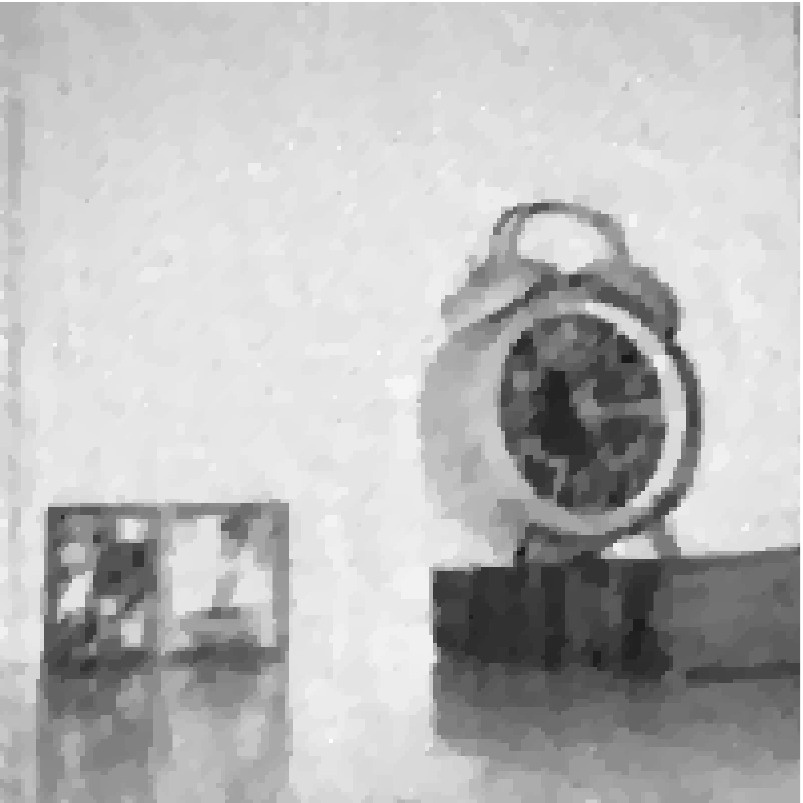}
     \caption*{L1-TV, PSNR: 27.18}
      \end{subfigure}
   \begin{subfigure}{0.3\linewidth}
            \includegraphics[width=\linewidth]{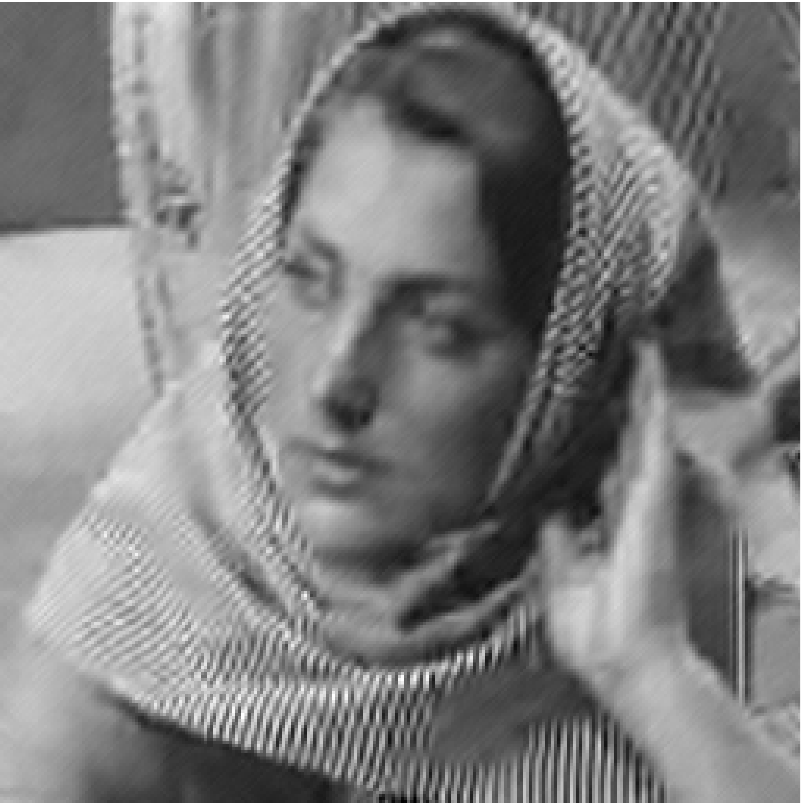}
            \caption*{L0-TF, PSNR: 26.18}
   \end{subfigure}
   \begin{subfigure}{0.3\linewidth}
            \includegraphics[width=\linewidth]{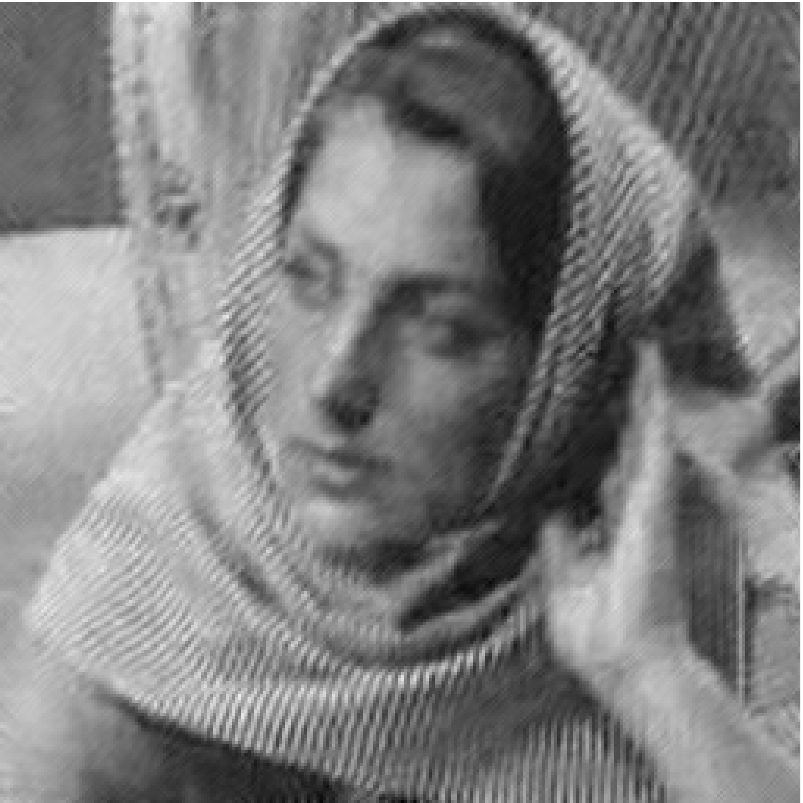}
            \caption*{L1-TF, PSNR: 25.89}
    \end{subfigure}
     \begin{subfigure}{0.3\linewidth}
     \includegraphics[width=\linewidth]{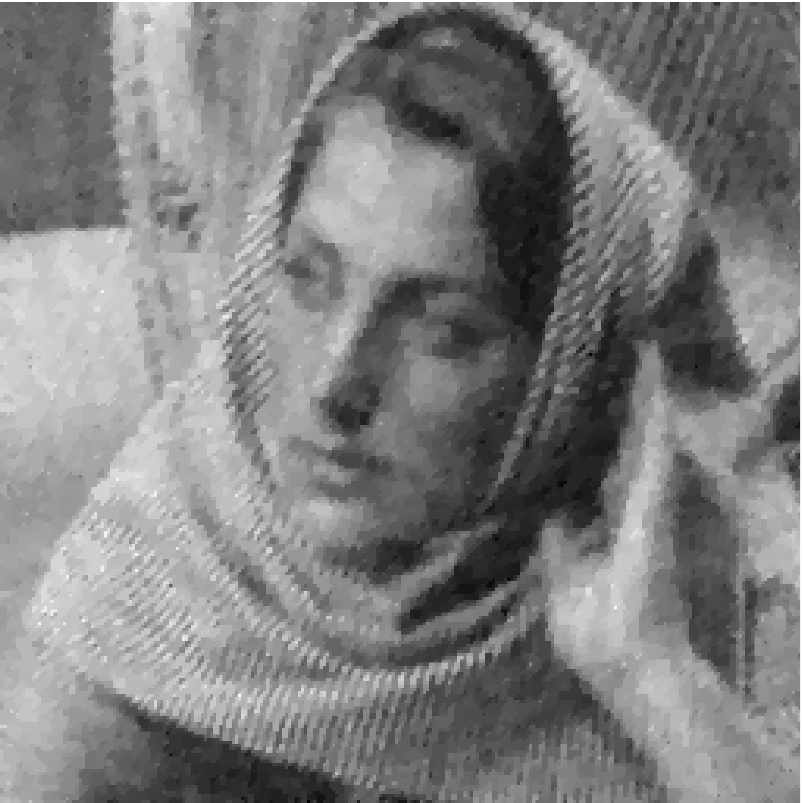}
     \caption*{L1-TV, PSNR: 24.40}
      \end{subfigure}

   \begin{subfigure}{0.3\linewidth}
            \includegraphics[width=\linewidth]{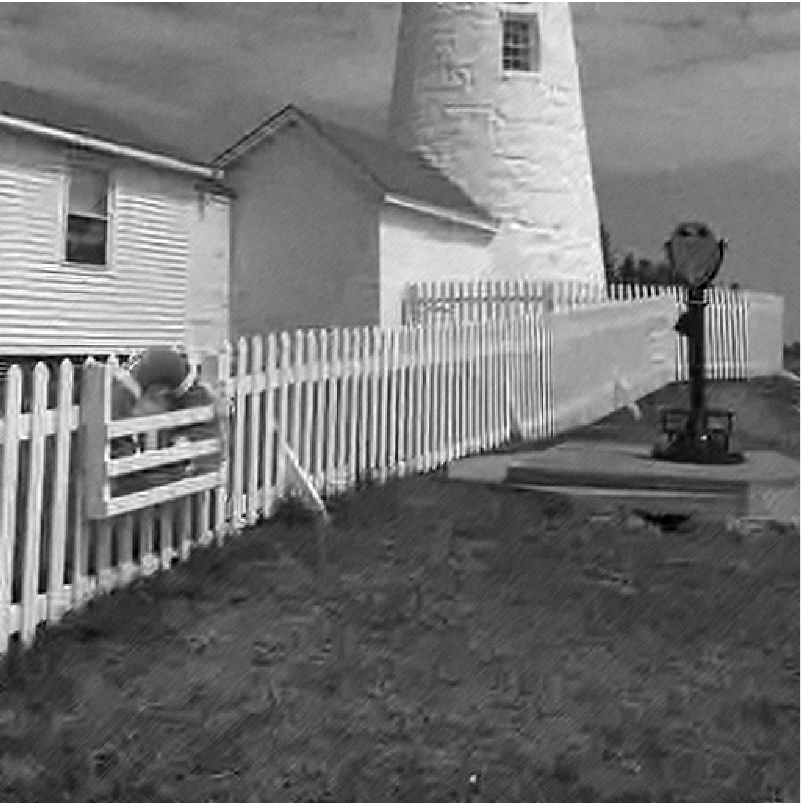}
            \caption*{L0-TF, PSNR: 25.85}
   \end{subfigure}
   \begin{subfigure}{0.3\linewidth}
            \includegraphics[width=\linewidth]{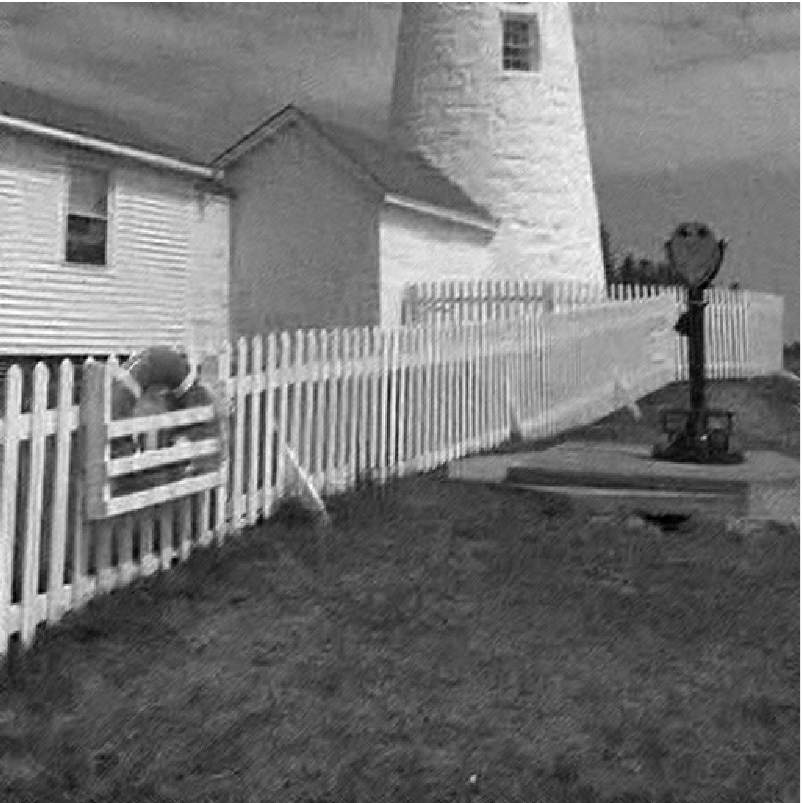}
            \caption*{L1-TF, PSNR: 24.93}
    \end{subfigure}
     \begin{subfigure}{0.3\linewidth}
     \includegraphics[width=\linewidth]{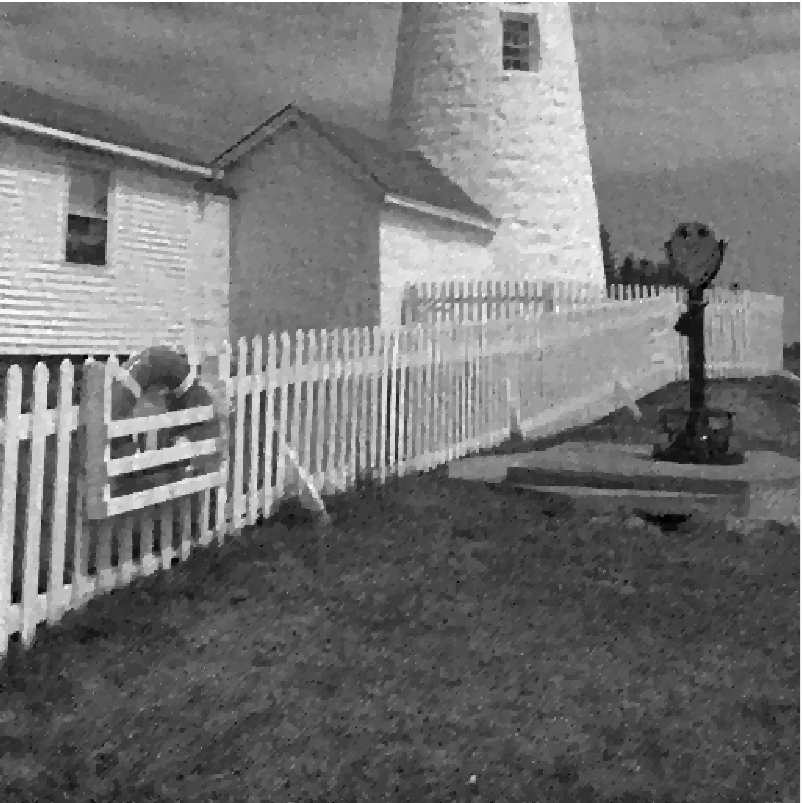}
     \caption*{L1-TV, PSNR: 24.08}
      \end{subfigure}

	\caption{Gaussian noise image deblurring: Deblurred images (clock, barbara, lighthouse) from corrupted images with MBL being 21. 
}
	\label{fig: TV-L0-L1-deblur part 1}
\end{figure} 
 
\begin{figure}
  \centering
   \begin{subfigure}{0.3\linewidth}
            \includegraphics[width=\linewidth]{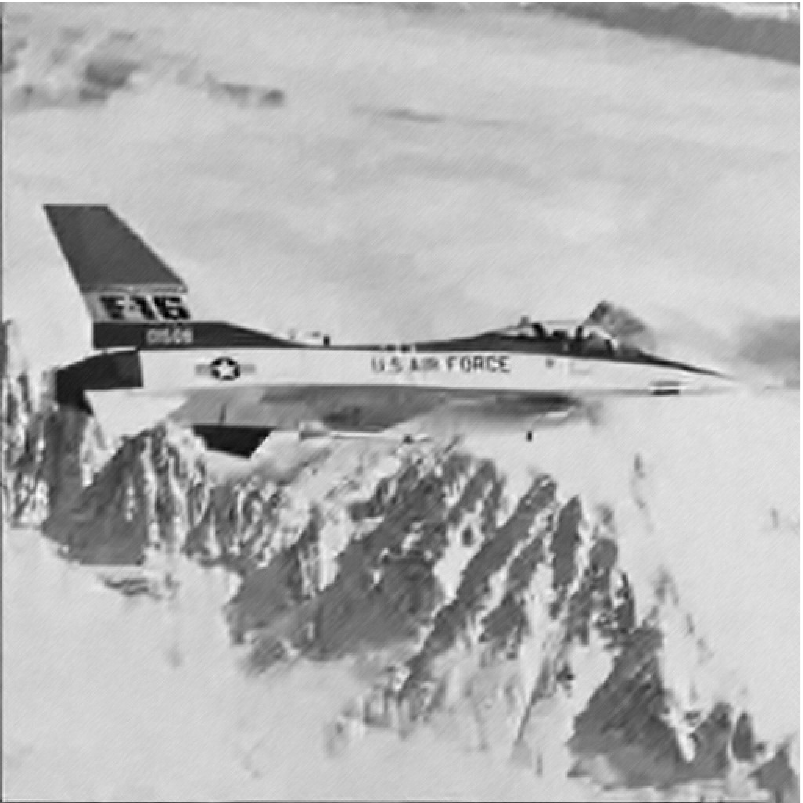}
            \caption*{L0-TF, PSNR: 30.13}
   \end{subfigure}
   \begin{subfigure}{0.3\linewidth}
            \includegraphics[width=\linewidth]{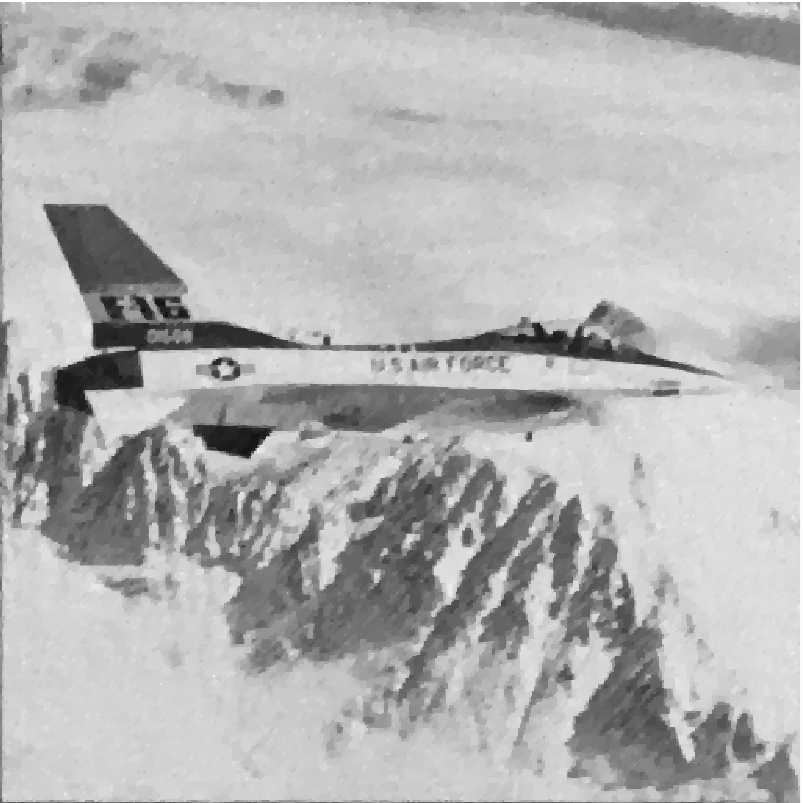}
            \caption*{L1-TF, PSNR: 29.39}
    \end{subfigure}
     \begin{subfigure}{0.3\linewidth}
     \includegraphics[width=\linewidth]{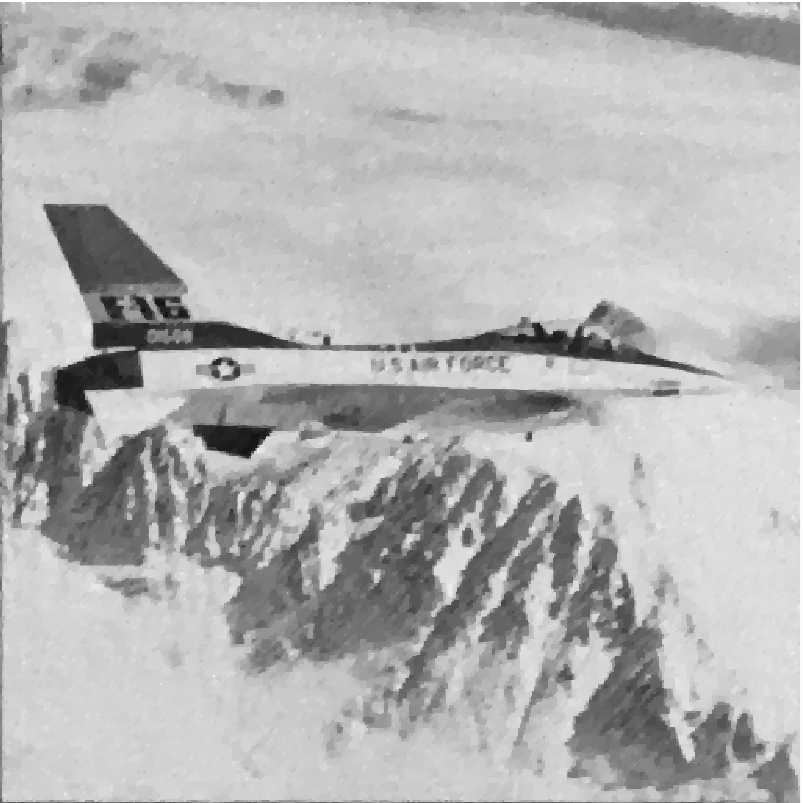}
     \caption*{L1-TV, PSNR: 28.71}
      \end{subfigure}
   \begin{subfigure}{0.3\linewidth}
   
            \includegraphics[width=\linewidth]{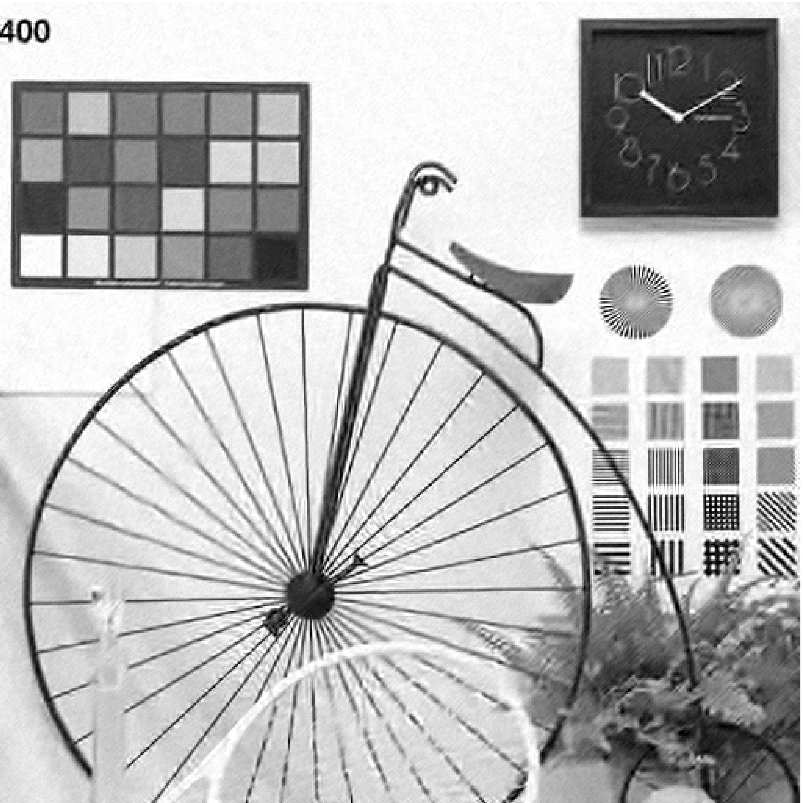}
            \caption*{ L0-TF, PSNR: 22.73}
   \end{subfigure}
   \begin{subfigure}{0.3\linewidth}
            \includegraphics[width=\linewidth]{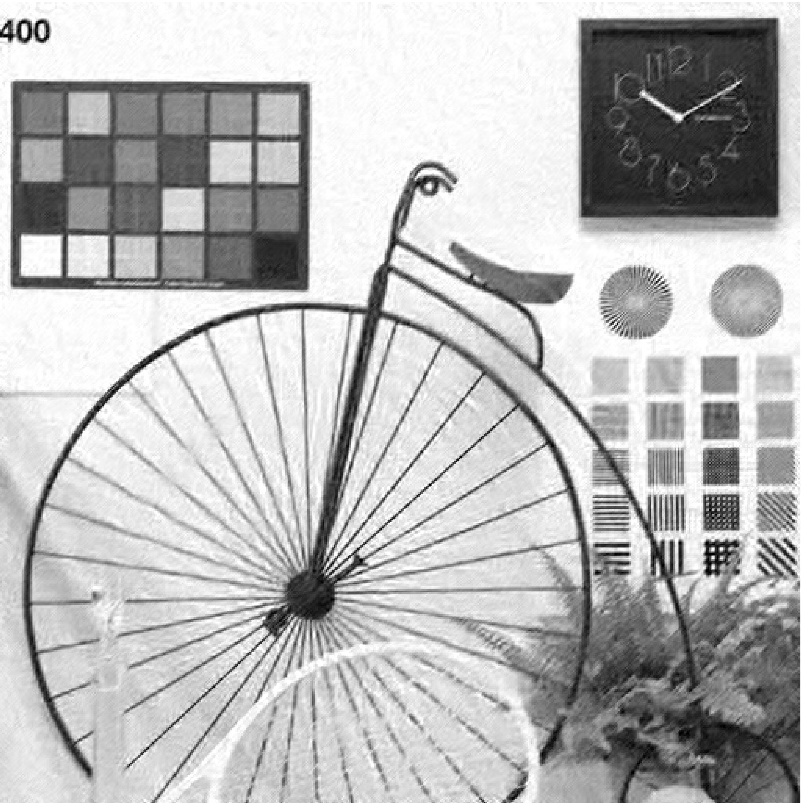}
            \caption*{L1-TF, PSNR: 21.91}
    \end{subfigure}
     \begin{subfigure}{0.3\linewidth}
     \includegraphics[width=\linewidth]{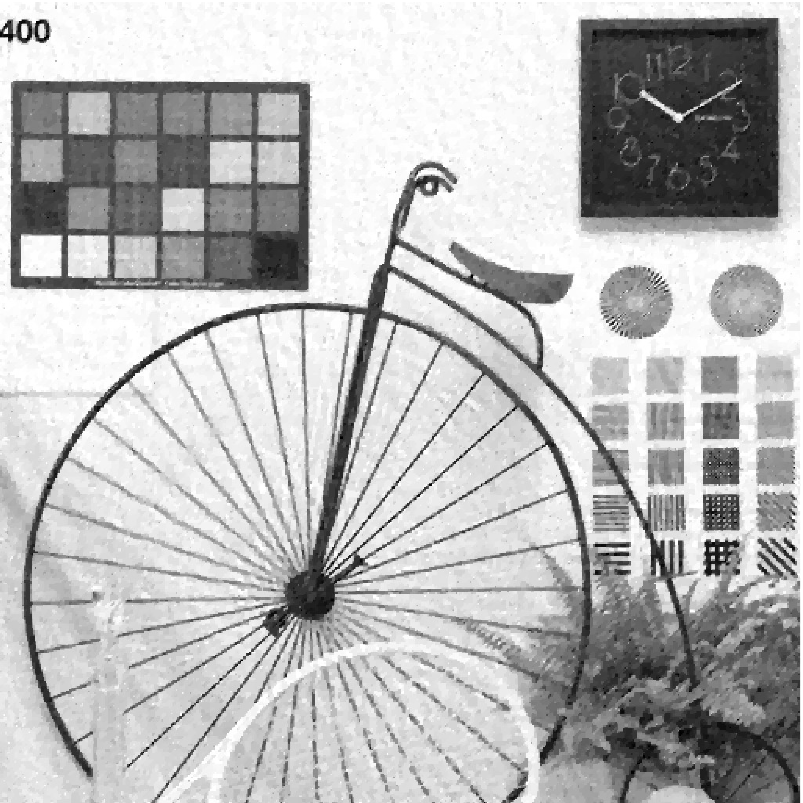}
     \caption*{L1-TV, PSNR: 21.70}
      \end{subfigure}

   \begin{subfigure}{0.3\linewidth}
            \includegraphics[width=\linewidth]{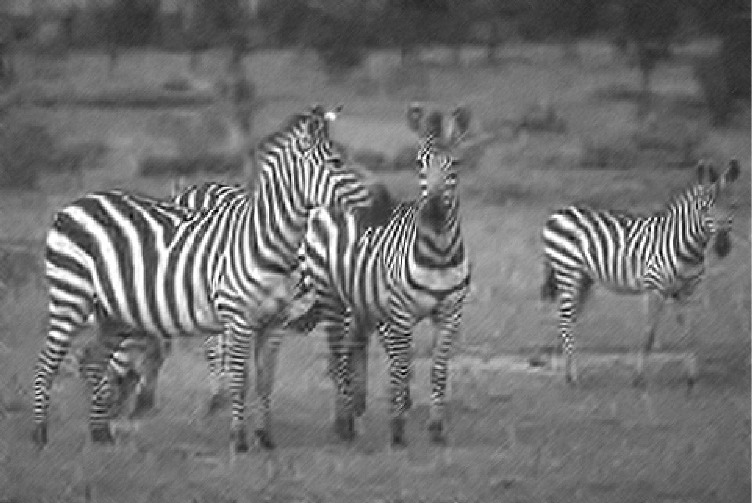}
            \caption*{L0-TF, PSNR: 24.23}
   \end{subfigure}
   \begin{subfigure}{0.3\linewidth}
            \includegraphics[width=\linewidth]{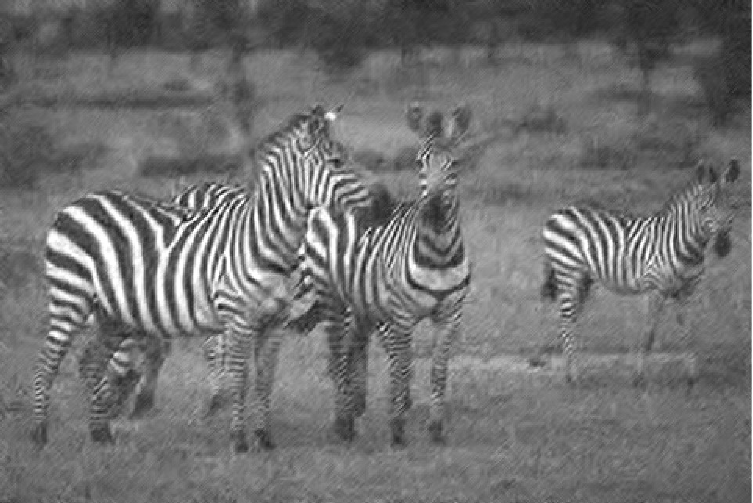}
            \caption*{ L1-TF, PSNR: 23.52}
    \end{subfigure}
     \begin{subfigure}{0.3\linewidth}
     \includegraphics[width=\linewidth]{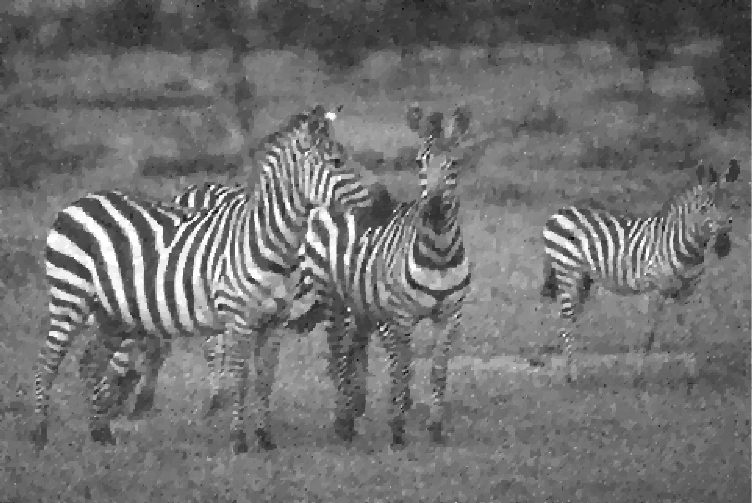}
     \caption*{L1-TV, PSNR: 23.01}
      \end{subfigure}

	\caption{Gaussian noise image deblurring: Deblurred images (airplane, bike, zebra) from corrupted images with MBL being 21. 
}
	\label{fig: TV-L0-L1-deblur part 2}
\end{figure} 

\begin{figure}
  \centering
   \begin{subfigure}{0.24\linewidth}
        \includegraphics[width=\linewidth]{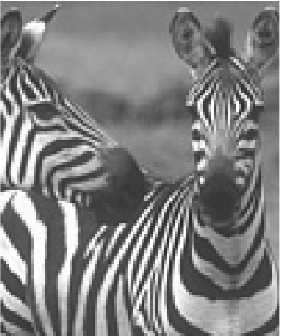}
        \caption*{\footnotesize{clean}}
   \end{subfigure}
   \begin{subfigure}{0.24\linewidth}
        \includegraphics[width=\linewidth]{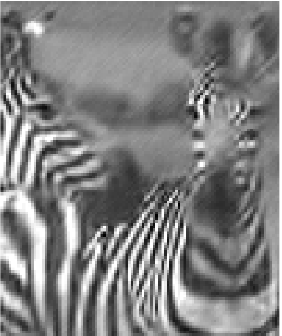}
        \caption*{PSNR: 21.47}
    \end{subfigure}
    \begin{subfigure}{0.24\linewidth}
         \includegraphics[width=\linewidth]{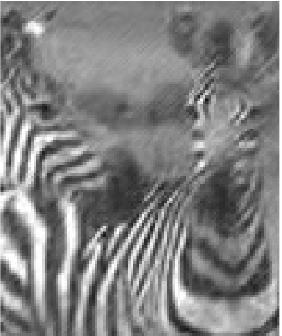}
        \caption*{PSNR: 20.87}
    \end{subfigure}
   \begin{subfigure}{0.24\linewidth}
        \includegraphics[width=\linewidth]{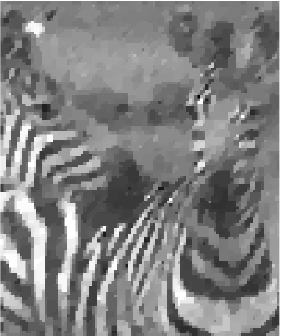}
        \caption*{ PSNR: 20.03}
   \end{subfigure}
   \begin{subfigure}{0.24\linewidth}
       \includegraphics[width=\linewidth]{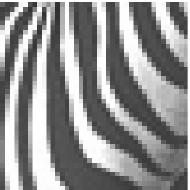}
        \caption*{\footnotesize{clean}}
   \end{subfigure}
   \begin{subfigure}{0.24\linewidth}
        \includegraphics[width=\linewidth]{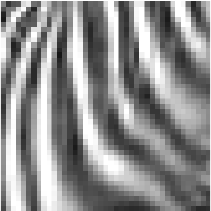}
        \caption*{PSNR: 19.35}
    \end{subfigure}
  \begin{subfigure}{0.24\linewidth}
       \includegraphics[width=\linewidth]{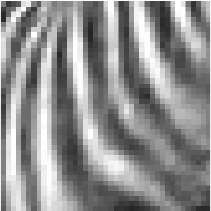}
       \caption*{PSNR: 18.04}
   \end{subfigure}
    \begin{subfigure}{0.24\linewidth}
        \includegraphics[width=\linewidth]{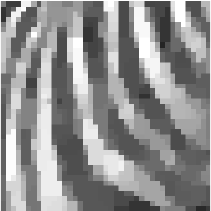}
        \caption*{PSNR: 17.83}
    \end{subfigure}     
	\caption{Gaussian noise image deblurring: Deblurred local zebra images from the corrupted images with MBL being 21. The second, third, and fourth columns correspond to the deblurred local zebra images obtained from L0-TF, L1-TF, and L1-TV, respectively.
}
	\label{fig: zebra head and neck}
\end{figure}

{
\subsubsection{Poisson noise image deblurring}
In this case, we choose the operation $\mathcal{P}$ as Poisson noise \cite{lefkimmiatis2013poisson,willett2010poisson,zheng2019sparsity}, and model \eqref{model: deblur problem} becomes 
\begin{equation}\label{model: poisson deblur problem}
\mathbf{x} := \text{Poisson}\left(\mathbf{K}\mathbf{v}\right)    
\end{equation}
where $\text{Poisson}(\mathbf{c})$ denotes a Poisson distributed random vector with mean $\mathbf{c}$. We generate a corrupted image in three steps: First, we scale a clean image to achieve a preset peak value, denoted as $\text{MAX}_{\mathbf{v}}$, which determines the level of Poisson noise. Second, we apply a motion blurring kernel to the scaled image. Third, we introduce Poisson noise to the blurred image. In this experiment, we set $\text{MAX}_{\mathbf{v}}$ to be 255. The resulting six corrupted images are shown in Figure \ref{fig:Blurred and poisson noise images}.
Given an observed corrupted image $\mathbf{x} \in \bR^{p}$, the function $\psi$ appearing in \eqref{model: l0} is chosen as
$$
\psi(\mathbf{z}):= \langle \mathbf{z}, \mathbf{1}\rangle - \langle \ln(\mathbf{z}), \mathbf{x} \rangle,  \ \ \mathbf{z} \in \bR^p.
$$
We specify the matrix $\mB$ appearing in both models \eqref{model: l0} and \eqref{model: l1} to be the motion blurring kernel matrix $\mK$. 
Models \eqref{model: l0} and \eqref{model: l1} with $\mD$ being the same tight framelet matrix constructed from discrete cosine transforms of size $7\times 7$ \cite{shen2016wavelet} are respectively referred to as ``L0-TF" and ``L1-TF". Model  \eqref{model: l1} with $\mD$ being the first order difference matrix \cite{micchelli2011proximity} is referred to as ``L1-TV".

The L0-TF model is solved by Algorithm \ref{algo: inexact FPPA l0}, while both the L1-TF and L1-TV models are solved by Algorithm \ref{algo: Inexact FPPA l1}. In all the three models, the parameter $\lambda$ is adjusted within the range $[0.001, 1]$. For the L0-TF model, we vary the parameter $\gamma$ within the range $[0.01, 10]$ and $p$ within the range $[0.0001, 0.1]$, while set $q$  to $(1+10^{-6}) \times 4 /p$ and $\alpha$ to $0.99$ in Algorithm \ref{algo: inexact FPPA l0}. For both the L1-TF and L1-TV models, the parameter $p_1 = p_2$ is varied within the range $[0.001, 1]$ and $q_2:=(1+10^{-6}) \times  4/p_2$, while $q_1:=(1+10^{-6})/p_1$ for L1-TF model and $q_1:=(1+10^{-6}) \times 8 /p_1$ for L1-TV model.
Additionally, we set $e^{k+1}:= M/k^{1.01}$ in both  algorithms with $M=10^{8}$. 
The stopping criterion for the two algorithms is when \eqref{eq: stop} is satisfied with $\mathbf{x}^k:=  \mathbf{\tilde v}^k$ and $\mathrm{TOL}:=10^{-5}$, or when the maximum iteration counts of 2000 is attained.

We show in Table \ref{tab: TV-poissonnoise-elL1-TF-noise-elL0-TF-deblurred} the highest PSNR values for each of the reconstructed images, with the corresponding values of parameters $\lambda$, $\gamma$, $p$ and $p_1$ listed in Table  \ref{tab: best parameter poisson noise}.
%
The results in Table \ref{tab: TV-poissonnoise-elL1-TF-noise-elL0-TF-deblurred} demonstrate that the L0-TF model achieves the highest PSNR-value, surpassing the second-highest value by approximately $0.6$ dB expect for the `barbara' image, where the PSNR-values of the L0-TF model are comparable with those of the L1-TF model. For visual comparison, we display the deblurred images in  Figures \ref{fig: Poisson-TV-L0-L1-deblur part 1} and \ref{fig: Poisson-TV-L0-L1-deblur part 2}.  One can see that the L0-TF model is more powerful to suppress noise on smooth areas compared with the L1-TF and L1-TV models.  In Figure \ref{fig: poisson zebra head and neck}, we compare performance of the three models on two specific local areas in the zebra image and observe that the images restored  from the L0-TF model preserve finer details and edges more effectively than those from the other two models.


}

 \begin{figure}
  \centering
   \begin{subfigure}{0.3\linewidth}
            \includegraphics[width=\linewidth]{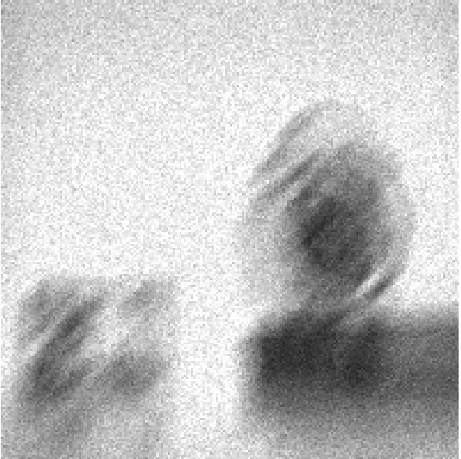}
   \end{subfigure}
   \begin{subfigure}{0.3\linewidth}
            \includegraphics[width=\linewidth]{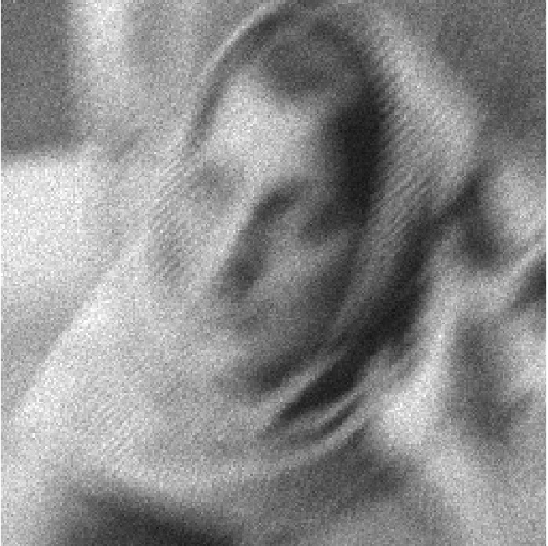}
    \end{subfigure}
     \begin{subfigure}{0.3\linewidth}
     \includegraphics[width=\linewidth]{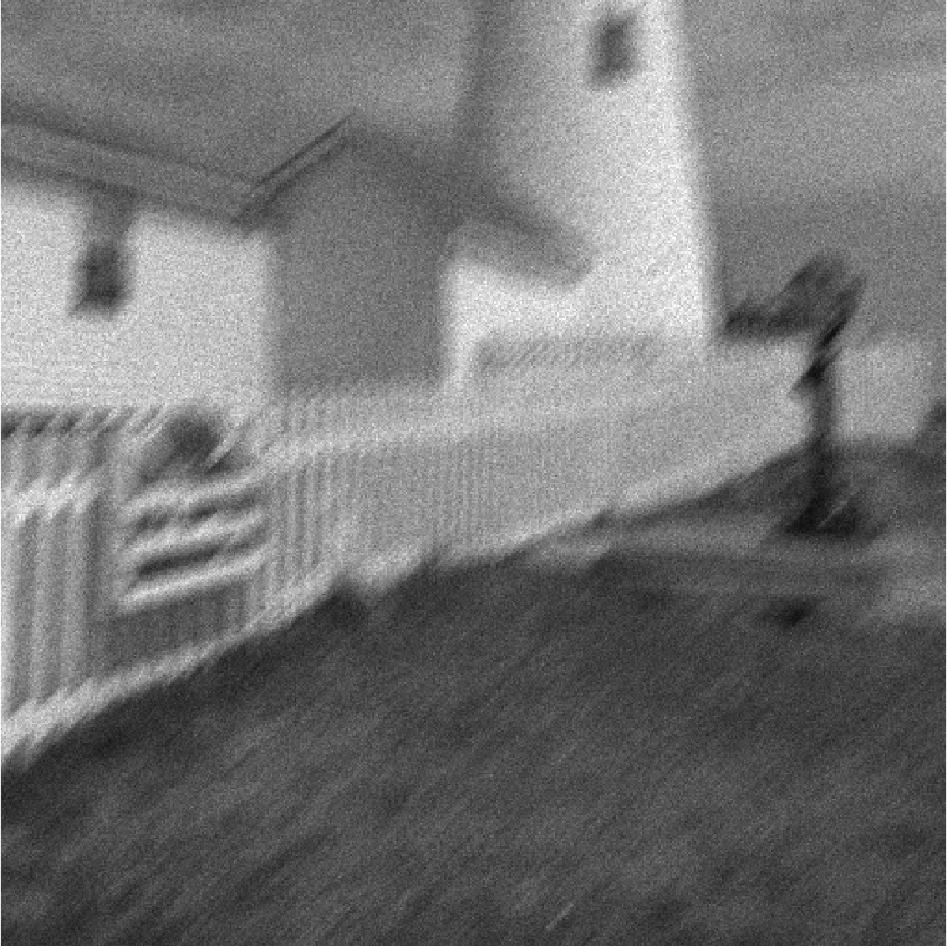}
      \end{subfigure}
   \begin{subfigure}{0.3\linewidth}
            \includegraphics[width=\linewidth]{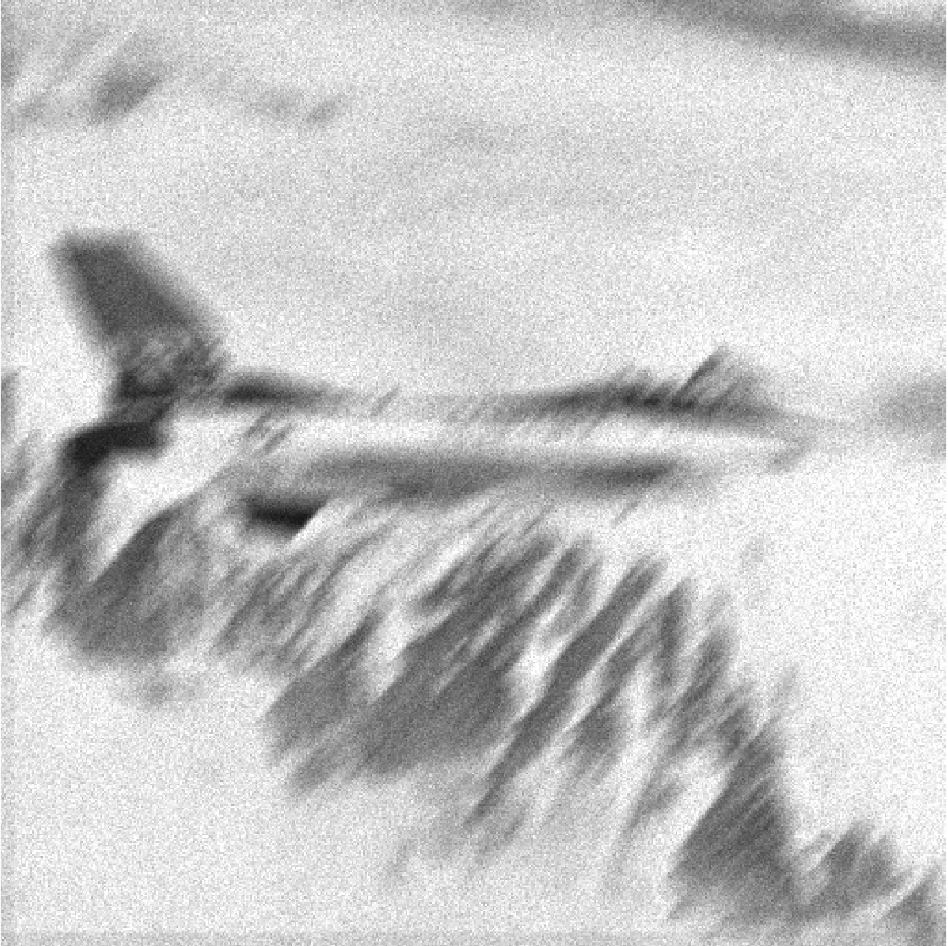}
   \end{subfigure}
   \begin{subfigure}{0.3\linewidth}
            \includegraphics[width=\linewidth]{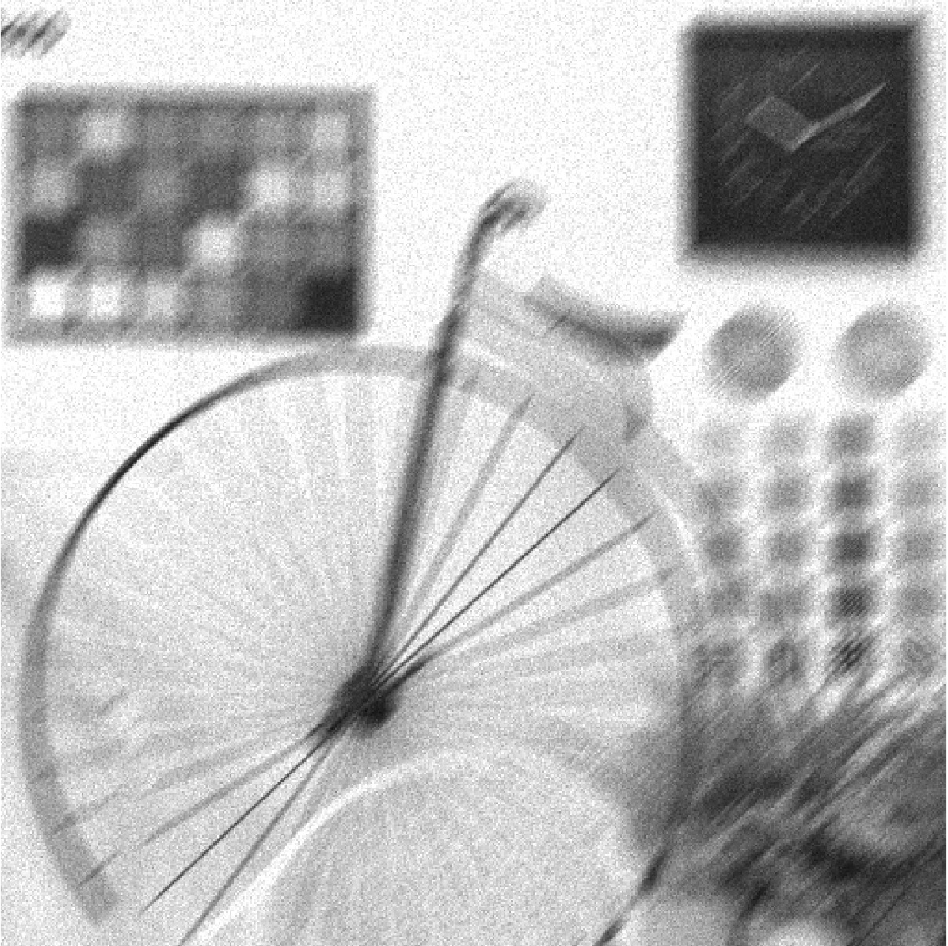}
    \end{subfigure}
     \begin{subfigure}{0.3\linewidth}
     \includegraphics[width=\linewidth]{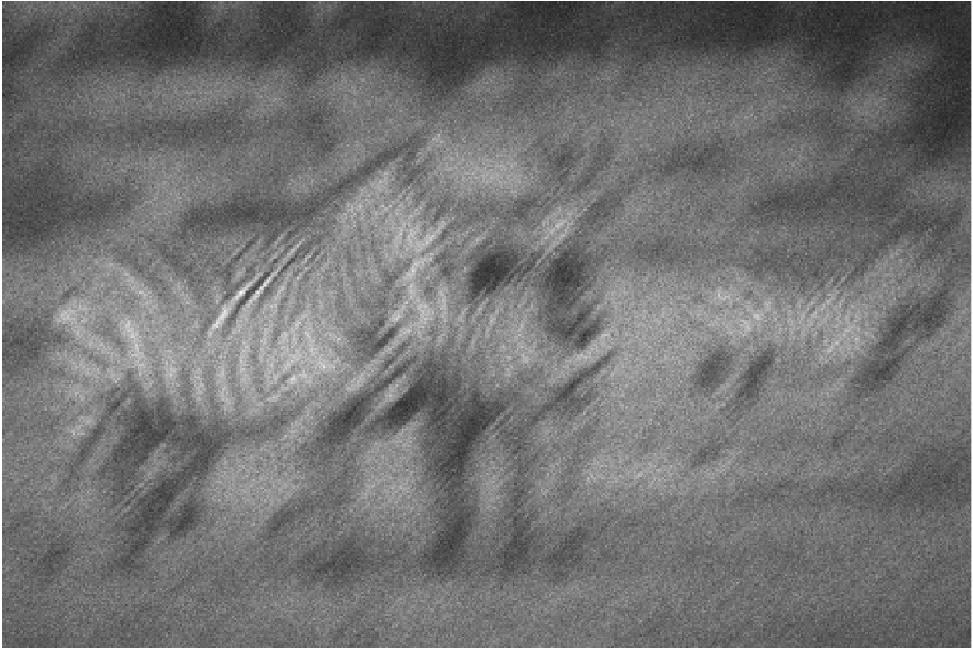}
      \end{subfigure}
	\caption{Corrupted images with MBL being 21 and Poisson noise.
}
	\label{fig:Blurred and poisson noise images}
\end{figure}

\begin{small}

\begin{table}
    \caption{Poisson noise image deblurring:
    PSNB values (dB) of L0-TF, L1-TF, and L1-TV} 
    \label{tab: TV-poissonnoise-elL1-TF-noise-elL0-TF-deblurred}
    \centering
    \begin{tabular}{c|ccc|ccc}
    \hline
    &
 
    \multicolumn{3}{c|}{clock $200 \times 200$} &     \multicolumn{3}{c}{barbara $256 \times 256$} 
    \\
    \hline 
    \backslashbox[1cm]{model}{\footnotesize{MBL}}
    & 9 & 15 & 21 &  9 & 15 & 21
    \\
    
    \hline
    L0-TF & \textbf{26.52}  &\textbf{25.61}  & \textbf{24.72} & \textbf{24.00} & 22.80 & 22.30 \\
    L1-TF & 25.87 & 24.99 & 23.83&23.97 & \textbf{22.85} &  \textbf{22.36}\\
    L1-TV  &  25.26  & 24.41  & 23.28  &23.41 & 22.53 & 21.98\\
    \hline 
&  \multicolumn{3}{c|}{ lighthouse $512 \times 512$}&

     \multicolumn{3}{c}{airplane $512 \times 512$}
    \\
    \hline 
    \backslashbox[1cm]{model}{\footnotesize{MBL}}
    & 9 & 15 & 21& 9 & 15 & 21
    \\
    \hline
    L0-TF & \textbf{23.84} & \textbf{22.85} & \textbf{22.23}
    & \textbf{27.63}& \textbf{26.34}  & \textbf{25.30}\\
    L1-TF 
    & 23.25 & 22.49 & 21.86
    &26.99 & 25.70 & 24.72 \\
    L1-TV  
    &  22.92 & 22.25 & 21.66
    & 26.31 & 25.07 & 24.21\\
\hline
&

    \multicolumn{3}{c|}{ bike $512 \times 512$}  &
    \multicolumn{3}{c}{ zebra $321 \times 481$} 
 
    \\
    \hline 
    \backslashbox[1cm]{model}{\footnotesize{MBL}}
    & 9 & 15 & 21& 9 & 15 & 21
    \\
    \hline
    L0-TF & \textbf{20.52}& \textbf{19.39}&  \textbf{18.77}&  
    \textbf{22.43} &  \textbf{21.30} &  \textbf{20.65} \\
     L1-TF 
    
    &  20.03&  19.01 &  18.43 
    
    &  21.70&  20.64 &  20.07 \\   
    L1-TV  &  19.84&  18.88& 18.20   
    
    &  21.23&  20.30& 19.80 

    \\
\hline
    \end{tabular}

\end{table}

\begin{table}
    \caption{
    Poisson noise image deblurring: chosen values of the parameters} 
    \label{tab: best parameter poisson noise}
    \centering
    \begin{tabular}{cc|ccc|cc|cc}
    \hline
    & &
    \multicolumn{3}{c|}{L0-TF} &     \multicolumn{2}{c|}{ L1-TF} & \multicolumn{2}{c}{L1-TV}     
    \\
\hline
image & MBL & $\lambda$ & $\gamma$ & $p$ & $\lambda $  & $p_1$ &  $\lambda $  & $p_1$
    \\
    \hline
     \multirow{3}{*}{clock}& 9& $1\times 10^{-2}$ & $7\times 10^{-1}$ & $5\times 10^{-3}$ & $3\times 10^{-3}$ & $2\times 10^{-2}$& $1\times 10^{-2}$& $4\times 10^{-2}$\\
    &15& $1 \times 10^{-2}$& $9 \times 10^{-1}$& $5\times 10^{-4}$ & $2\times 10^{-3}$& $1\times 10^{-2}$& $9\times 10^{-3}$ & $2\times 10^{-2}$\\
        &21& $9 \times 10^{-3}$& $7\times 10^{-1} $& $1\times 10^{-4}$ & $2\times 10^{-3}$ &  $1\times 10^{-2}$& $7\times 10^{-3}$ & $2\times 10^{-2}$\\
    \hline
    \multirow{3}{*}{barbara}& 9& $1\times 10^{-2}$ & $2\times 10^{0}$& $1\times 10^{-1}$& $2\times 10^{-3}$& $8\times 10^{-3}$& $7\times 10^{-3}$ & $3\times 10^{-1}$\\
    &15& $1\times 10^{-2}$ & $2\times 10^{0}$ & $1\times 10^{-1}$& $2\times 10^{-3}$& $8\times 10^{-3}$& $1\times 10^{-2}$ & $1\times 10^{-1}$\\
    &21& $1\times 10^{-2}$ & $2\times 10^{0}$ & $5\times 10^{-2}$ & $2\times 10^{-3}$ & $8\times 10^{-3}$ & $1\times 10^{-2}$ & $5\times 10^{-2}$ \\
    \hline
    \multirow{3}{*}{lighthouse}& 9& $1\times 10^{-2}$ & $2\times 10^0$ & $1\times 10^{-2}$ & $2\times 10^{-3}$ & $9\times 10^{-3}$ & $9\times 10^{-3}$ & $2\times 10^{-2}$ \\
    &15& $1\times 10^{-2}$ & $2\times 10^0$ & $1\times 10^{-2}$ & $2\times 10^{-3}$ & $8 \times 10^{-3}$& $9\times 10^{-3}$ & $2\times 10^{-2}$\\
    &21& $1\times 10^{-2}$ & $3\times 10^{0}$ & $1\times 10^{-4}$ & $2\times 10^{-3}$& $8\times 10^{-3}$ & $9\times 10^{-3}$ & $1\times 10^{-2}$\\
    \hline
    \multirow{3}{*}{airplane}& 9& $1\times 10^{-2}$ &  $7\times 10^{-1}$& $5\times 10^{-3}$  & $3\times 10^{-3}$& $9\times 10^{-3}$ & $1\times 10^{-2}$& $6\times 10^{-2}$\\
    &15& $1\times 10^{-2}$ & $1\times 10^0$ & $1\times 10^{-3}$ & $2\times 10^{-3}$ & $1\times 10^{-2}$ & $1\times 10^{-2}$ & $3\times 10^{-2}$\\
   &21& $1.3 \times 10^{-2}$ & $2.5\times 10^0$&  $1\times 10^{-4}$& $2\times 10^{-3}$ & $1\times 10^{-2}$ & $9\times 10^{-3}$ & $2\times 10^{-2}$\\
    \hline
    \multirow{3}{*}{bike}& 9& $1\times 10^{-2}$ & $2\times 10^0$ & $1\times 10^{-2}$ & $2\times 10^{-3}$ & $7\times 10^{-3}$ & $7\times 10^{-3}$ & $2\times 10^{-2}$\\
    &15& $1\times 10^{-2}$ & $3\times 10^{0}$ & $1\times 10^{-2}$& $2\times 10^{-3}$& $9\times 10^{-3}$& $7\times 10^{-3}$ & $1\times 10^{-2}$\\
    &21& $9\times 10^{-3}$& $4\times 10^0$& $5\times 10^{-3}$ & $1\times 10^{-3}$ & $5\times 10^{-3}$& $7\times 10^{-3}$ & $1\times 10^{-2}$\\
    \hline
    \multirow{3}{*}{zebra}& 9& $1\times 10^{-2}$& $1\times 10^0$ & $1\times 10^{-2}$& $2\times 10^{-3}$ & $1\times 10^{-2}$& $9\times 10^{-3}$ & $3\times 10^{-2}$\\
    &15& $1\times 10^{-2}$& $2.5\times 10^0$ & $1\times 10^{-2}$& $2\times 10^{-3}$ & $1\times 10^{-2}$ & $7\times 10^{-3}$ & $2\times 10^{-2}$\\
    &21& $1\times 10^{-2}$& $3\times 10^0$ & $5\times 10^{-3}$& $2\times 10^{-3}$ & $1\times 10^{-2}$& $7\times 10^{-3}$ & $1\times 10^{-2}$\\
    \hline
\end{tabular}
\end{table}
\end{small}

\begin{figure}
  \centering
   \begin{subfigure}{0.3\linewidth}
            \includegraphics[width=\linewidth]{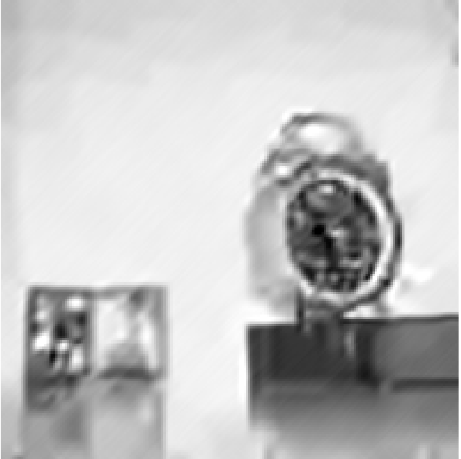}
            \caption*{L0-TF, PSNR: 24.72}
   \end{subfigure}
   \begin{subfigure}{0.3\linewidth}
            \includegraphics[width=\linewidth]{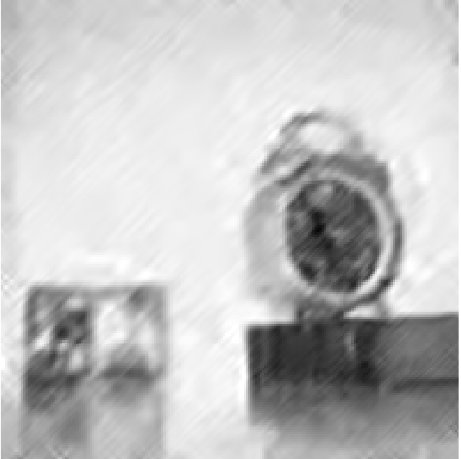}
            \caption*{L1-TF, PSNR: 23.83}
    \end{subfigure}
     \begin{subfigure}{0.3\linewidth}
     \includegraphics[width=\linewidth]{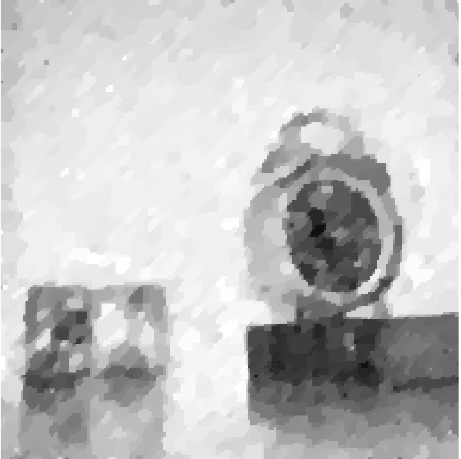}
     \caption*{L1-TV, PSNR: 23.28}
      \end{subfigure}
   \begin{subfigure}{0.3\linewidth}
   
            \includegraphics[width=\linewidth]{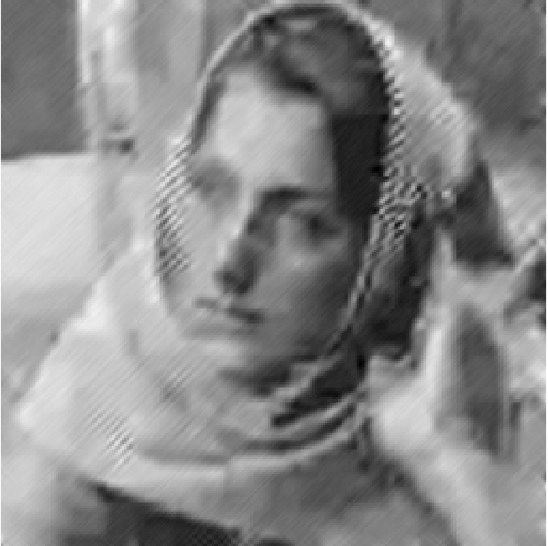}
            \caption*{ L0-TF, PSNR: 22.30}
   \end{subfigure}
   \begin{subfigure}{0.3\linewidth}
            \includegraphics[width=\linewidth]{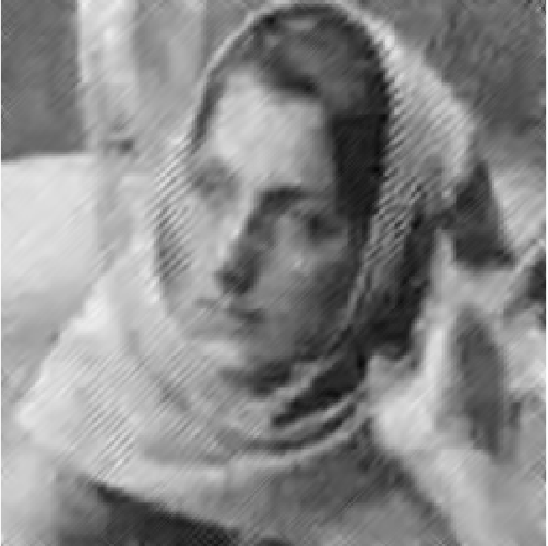}
            \caption*{L1-TF, PSNR: 22.36}
    \end{subfigure}
     \begin{subfigure}{0.3\linewidth}
     \includegraphics[width=\linewidth]{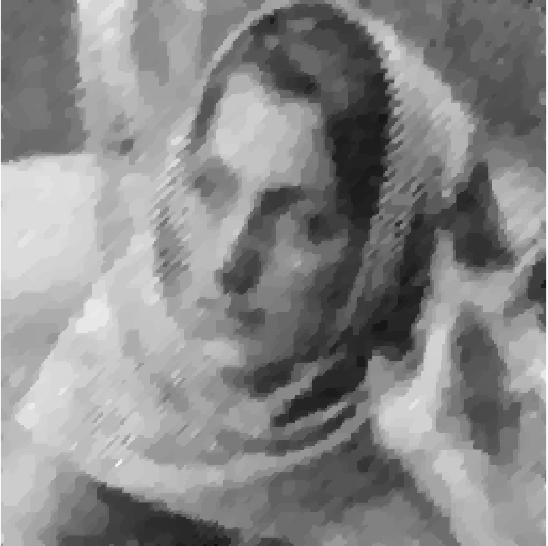}
     \caption*{L1-TV, PSNR: 21.98}
      \end{subfigure}

   \begin{subfigure}{0.3\linewidth}
            \includegraphics[width=\linewidth]{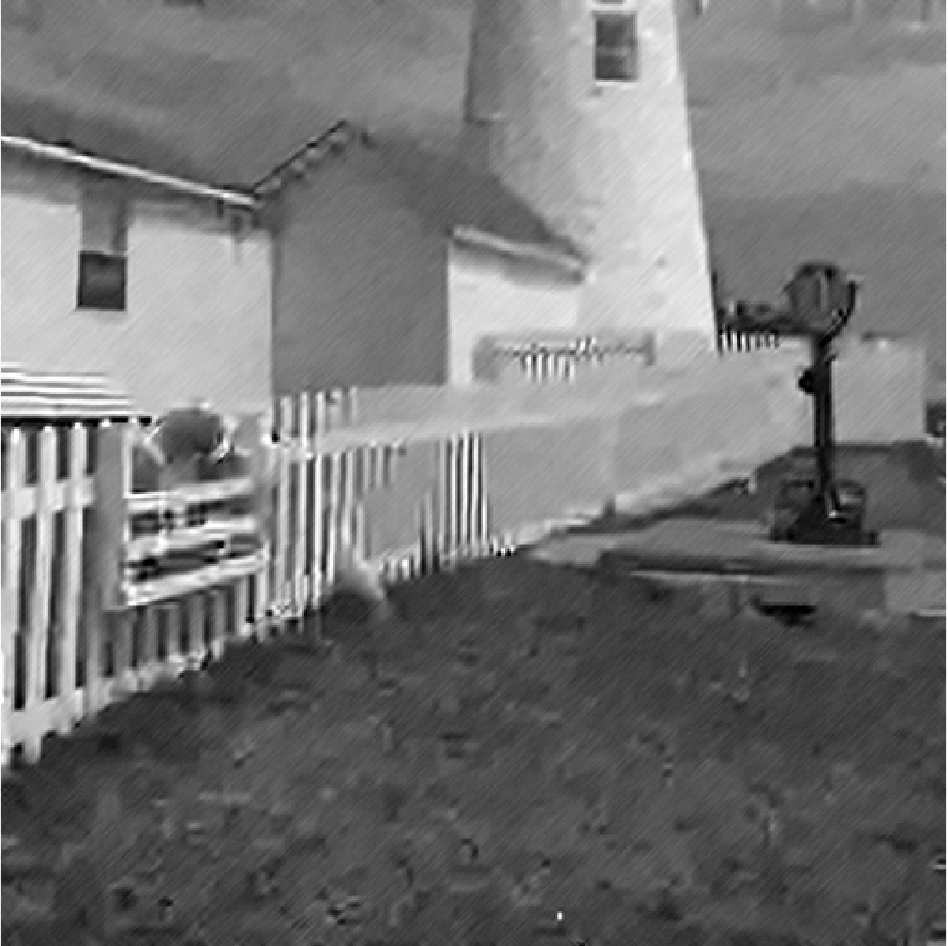}
            \caption*{L0-TF, PSNR: 22.23}
   \end{subfigure}
   \begin{subfigure}{0.3\linewidth}
            \includegraphics[width=\linewidth]{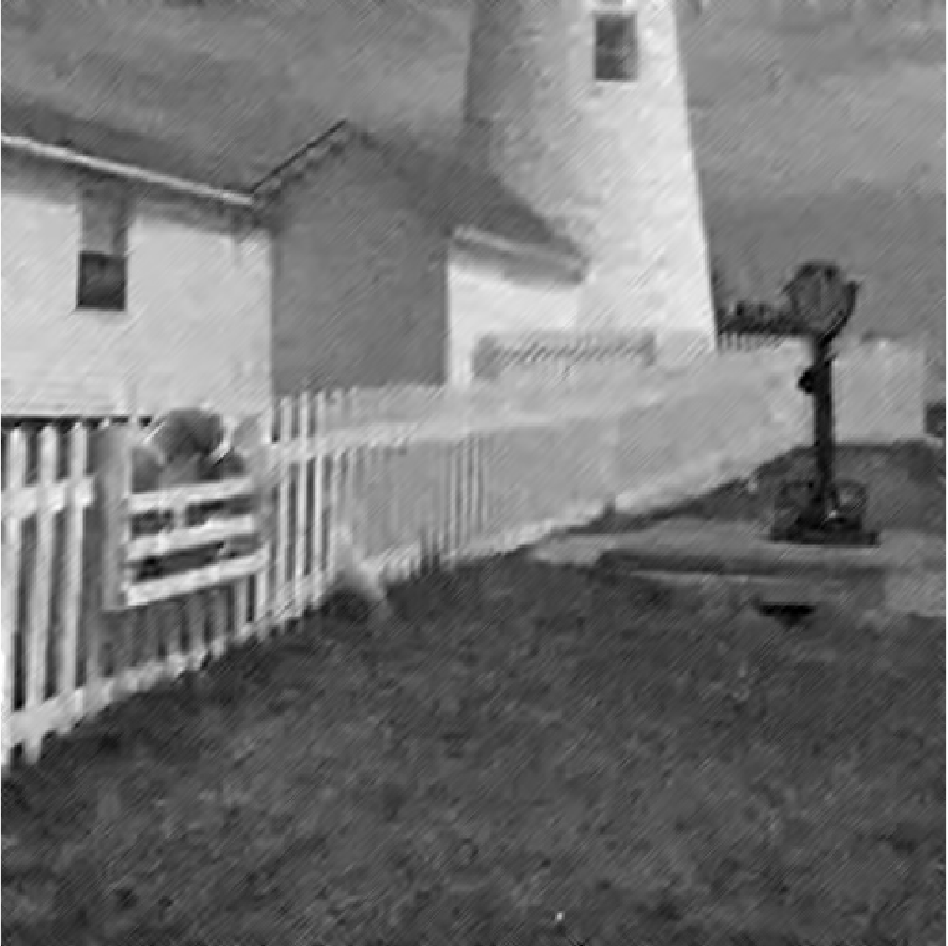}
            \caption*{ L1-TF, PSNR: 21.86}
    \end{subfigure}
     \begin{subfigure}{0.3\linewidth}
     \includegraphics[width=\linewidth]{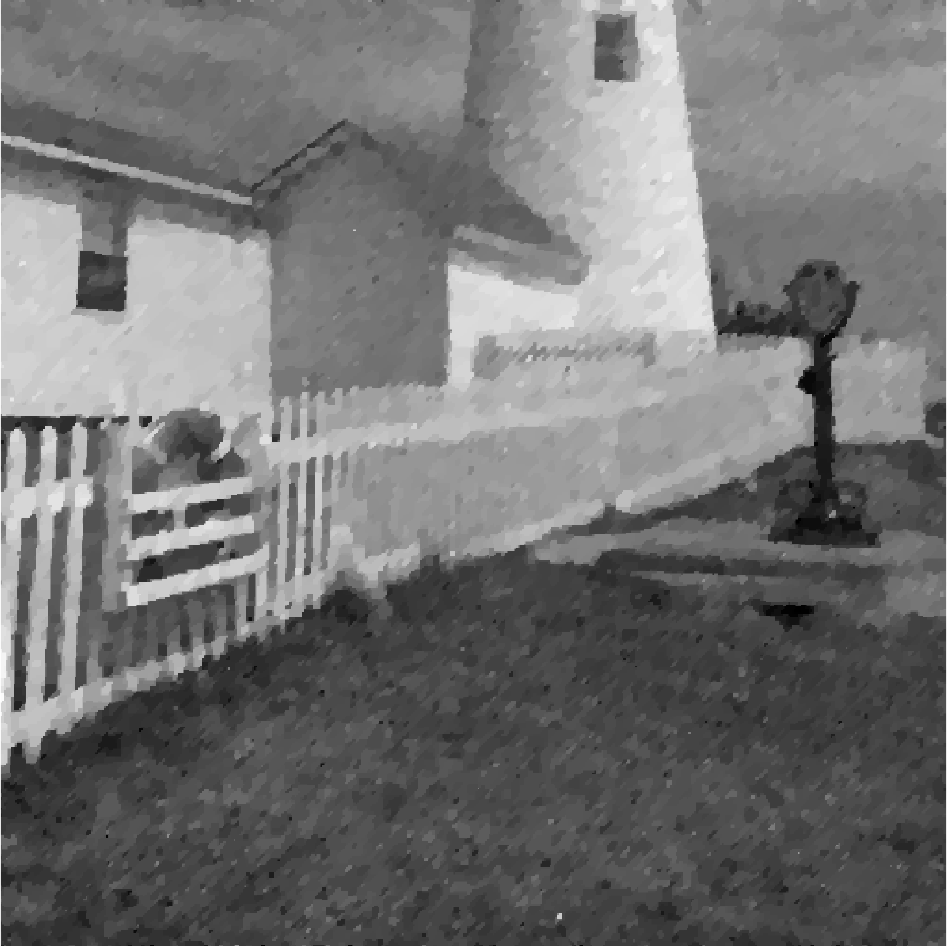}
     \caption*{L1-TV, PSNR: 21.66}
      \end{subfigure}

	\caption{Poisson noise image deblurring:  Deblurred images (clock, barbara, lighthouse) with MBL being 21.  
}
	\label{fig: Poisson-TV-L0-L1-deblur part 1}
\end{figure}

\begin{figure}
  \centering
   \begin{subfigure}{0.3\linewidth}
            \includegraphics[width=\linewidth]{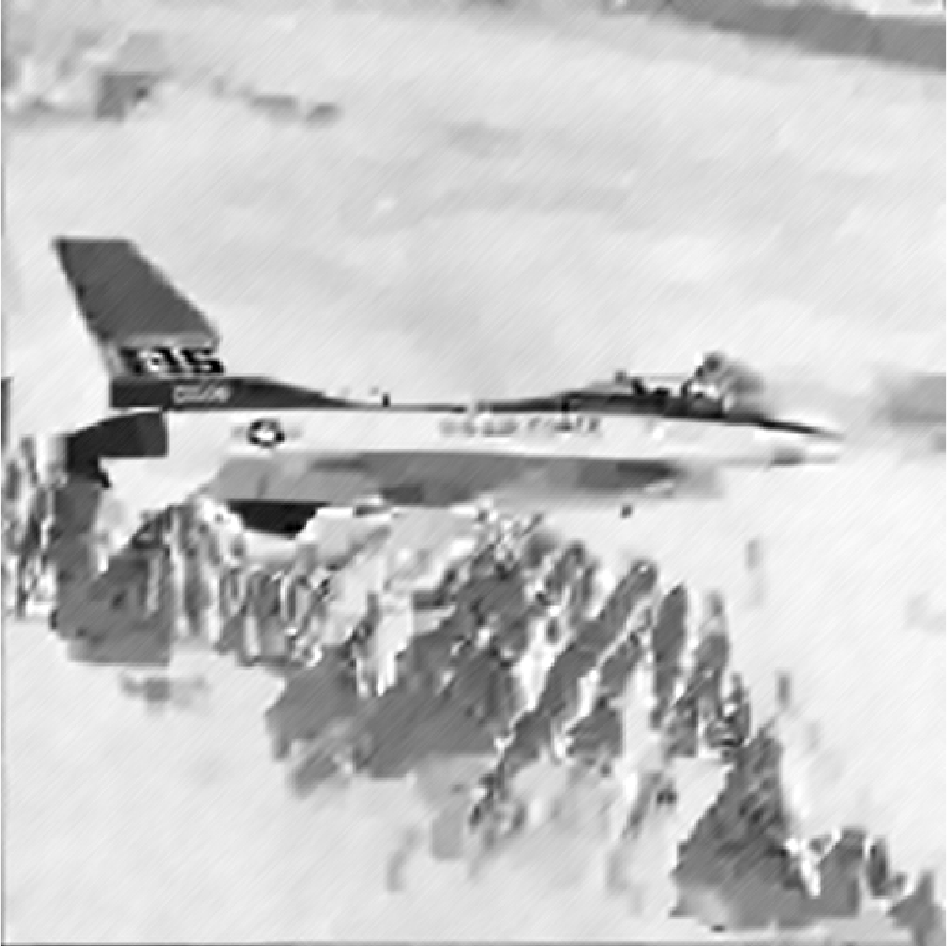}
            \caption*{L0-TF, PSNR: 25.30}
   \end{subfigure}
   \begin{subfigure}{0.3\linewidth}
            \includegraphics[width=\linewidth]{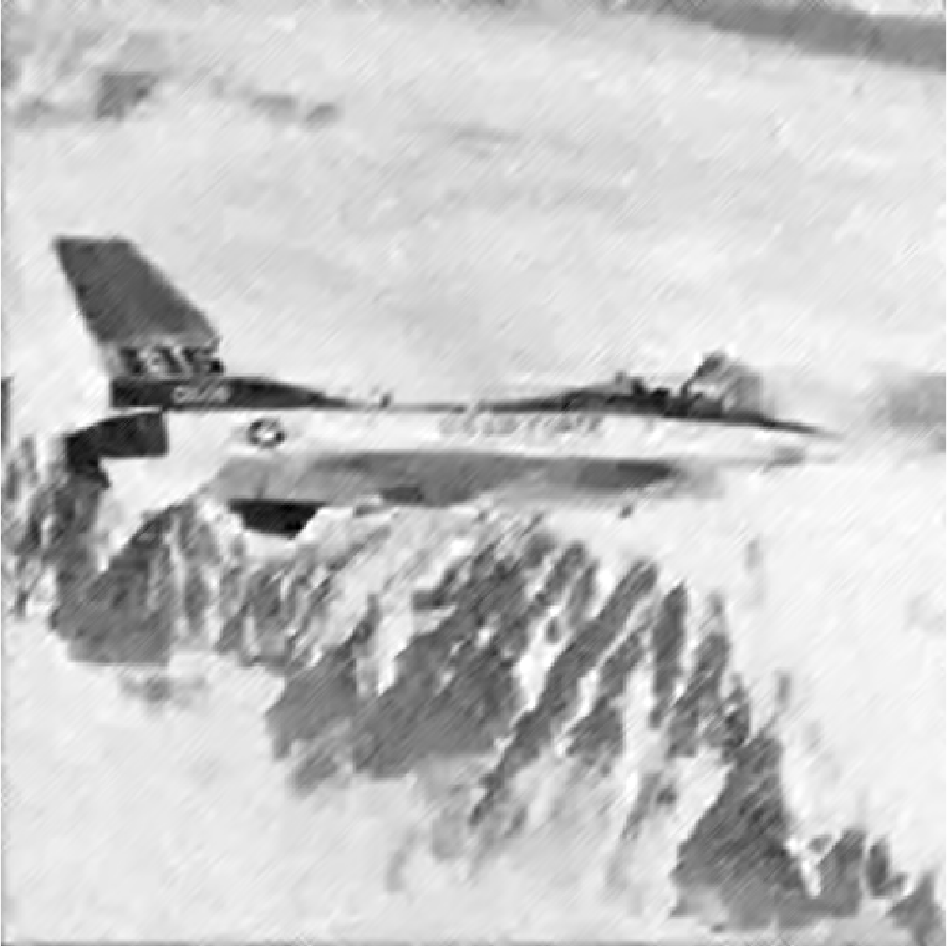}
            \caption*{L1-TF, PSNR: 24.72}
    \end{subfigure}
     \begin{subfigure}{0.3\linewidth}
     \includegraphics[width=\linewidth]{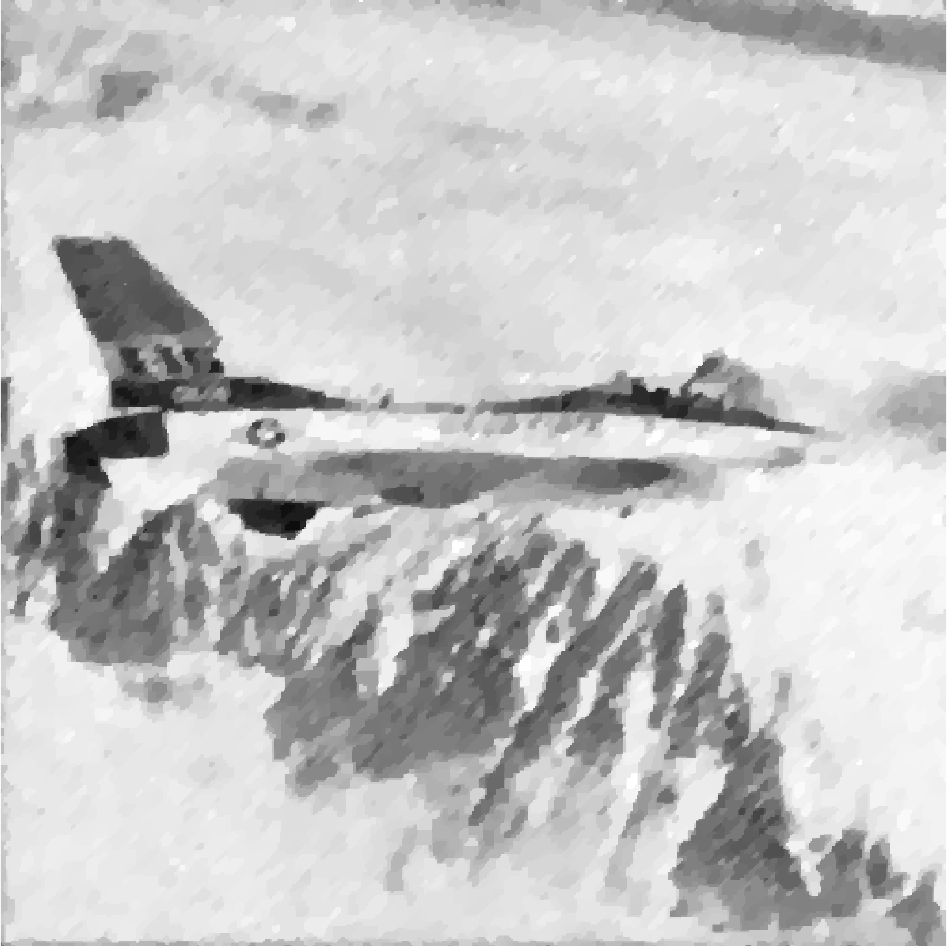}
     \caption*{L1-TV, PSNR: 24.21}
      \end{subfigure}
   \begin{subfigure}{0.3\linewidth}
   
            \includegraphics[width=\linewidth]{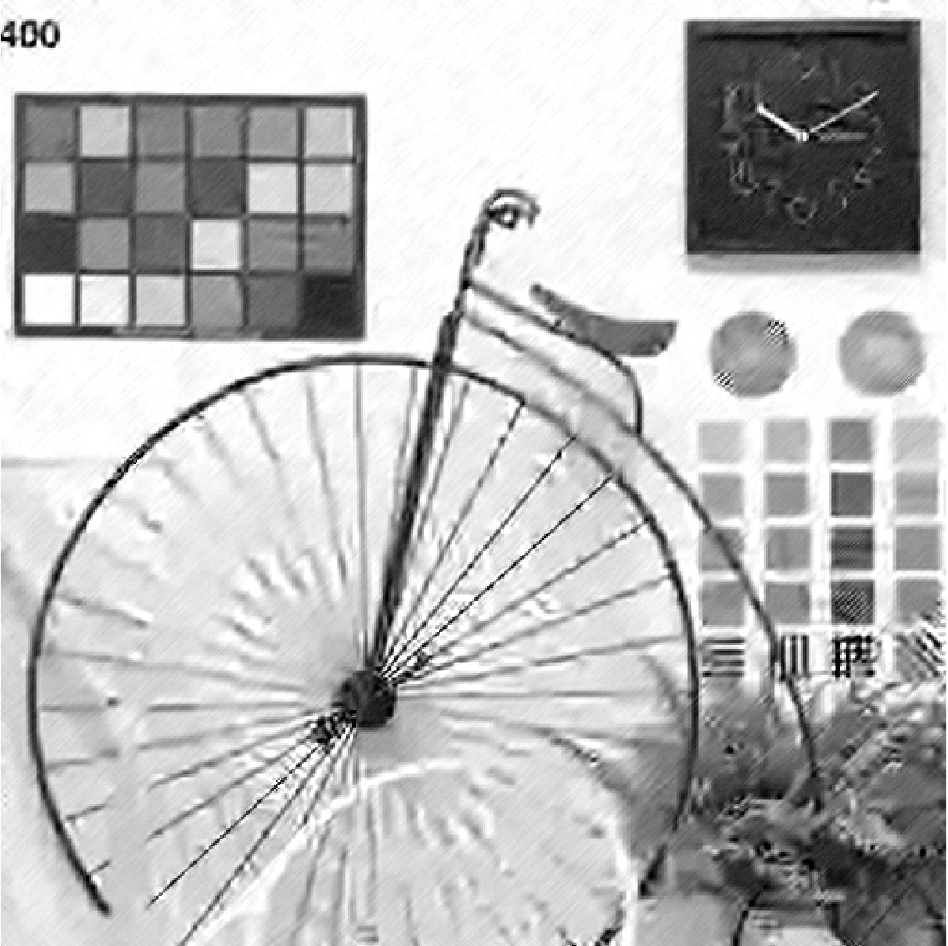}
            \caption*{ L0-TF, PSNR: 18.77}
   \end{subfigure}
   \begin{subfigure}{0.3\linewidth}
            \includegraphics[width=\linewidth]{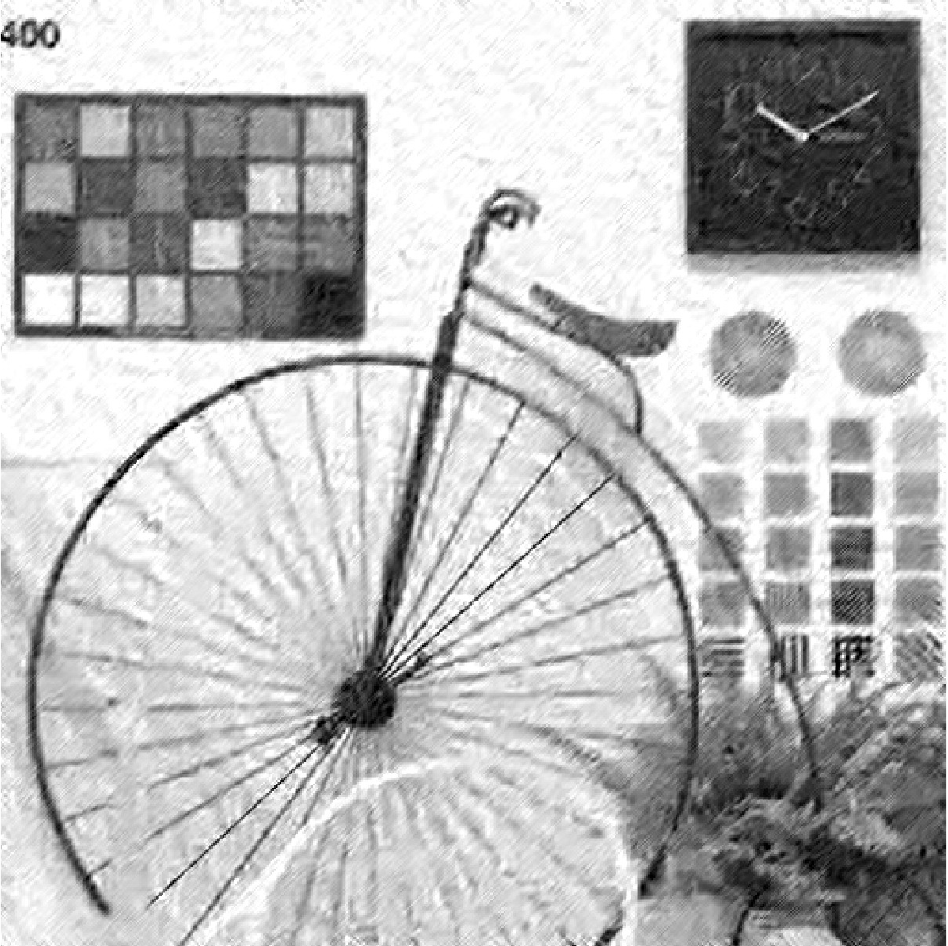}
            \caption*{L1-TF, PSNR: 18.43}
    \end{subfigure}
     \begin{subfigure}{0.3\linewidth}
     \includegraphics[width=\linewidth]{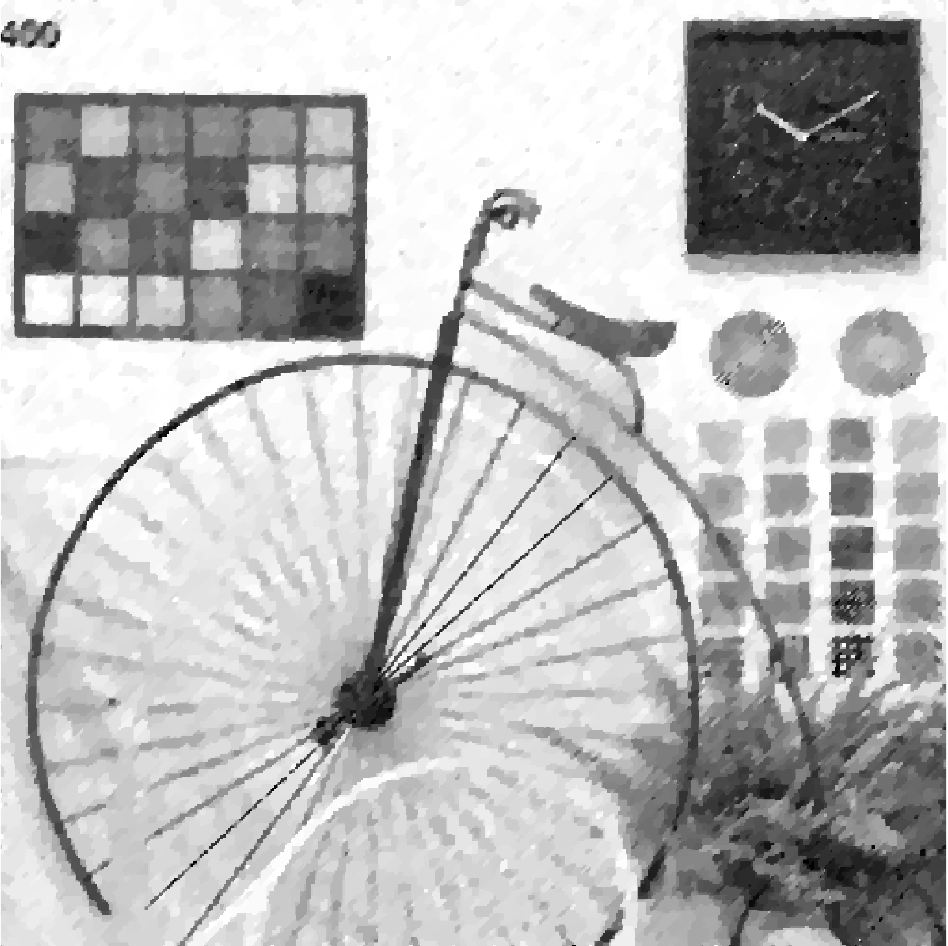}
     \caption*{L1-TV, PSNR: 18.20}
      \end{subfigure}

   \begin{subfigure}{0.3\linewidth}
            \includegraphics[width=\linewidth]{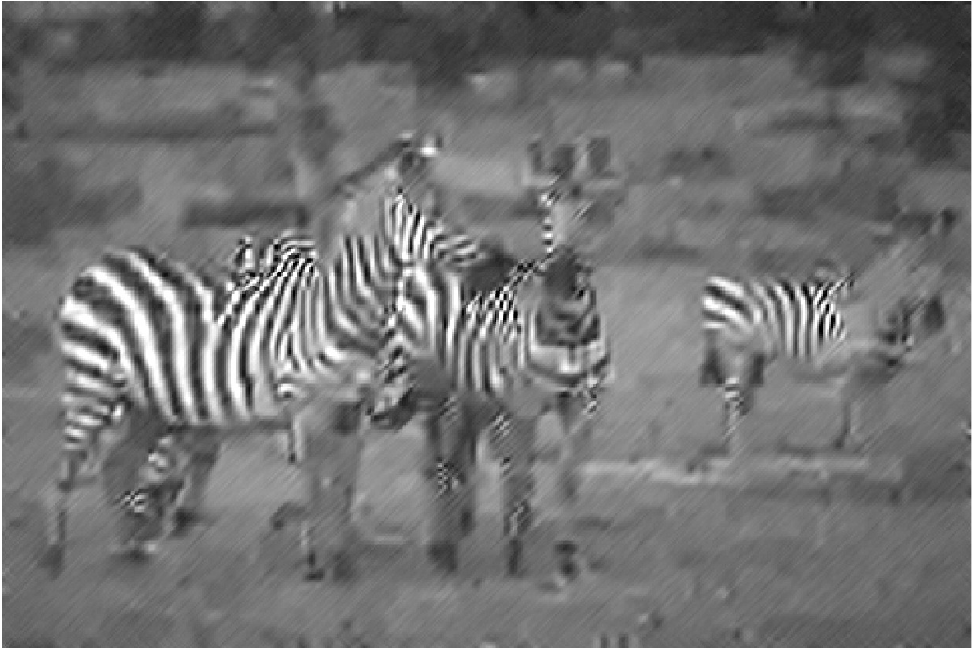}
            \caption*{L0-TF, PSNR: 20.65}
   \end{subfigure}
   \begin{subfigure}{0.3\linewidth}
            \includegraphics[width=\linewidth]{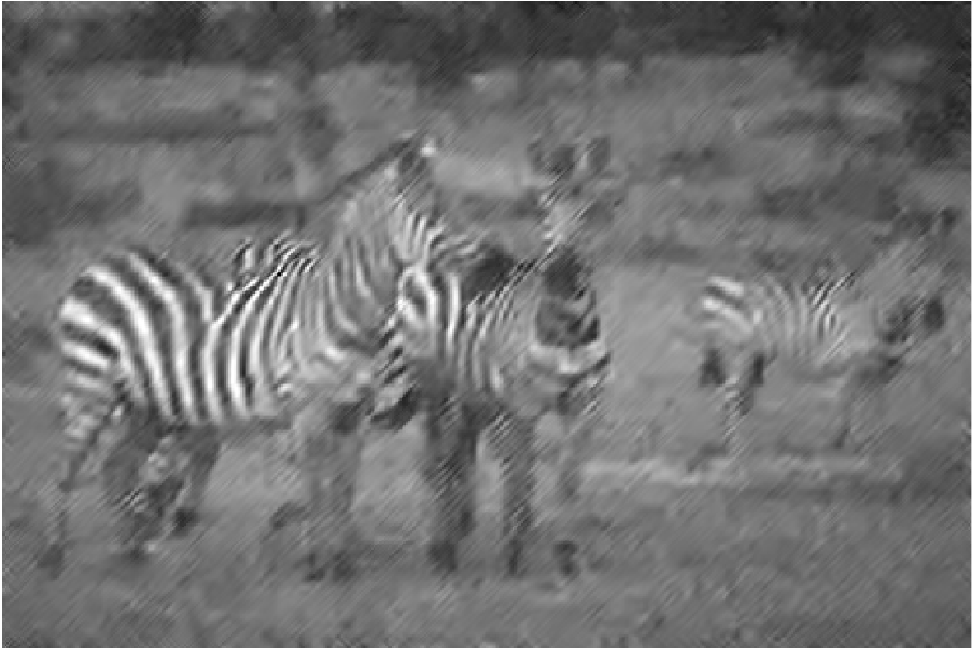}
            \caption*{ L1-TF, PSNR: 20.07}
    \end{subfigure}
     \begin{subfigure}{0.3\linewidth}
     \includegraphics[width=\linewidth]{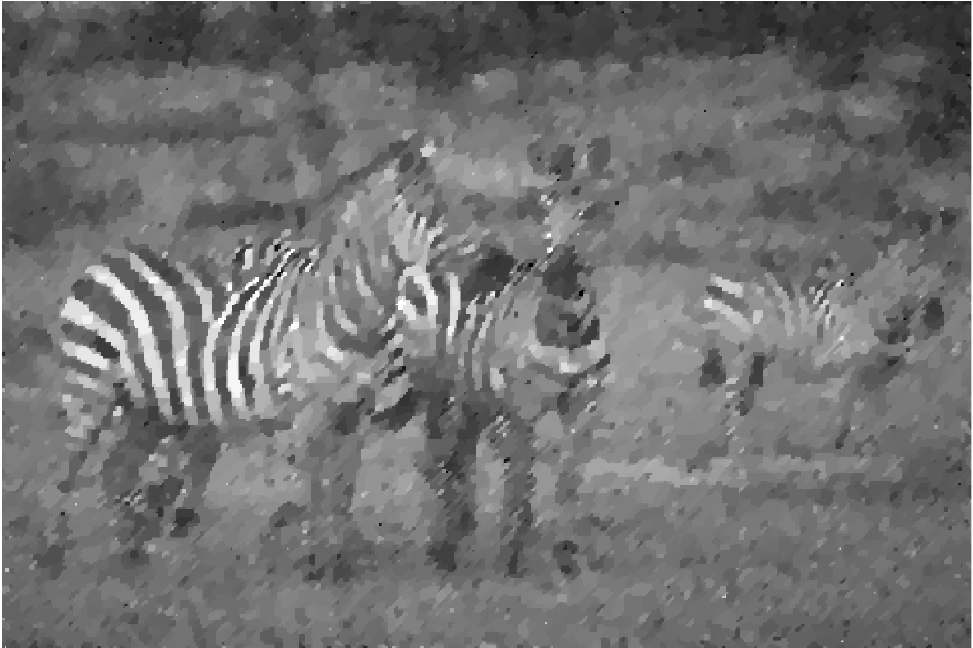}
     \caption*{L1-TV, PSNR: 19.80}
      \end{subfigure}

	\caption{Poisson noise image deblurring:  Deblurred images (airplane, bike, zebra) with MBL being 21. 
}
	\label{fig: Poisson-TV-L0-L1-deblur part 2}
\end{figure}

\begin{figure}
  \centering
   \begin{subfigure}{0.24\linewidth}
        \includegraphics[width=\linewidth]{fig/Clean_Zebra_head.pdf}
        \caption*{\footnotesize{clean}}
   \end{subfigure}
   \begin{subfigure}{0.24\linewidth}
        \includegraphics[width=\linewidth]{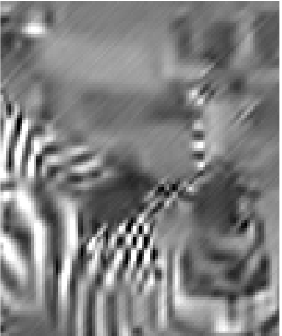}
        \caption*{PSNR: 17.55 }
    \end{subfigure}
    \begin{subfigure}{0.24\linewidth}
         \includegraphics[width=\linewidth]{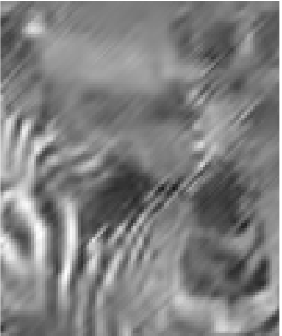}
        \caption*{PSNR: 17.00}
    \end{subfigure}
   \begin{subfigure}{0.24\linewidth}
        \includegraphics[width=\linewidth]{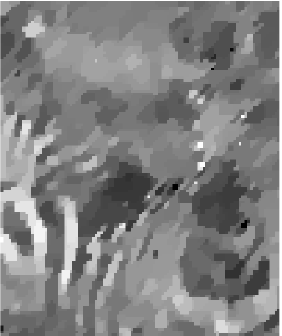}
        \caption*{ PSNR: 16.69}
   \end{subfigure}
   \begin{subfigure}{0.24\linewidth}
       \includegraphics[width=\linewidth]{fig/Clean_Zebra_back.pdf}
        \caption*{\footnotesize{clean}}
   \end{subfigure}
   \begin{subfigure}{0.24\linewidth}
        \includegraphics[width=\linewidth]{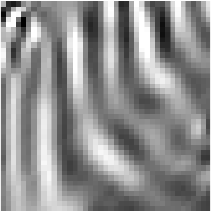}
        \caption*{PSNR: 14.71}
    \end{subfigure}
  \begin{subfigure}{0.24\linewidth}
       \includegraphics[width=\linewidth]{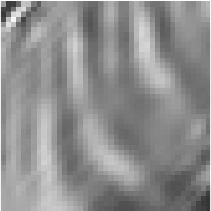}
       \caption*{PSNR: 12.90}
   \end{subfigure}
    \begin{subfigure}{0.24\linewidth}
        \includegraphics[width=\linewidth]{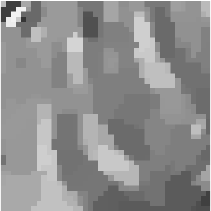}
        \caption*{PSNR: 12.63}
    \end{subfigure}     
	\caption{Poisson noise image deblurring: Deblurred local zebra images from the corrupted images with MBL being 21. The second, third, and fourth columns correspond to the restored local zebra images obtained from L0-TF, L1-TF, and L1-TV, respectively.
}
	\label{fig: poisson zebra head and neck}
\end{figure} 

In summary, the numerical experiments presented above all confirm the convergence of the proposed inexact algorithm and demonstrate the superiority of the $\ell_0$ regularization over the $\ell_1$ regularization, in both approximation accuracy and sparsity.
The numerical results also confirm that the importance of the inner iteration in the proposed inexact algorithm depends on the magnitude of the norm of the matrix $\mathbf{B}$ that appears in the objective function \eqref{target function in model: l0}. In the case of regression and classification problems, inner iterations are crucial due to the norm of matrix $\mathbf{B}$ exceeding 500, while in deblurring problems, the significance of inner iterations is limited as the norm of matrix $\mathbf{B}$ is not greater than 2.  { The numerical results consistently verified that increasing the value of $\alpha$ can speed up convergence of the proposed algorithm when $\alpha$ falls within the range $(0, 1)$, which is guaranteed by Theorem \ref{thm : convergence algo INFPPA}. While for $\alpha>1$, the impact of  $\alpha$ on the algorithm varies depending on a specific application.}

\section{Conclusion}\label{section: Conclusion}

The theoretical study and numerical experiments conducted in this paper have shown that the proposed inexact fixed-point proximity algorithms are effective for solving sparse regularization models involving the $\ell_0$-norm. 
The proposed inexact algorithm incorporates an inner iteration for the evaluation of the proximity operator of a convex function composed with a linear transformation. We have shown convergence of the proposed inexact algorithm to a local minimizer of the $\ell_0$ model under the hypothesis that the sequence of the numerical errors of the inner iteration is $\ell_1$ summable. The analysis is performed through analyzing the convergence of the related convex optimization problem, constrained on the support of the sparse solution, with applying the theory of non-expansive operators. 

We have applied the proposed algorithm to solve the problems of regression, classification, and image deblurring. 
Numerical experiments show that the proposed inexact fixed-point proximity algorithm can effectively solve the $\ell_0$ models, leading to convergent solutions, which exhibit either higher sparsity or accuracy than those obtained from solving the related $\ell_1$ models. In particular, numerical results for the regression problem demonstrate that the local minimizer obtained using our proposed algorithm for the $\ell_0$ model has higher sparsity while maintaining similar accuracy, compared to the global minimizer of the related $\ell_1$ model. The inexact algorithm for the $\ell_0$ models displays elevated sparsity and accuracy for the classification problem, and improved image restoration quality for the image deblurring problem, when compared to the related $\ell_1$ models. The numerical examples also demonstrate that the inner iteration plays a crucial role in regression and classification problems, while it has a limited impact on image deblurring problems, depending on the magnitude of the norm of the matrix $\mathbf{B}$ in the objective function.

The proposed algorithm significantly broadens the applicability of the exact fixed-point proximity algorithm.
The convergence analysis methodology employed in this study may be extendable to other sparse regularization problems involving the $\ell_0$ norm.

\bigskip

\noindent{\bf 
Data Availability Statement:} The computer codes that generate the numerical results presented in this paper can be found in the following website:
https://github.com/msyan2023/InexactFPPA-L0


%
%

\bibliographystyle{plain}      
\bibliography{references}   


\end{sloppypar}
\end{document}